%% file: Article.tex
\documentclass[a4paper,11pt]{article}
\usepackage[T1]{fontenc}
\usepackage[utf8x]{inputenc} 
\usepackage{lmodern}
\usepackage{amsmath}
\usepackage{amsthm}
\usepackage{amssymb}
\usepackage{authblk}
\usepackage{graphicx}
\usepackage{enumitem}
\usepackage{xcolor}
\usepackage{dsfont}
\newcommand\1{\mathds{1}}
\usepackage{mathrsfs}
\include{thm-config}

\usepackage{tikz}

\renewcommand{\dd}{\mathrm{d}}

\newcommand{\RomanNum}[1]{%
  \textup{\uppercase\expandafter{\romannumeral#1}}%
  }

\usepackage{a4wide}
\title{Regular Bohr-Sommerfeld rules for non-self-adjoint
  Berezin--Toeplitz operators and complex Lagrangian states}
\author{Alix
  Deleporte\thanks{Universit\'e Paris-Saclay, CNRS, Laboratoire de math\'ematiques d'Orsay, 91405, Orsay, France. {\em E-mail:} \texttt{alix.deleporte@universite-paris-saclay.fr}}, Yohann
  Le Floch\thanks{Institut de Recherche Math\'ematique avanc\'ee, UMR
    7501, Universit\'e de Strasbourg et CNRS, 7 rue Ren\'e Descartes,
    67000 Strasbourg, France. {\em E-mail:}
    \texttt{ylefloch@unistra.fr}}}

\usepackage{hyperref}
\hypersetup{pdfborder={0 0 0}, colorlinks=true, linkcolor={blue!80!black}, citecolor={red!85!black},urlcolor={green!50!black}}  
  
\begin{document}

\maketitle

\begin{abstract}
 We describe the eigenvalues and eigenvectors of real-analytic, non-self-adjoint
 Berezin--Toeplitz operators, up to exponentially small error, on complex one-dimensional compact
 manifolds, under the hypothesis of regularity of the energy
 levels. These results form a complex version of the Bohr-Sommerfeld quantization conditions; they hold under a hypothesis that the skew-adjoint
 part is small but can be of principal order with respect to the semiclassical parameter.

 To this end, we develop a calculus of Fourier Integral Operators and Lagrangian states
 associated with complex Lagrangians; these tools are of independent interest. 
\end{abstract}

\section{Introduction}
\label{sec:introduction}

In semiclassical analysis, the quantum state space (a Hilbert space)
and quantum observables (self-adjoint operators acting on this Hilbert space)
depend on a small parameter $\hbar>0$, and in the limit $\hbar \to 0$
one expects to recover footprints of classical (Hamiltonian)
mechanics. For instance, given an \emph{integrable}
classical observable $f$ (a function on $(M,\omega)$, the phase space,
which is a symplectic manifold), so that a quantum observable
$T_{\hbar}$ quantizing $f$ has discrete spectrum in a certain region,
one expects to describe the eigenvalues of $T_{\hbar}$, in the
semiclassical limit $\hbar \to 0$, thanks to classical, geometric
quantities associated with $f$. Such a description has been long known
under the name ``Bohr-Sommerfeld quantization conditions'' in physics,
and has been mathematically proven in various settings, in particular
in the case where $M = T^*\R^n$ and $T_{\hbar}$ is a self-adjoint
semiclassical pseudodifferential operator acting on $L_2(\R^n)$, see
the review \cite{vu_ngoc_systemes_2006}.

To some extent, these results have been extended to the case of
\emph{non-self-adjoint} operators \cite{hitrik_non-selfadjoint_2004,hitrik_boundary_2004,rouby_bohrsommerfeld_2017}, with the goal of studying smoothing
or decaying properties for partial differential equations with a
damping term \cite{alphonse_polar_2023}. A powerful tool consists in
weighted FBI transforms \cite{melin_determinants_2002,melin_bohr-sommerfeld_2003}.

FBI transforms microlocally conjugate pseudodifferential quantization into
\emph{Berezin--Toeplitz quantization}, acting on Kähler manifolds. The
goal of this article is to study Bohr-Sommerfeld rules for
one-dimensional (therefore integrable), non-self-adjoint systems near regular
trajectories, generalising both the self-adjoint Bohr-Sommerfeld rules
for Berezin--Toeplitz operators \cite{charles_quasimodes_2003} and the non-self-adjoint
Bohr-Sommerfeld rules for pseudodifferential operators.

To obtain a good semiclassical description of the eigenvalues in this case, we will assume that all the geometric data is real-analytic, which will allow us to complexify the geometry and construct complex analogues of the usual tools from the self-adjoint setting: Lagrangian (WKB) states, normal forms via Fourier Integral Operators, etc. In fact these constructions are somewhat delicate, and are of independent interest, so they will constitute the core of the paper.

\subsection{Prequantum line bundles and their holomorphic sections}
\label{sec:preq-line-bundl}

Let $(M,J,\omega)$ be a Kähler manifold. Locally in a holomorphic chart for $(M,J)$,
the Kähler data is represented by a real-valued function $\phi$ which
is plurisubharmonic: its
Levi matrix $[\partial_j\overline{\partial}_k\phi]_{j,k}$ is positive
definite. One has then $\omega=i\partial\overline{\partial}\phi$. In
particular, this
data does not change if one replaces $\phi$ with $\phi+{\rm Re}(f)$
where $f$ is holomorphic.

A \emph{prequantum line bundle} $L\to M$ is a holomorphic $\C$-bundle endowed with
a Hermitian metric $h$ whose curvature is $-i\omega$: this means that, when
$s$ is a local non-vanishing holomorphic section, $\log(\|s\|_h)$ is a
Kähler potential. Fixing such a section and denoting by $\phi$ the
Kähler potential, locally, the holomorphic sections of $L^{\otimes k}$
are of the form
\[
  \left\{u\in L^2(U,\C), e^{k\phi}u\text{ is holomorphic}\right\};
\]
the corresponding charts on $L$ are called \emph{Hermitian charts},
because the metric $h$ is mapped to the standard Hermitian metric on
$M\times \C$.

The existence of such a line bundle over the whole of $M$ is
conditioned to the fact that $\int_{\Sigma}\omega\in 2\pi\Z$ for every
closed surface $\Sigma\Subset M$. When this condition is satisfied we
will say that $M$ is \emph{quantizable}. The Hilbert space of holomorphic
sections $H^0(M,L^{\otimes k})$ is finite-dimensional when $M$ is
compact (the dimension grows with $k$) and we are interested in the
spectral theory of operators acting on $H^0(M,L^{\otimes k})$ which
\emph{quantize} a function $f:M\to \C$. A crucial object is the
self-adjoint projector $\Pi_k:L^2(M,L^{\otimes k})\to H^0(M,L^{\otimes
  k})$, named the Bergman projector. One way to quantize a function is to let
\begin{equation}\label{eq:contravariant}
  T_k(f)=\Pi_kf\Pi_k.
\end{equation}
This \emph{contravariant Berezin--Toeplitz quantization}
\cite{charles_berezin-toeplitz_2003} happens not to
be the most practical in real-analytic regularity, but it is
equivalent to another definition we shall introduce later.

A basic example of Berezin--Toeplitz quantization is $M=\C^n$, with
the global Kähler potential $\phi:z\mapsto |z|^2$. The quantum space
$H^0(\C,L^{\otimes k})$, called \emph{Bargmann-Fock space}, is the
image of $L^2(\R^n)$ under the FBI or wavelet transform, which
conjugates Berezin--Toeplitz quantization with pseudodifferential
operators. See \cite{folland_harmonic_1989} and Chapter 13 of
\cite{zworski_semiclassical_2012} for a general presentation of this
case. Here the inverse semiclassical parameter is $\hbar=k^{-1}$.

\subsection{Non-self-adjoint spectral asymptotics}
\label{sec:non-self-adjoint}

In a sense, Berezin--Toeplitz quantization allows to perform
semiclassical analysis, and in particular to generalise
pseudo-differential operators to other geometric settings, while
working directly in phase space. The goal of this article is to use
this paradigm to study non-self-adjoint problems in (complex) dimension 1.

The main difficulty in the spectral analysis of non-self-adjoint
operators is the presence of \emph{pseudospectral effects}: the set of
approximate solutions to the eigenvalue problem is much larger than
the spectrum. It was
shown for instance in \cite{borthwick_pseudospectra_2003} that if
${\rm dim}_\C(M)=1$, given
$p,q\in C^{\infty}(M,\R)$, for every $\lambda\in \C$ such that there
exists $x\in M$ satisfying $p(x)+iq(x)=\lambda$ and $\{p,q\}(x) < 0$,
there exists a normalised sequence $u_k\in H^0(M,L^{\otimes k})$ such
that $\|T_k(p+iq-\lambda)u_k\|_{L^2}=O(k^{-\infty})$, generalising a
previously known result for pseudodifferential operators, see
\cite{zworski_remark_2001}. In the pseudodifferential case, the
pseudospectrum begins to shrink if one enforces exponential accuracy
of quasimodes, that is $\|T_k(p+iq-\lambda)u_k\|_{L^2}=O(e^{-ck})$ for
some $c>0$ \cite{dencker_pseudospectra_2004}, which motivated the study of the spectrum of non-self-adjoint
operators with real-analytic symbols, where one can hope to describe
quantities up to exponentially small remainders, see Section
\ref{sec:spac-analyt-funct}.

The spectrum of non-self-adjoint
pseudodifferential operators in dimension $1$ was described in
\cite{rouby_bohrsommerfeld_2017} under the
following hypotheses: letting $\lambda$ be a regular energy level of
$p\in C^{\omega}(\R^2,\R)$ such that $\{p=\lambda\}$ is connected,
given $q\in C^{\omega}(\R^2,\R)$, there exists $\varepsilon_0>0$ such
that, for every $|\varepsilon|<\varepsilon_0$, the spectrum of
the Weyl quantization $\mathrm{Op}_{\hbar}(p+i\varepsilon q)$ near $\lambda$ is given by Bohr-Sommerfeld quantization
conditions, generalising the result known in the self-adjoint case
\cite{colin_de_verdiere_spectre_1980}. In particular, in this regime,
eigenvalues are regularly spaced (with a distance of order $\hbar$)
along complex curves. In the self-adjoint case, one has in fact a
description of the eigenvalues modulo exponentially small remainders
\cite{duraffour_analytic_2025}; a work in progress by the same author aims at extending this description to the non-self-adjoint case, see \cite{duraffour_these}. In the special case of Schrödinger
operators, under the same hypotheses, eigenvalues were described in
\cite{boussekkine_ptsymmetry_2016}; the case where $\{p=\lambda\}$ has
two symmetric connected components, and the spectra separate under the action of
$q$, was also treated in \cite{mecherout_pt-symmetry_2016}.

The goal of this article is to generalise these results, in the
Berezin--Toeplitz setting, by describing the eigenvalues and generalised
eigenfunctions of
$T_k(p+i\varepsilon q)$, for $\varepsilon$ small and $k$-independent, near regular energy levels
of $p$; our result holds independently on the number of connected components.

\begin{theorem}\label{thr:1}
  Let $(M,J,\omega)$ be a quantizable, compact, real-analytic Kähler manifold of complex dimension 1
  and let $L\to M$ be a prequantum line bundle over $M$. Let
  $p:\C\times M\to \C$ be a real-analytic map, holomorphic in the
  first variable, and such that $p_0:x\mapsto p(0,x)$ is
  real-valued. Let $\lambda_0\in \R$ be a regular value of $p_0$. Let
  $N\geq 1$ be the number of connected components of $\{p_0=\lambda_0\}$.

  There exist $c>0$, a neighbourhood $\mathcal{Z}$ of $0$ in $\C$, a
  neighbourhood $\mathcal{E}$ of $\lambda_0$ in $\C$, a family
  $(I_1,\cdots,I_N)$ of holomorphic \emph{classical analytic
    symbols} from $\mathcal{Z}\times \mathcal{E}$ to $\C$ (see Section
  \ref{sec:analyt-symb-class}), satisfying $\partial_\lambda I_n\in
  \C^*$ for every $1\leq n\leq N$, and a bijective map between the
  multiset ${\rm sp}(T_k(p_z))\cap \mathcal{E}$ (where eigenvalues are
  counted with geometric
  multiplicity) and the multiset
  \begin{equation}\label{eq:claim_BS}
    \{\lambda\in \mathcal{E},\exists 1\leq n\leq N, \exists j\in \N,
    I_n(z,\lambda;k^{-1})=2\pi jk^{-1}\}
  \end{equation}
  such that the difference between one element of the spectrum and the
  corresponding Bohr-Sommerfeld solution (element of the set \eqref{eq:claim_BS}) is $O(e^{-ck})$. In
  particular, the geometric multiplicity of eigenvalues is at most
  $N$.

  Given open neighbourhoods $U_1,\cdots,U_N$ of the connected
  components of $\{p_0=\lambda_0\}$, up to further reducing $c$, $\mathcal{Z}$ and
  $\mathcal{E}$, generalised eigenfunctions $u$ of $T_k(p_z)$ with
  eigenvalue $\lambda$ in $\mathcal{E}$, with norm $1$ in $H^0(M,L^{\otimes
    k})$, satisfy
  \[
    \|u\|_{L^2(M\setminus U_1\cup \cdots \cup U_N,L^{\otimes
        k})}=O(e^{-ck})
  \]
  and on each $U_n$, there exists a non-vanishing section $\Phi_n$ of
  $L$ and a holomorphic, real-analytic symbol $a_n$ such that
  \begin{equation}\label{eq:claim_WKB}
    \|u-\Phi_n^{\otimes k}a_n(\cdot;k^{-1})\|_{L^2(U_n,L^{\otimes k})}=O(e^{-ck}).
  \end{equation}
  In fact, one has also $\|u\|_{L^2(U_n,L^{\otimes k})}=O(e^{-ck})$
  unless $I_n(z,\lambda;k^{-1})\in 2\pi k^{-1}\Z+O(e^{-c'k})$ for some
  $c'>0$.
\end{theorem}

WKB-type functions as appearing in \eqref{eq:claim_WKB} are
exponentially accurate quasimodes for $T_k(p_z)$, but even in the
self-adjoint case, in the presence of \emph{resonances} (different
values of $n$ yielding the same Bohr-Sommerfeld conditions), actual
eigenfunctions will be non-trivial linear combinations of these
quasimodes. In the setting of this article, in addition to this
phenomenon, resonances may a priori generate
non-trivial Jordan blocks. For instance, on the sphere $M = S^2$ with the usual embedding $(x,y,z):S^2\to
\R^3$, the operator $T_k(x+iy)$ has only one simple
eigenvalue at $\lambda=0$, and a full-dimensional Jordan block.

The principal and subprincipal terms in the symbols $I_n$ appearing in the Bohr-Sommerfeld conditions respectively encode complex generalisations of the action and some subprincipal
contribution, which is related to the Maslov index in the case $M=\C$;
see Remark \ref{rem:subprincipal_Lagrangian} and Proposition
\ref{prop:BS_mod_k-2} for details and a comparison with formulas
previously appearing in the literature.

At the heart of the proof of Theorem \ref{thr:1} is a construction of
WKB quasimodes associated with regular trajectories, and an associated
``local resolvent estimate''. These results, found in Sections
\ref{sec:local} and
\ref{sec:semiglobal}, hold under (micro)local assumptions, and are therefore
valid in more general situations. 

\subsection{Complex semiclassical analysis}
\label{sec:methods-proof}

The spectral study of self-adjoint integrable systems
relies on a quantum normal form procedure
\cite{vu_ngoc_systemes_2006}. In the non-degenerate case, classical
Hamiltonians are treated by the construction of action-angle
coordinates, and to this symplectic change of variables corresponds a
unitary transform (a Fourier Integral Operator) which locally conjugates the
operator under study with a spectral function of
$ik^{-1}\frac{\partial}{\partial \theta}$
acting on $L^2(S^1)$, whose eigenvalues and eigenfunctions are
explicit.

Roughly speaking, this method can be generalised to the
non-self-adjoint setting, and this is exactly what we do,
but there are three serious difficulties. The first task is to understand holomorphic (complexified) versions of the usual
real-valued geometric statements of symplectic geometry, including the
action-angle theorem. This requires in particular to describe
``holomorphic extensions'' of the geometric data $(M,J,\omega)$ and
$(L,h)\to M$. The second difficulty is the study of a generalisation
of Fourier Integral Operators in this setting. They will be associated
to \emph{complex Lagrangians}, and therefore will not be unitary; to
the contrary, these operators make norms grow by as much as
$\exp(ak)$ where $a>0$ measures how far away the
Lagrangian lies from the real locus. The third challenge is that, in the non-self-adjoint
setting, the pseudospectral effect which we presented above means that
the existence of a quasimode is not sufficient to obtain the existence of an associated
eigenvalue. To overcome this last difficulty, we develop resolvent estimates. We prove in
particular that the Bohr-Sommerfeld condition \eqref{eq:claim_BS} is
necessarily satisfied by eigenvalues up to an exponentially small
error. Then, in order to describe the generalised eigenvectors, we express the spectral projectors as contour integrals of the resolvent.

In spirit, these techniques are already used in the literature
concerned with non-self-adjoint semiclassical spectral theory,
beginning in \cite{melin_determinants_2002,melin_bohr-sommerfeld_2003}
with the introduction of ``complex FBI transforms'' which are a
particular case of FIOs with complex phase. In the context of
pseudodifferential operators, however, manipulating these operators is
no easy task. In the setting of Berezin--Toeplitz quantization, all natural objects
(including complex Fourier Integral Operators) are described by WKB
kernel asymptotics without phase variables, and there are no caustics as
long as one does not deform too far away from the real locus. We hope that our
construction will be useful in other settings involving
non-self-adjoint operators, such as quantum dynamics and a spectral
study under other geometric conditions.

In a similar way, starting with the description of quasimodes, rather
than direct resolvent estimates, it is usual to construct
eigenfunctions by introducing a Grushin problem. Again, our approach
is morally equivalent but, in our situation, could be used more
directly. Microlocal resolvent estimates away from the spectrum can also be used
for other purposes, including the study of non-self-adjoint quantum dynamics.

In the spirit of \cite{le_floch_singular_2014,le_floch_singular_2014-1}, using the techniques developped in the present paper, one should be
able to describe the full spectrum of $T_k(p)$ in the Morse case. A description near elliptic points, in the pseudodifferential case, can be found in \cite{hitrik_boundary_2004,hitrik_overdamped_2025}; the Berezin-Toeplitz case is being handled in the thesis \cite{reguer_these} in progress. In future work, we will investigate the hyperbolic case.

\subsection{Acknowledgements}
\label{sec:acknowledgements}

This work was supported by the ANR-24-CE40-5905-01 ``STENTOR''
project. The authors thank Laurent Charles, Ophélie
Rouby and San V\~u Ng\d{o}c for useful discussions.

\section{Berezin--Toeplitz quantization in real-analytic regularity}
\label{sec:spac-analyt-funct}

\subsection{Analytic symbol classes}
\label{sec:analyt-symb-class}

In this article, we will only consider \emph{classical order 0} symbols, which have a formal expansion in integer powers of the
semiclassical parameter. The first such class of analytic symbols was
introduced by Boutet and Kr\'ee in \cite{boutet_de_monvel_pseudo-differential_1967}, and it
adapts well to Berezin--Toeplitz quantization.

\begin{defn}\label{def:Boutet-Kree}
  Let $K$ be a compact set of $\R^d$ and let $T>0$. Given a (classical
  order 0)
  formal symbol $p=(p_\ell)_{\ell\in \N}$, define
  \[
    p_{\ell,\alpha}^{\beta}:z \mapsto
    \partial^{\alpha}\overline{\partial}^{\beta}p_\ell(z) \qquad \qquad
    \alpha,\beta\in \N^d
  \]
  and then
  \[
    \|p\|_{BK(T,K)}=\sum_{\alpha,\beta,\ell}\frac{2(2d)^{-\ell}\ell!}{(\ell+|\alpha|!)(\ell+|\beta|)!}\sup_{K}|p_{\ell,\alpha}^{\beta}|T^{2\ell+|\alpha+\beta|}.
  \]
\end{defn}
Among the alternative definitions, we will also use the following one
from \cite{deleporte_toeplitz_2018}.

\begin{defn}\label{def:formal_semicl_ampl}
  Let $U$ be an open set of $\R^d$. We define the space $S^{r,R}_m(U)$ as the space
  of sequences $(a_\ell)_{\ell\in \N}$ of functions on $U$ such that
  \[
    \exists C, \forall j,\ell\in \N,\forall x\in U, \quad \sum_{|\alpha|=j}|\partial^{\alpha}a_\ell(x)|\leq
    C\frac{r^jR^\ell(j+\ell)!}{(1+j+\ell)^{m}}.\]
  The best such constant $C$ is written $\|a\|_{S^{r,R}_m(U)}$.
\end{defn}

The union over $T>0$ of the spaces $BK(T)$ coincides with the union over
$r>0,R>0,m\in \R$ of the spaces $S^{r,R}_m$; we call such elements \emph{formal
  analytic amplitudes}. Such amplitudes can be \emph{summed} via a
lower term summation procedure: we define for $c > 0$ small enough
\[
  a(x;\hbar)=\sum_{\ell=0}^{\lfloor
    c\hbar^{-1}\rfloor}\hbar^{\ell}a_{\ell}.
\]
This does not depend on $c$ up to an exponentially small error
$O(e^{-c'\hbar^{-1}})$, see \cite{deleporte_toeplitz_2018},
Proposition 3.6. This notion is compatible with stationary phase in
real-analytic geometry in the following sense: the result of a stationary
phase integral with a real-analytic phase function having positive
imaginary part near the boundary of the integration domain, and an
analytic symbol as amplitude, is another analytic symbol, see \cite{sjostrand_singularites_1982}, Chapter 2.

\subsection{Asymptotics of the Bergman kernel and covariant Berezin--Toeplitz
  operators}
\label{sec:bergm-kern-covar}

The Bergman kernel on a real-analytic, quantizable Kähler manifold can
be understood using formal analytic amplitudes, and the latter also allow us to
introduce \emph{covariant
Berezin--Toeplitz operators}.

\begin{prop}\label{prop:Bergman_asymp}
  Let $(M,J,\omega)$ be a compact quantizable Kähler manifold and let $(L,h)$
  be a prequantum line bundle over $M$. Suppose that $\omega$ is
  real-analytic in $J$-holomorphic charts. Then, as $k\to +\infty$,
  the Bergman kernel on $H_0(M,L^{\otimes k})$ is exponentially small
  away from the diagonal. Near the diagonal, in a Hermitian chart with
  (real-analytic) Kähler potential $\phi$, it is of the form
  \begin{equation}\label{eq:Bergman-analytic}
    (x,y)\mapsto
    k^de^{\frac{k}{2}(-\phi(x)+2\psi(x,y)-\phi(y))}s(x,y;k^{-1})+O(e^{-ck})
  \end{equation}
  for some $c>0$, some classical analytic amplitude $s$, and where $\psi$ is
  the polarisation of $\phi$:
  \[
    \psi(x,x)=\phi(x)\qquad \qquad \overline{\partial}_x\psi=0\qquad
    \qquad \partial_y\psi=0.
  \]
  The amplitude $s$ is also $x$-holomorphic and $y$-antiholomorphic.

 Define the \emph{covariant Berezin-Toeplitz operator} associated with $a$ as the operator whose kernel is exponentially small away
  from the diagonal and, near the diagonal, of the form
  \begin{equation}\label{eq:covariant_Toep}
    T_k^{\rm cov}(a)(x,y)\mapsto
    k^de^{\frac{k}{2}(-\phi(x)+2\psi(x,y)-\phi(y))}s(x,y;k^{-1})a(x,y;k^{-1})+O(e^{-ck});
  \end{equation}
here $a$ is a classical analytic amplitude which is $x$-holomorphic and $y$-antiholomorphic.
  
  Then the space of covariant Berezin-Toeplitz operators forms an
  algebra for composition. Its invertible elements are exactly those
  for which the principal symbol never vanishes. In particular, the
  composition law $\star_{\rm cov}$ of classical analytic amplitudes satisfies
  \begin{equation}\label{eq:stability_analytic}
   \forall m\geq m_0,\forall r\geq r_0, \forall R\geq R_0(m,r), \quad \|a\star_{\rm cov} b\|_{S^{r,R}_m}\leq C(m,r,R)\|a\|_{S^{r,R}_m}\|b\|_{S^{\frac r2,
        \frac R2}_m}.
  \end{equation}
\end{prop}

Berezin--Toeplitz operators were introduced in
\cite{berezin_general_1975-1}, a microlocal analysis of related
operators was initiated in \cite{boutet_de_monvel_spectral_1981}, and
in the smooth case they are now well-studied \cite{bordemann_toeplitz_1994,guillemin_star_1995,charles_berezin-toeplitz_2003,ma_toeplitz_2008-1}. In the general smooth Berezin-Toeplitz setting, the relations between covariant and contravariant Berezin-Toeplitz operators was studied in \cite{charles_berezin-toeplitz_2003}. In the analytic case, these two definitions are also related: if $f$ is an analytic symbol, then $\Pi_k f \Pi_k$ is
of the form \eqref{eq:covariant_Toep} for some $a$ obtained from $f$ (see
\cite{deleporte_toeplitz_2018}, Proposition 4.11); the converse is
also true \cite{bonthonneau_microlocal_2024}.

The precise statement of Proposition \ref{prop:Bergman_asymp} is
contained in \cite{deleporte_toeplitz_2018}, (see Theorem A, Theorem B, and Remark
4.10). Statements of a similar nature appear in
\cite{rouby_analytic_2018}, and later on the proof of
\eqref{eq:Bergman-analytic} was greatly simplified
\cite{charles_analytic_2021,deleporte_direct_2024} but we will need
the precise statement \eqref{eq:stability_analytic}.

An example (albeit non-compact) for Berezin--Toeplitz quantization is
the complex line $\C$. In a convenient Hermitian chart, the symplectic
form is $\dd x\wedge \dd \xi$ where the complex variable is
$z=\frac{x+i\xi}{\sqrt{2}}$; an associated Kähler potential is
$(x,\xi)\mapsto \frac{\xi^2}{2}={\rm Im(z)}^2$. Consequently, the
Hilbert space under study is the Bargmann space
\[
  \mathcal{B}_k=\left\{u\in L^2(\C,\C),e^{\frac{|\xi|^2}{2}}u\text{ is holomorphic}\right\}
\]
and the Bergman kernel
is
\begin{equation}\label{eq:Bargmann}
  \Pi_k(z,z')=\frac{k}{2\pi}\exp\left[k\left(-{\rm Im}(z)^2-{\rm
        Im}(z')^2+2\left(\frac{\overline{z'}-z}{2}\right)^2\right)\right].
\end{equation}
In this case, covariant Toeplitz quantization coincides with ``Wick
ordering'' of symbols \cite{folland_harmonic_1989}; one has
\[
  f\star_{\rm cov}g=\sum_{\ell\in
    \N}\frac{(-k)^{\ell}}{\ell!}\overline{\partial}^\ell
  f\partial^\ell g.
\]
Substituting $z$ for $x$ and $\overline{z}$ for $\xi$, this
star-product coincides with that of left-quantization on $\R^2$. In
particular, the main result of
\cite{boutet_de_monvel_pseudo-differential_1967} applies in this case.

\begin{prop}\label{prop:BK}
  In the case $(M,J,\omega)=(\C,J_{\rm st},\omega_{\rm st})$, for every
$T>0$, $(BK(T),\star_{\rm cov})$ is a Banach algebra.
\end{prop}

Another useful local model is $M=S^1_{\theta}\times \R_{\xi}$; we take the convention
that $S^1=\R/2\pi\Z$, $J\frac{\partial}{\partial
  \theta}=\frac{\partial}{\partial \xi}$, and $\omega= i \dd z\wedge \dd
\overline{z}$ where $z=\frac{\theta+i\xi}{\sqrt{2}}$. We consider
$(\theta,\xi)\mapsto \frac{\xi^2}2$ as a Kähler potential as
before. The Bergman kernel is given by a sum of \eqref{eq:Bargmann} over
periods, leading to a theta function; because of the off-diagonal
decay of \eqref{eq:Bargmann}, however, the Bergman kernel is
exponentially close to \eqref{eq:Bargmann}. In particular, the formal
covariant star-product coincides with the Wick product, so that
Proposition \ref{prop:BK} holds in this case as well. We will denote by
$\mathcal{B}_k^{S^1}$ the space of global $L^2$ holomorphic sections of
$L^{\otimes k}$ over $T^*S^1$.

\section{Complex Lagrangian states and Fourier Integral Operators}

\label{sec:compl-lagr-stat}

\subsection{Holomorphic extensions}\label{sec:holom-extens}
The topic of this subsection is to review, in a more geometric way,
the constructions in \cite{deleporte_real-analytic_2022}. The base principle is,
given a Kähler manifold $(M,\omega,J)$ and a prequantum line bundle
$L\to M$, to construct natural notions of holomorphic extensions for
$\omega$ and the connection $\nabla$. In spirit, these constructions
are already present in the works of Sjöstrand starting from \cite{sjostrand_singularites_1982}; the holomorphic
extension $\Omega$ of $\omega$ is such that both its real and
imaginary parts are symplectic forms, and the real locus will be
symplectic for the real part, and Lagrangian for the imaginary part. 

We begin with general notions of holomorphic extensions of
differential forms.

\begin{lem}\label{prop:extension_forms}
Let $N$ be a complex manifold, let $E \to N$ be a holomorphic vector bundle. Let $P$ be a compact, maximally totally real, real-analytic submanifold of $N$. Let $p \in \N$, and let $\alpha \in \Omega^p(P,E_{|P})$ be a real-analytic differential form. There exist a neighbourhood $V$ of $P$ in $N$ and a unique $\tilde{\alpha} \in \Omega^{(p,0)}(V,E)$ such that
  \[ \left\{\begin{matrix}\overline{\partial} \tilde{\alpha} = 0\\
        \alpha = \iota^*\tilde{\alpha}
      \end{matrix}
    \right.\]
with $\iota: P \hookrightarrow N$ the inclusion. We call $\tilde{\alpha}$ the \emph{holomorphic extension} of $\alpha$.
\end{lem}

\begin{proof}
Since the fiber bundle $\Omega^{(p,0)}(N,E)$ is holomorphic, any real-analytic section over $P$ of $\Omega^{(p,0)}(N,E)$ admits a unique holomorphic extension to a neighbourhood of $P$ in $N$ (this is standard and done by extending the coefficients of $\alpha$ in charts). Hence it remains to interpretate $\alpha$ as such as section. This can be done through the isomorphism
\[ (T^*_P N)^{(1,0)} \to T^* P \otimes \C, \quad \gamma \mapsto \iota^*\gamma  \]
which can be passed to tensor products to obtain an isomorphism between $\Omega_P^{(p,0)}(N,E)$ and $\Omega^p(P,E) \otimes \C$.
\end{proof}

\begin{corr}\label{prop:extension_forms_untwisted}
Let $N$ be a complex manifold, and let $P$ be a compact, maximally
totally real, real-analytic submanifold of $N$. Let $p \in \N$, and
let $\alpha \in \Omega^p(P)$ be a real-analytic differential
form. There exist a neighbourhood $V$ of $P$ in $N$ and a unique
$\tilde{\alpha} \in \Omega^{(p,0)}(V)$ such that
  \[ \left\{\begin{matrix}\overline{\partial} \tilde{\alpha} = 0\\
        \alpha = \iota^*\tilde{\alpha}
      \end{matrix}
    \right.\]
with $\iota: P \hookrightarrow N$ the inclusion.
\end{corr}

\begin{lem}\label{prop:extension_d}
With the same notation as in the previous corollary, the holomorphic extensions of $\alpha$ and $\dd \alpha$ satisfy
\[ \widetilde{\dd \alpha} = \partial \tilde{\alpha}. \]
\end{lem}

\begin{proof}
By uniqueness of the holomorphic extension, it suffices to check that both terms in the equality agree on $P$. Obviously $\iota^*(\widetilde{\dd \alpha}) = \dd \alpha$ by definition, and
\[ \iota^*(\partial \tilde{\alpha}) = \iota^*(\dd \tilde{\alpha}) = \dd (\iota^* \tilde{\alpha}) = \dd \alpha. \]
\end{proof}

\begin{rem}\label{rem:extension_objects}
  Most of the natural notions about differential forms are compatible
  with the holomorphic extensions of Lemma \ref{prop:extension_forms},
  such as the wedge operator and tensor products.
\end{rem}

\begin{lem}\label{prop:extension_nabla}
With the same notation as above, let $(E,\nabla) \to P$ be a complex
vector bundle with a real-analytic connection. There exist a
neighbourhood $V$ of $P$ in $N$ and a unique holomorphic vector bundle with holomorphic connection $(\tilde{E},\widetilde{\nabla}) \to V$ such that 
\[ \iota^*(\tilde{E},\widetilde{\nabla}) = (E,\nabla). \]
Moreover, for any real-analytic section $s$ of $E \to P$, 
\[ \widetilde{\nabla s} = \widetilde{\nabla} \tilde{s}. \]
Furthermore, the curvature form of $\widetilde{\nabla}$ is the
holomorphic extension of the curvature form of $\nabla$. 
\end{lem}

\begin{proof}
First we define $\tilde{E}$ by working with an atlas $(U_i)_{1 \leq i
  \leq m}$ and extending holomorphically the transition functions of
$E$, which are real-analytic. To define $\widetilde{\nabla}$, we consider
the local connection 1-forms $A_1, \cdots, A_m$ for $\nabla$
associated with local frames $\mathcal{B}_1, \cdots, \mathcal{B}_m$ of
$E$, which are real-analytic sections of $\Omega^1(P,E)$. We extend
them holomorphically using Lemma \ref{prop:extension_forms}; let $\widetilde{A_1}, \cdots,
\widetilde{A_m}$ be these extensions. Now we define $\tilde{\nabla}$ to be
given by $\partial + \widetilde{A_i}$ in the frame $\mathcal{B}_i$. To
show that this defines a global object, consider a real-analytic
section $s$ of $E \to U_i$ and its holomorphic extension $\tilde{s}$,
which is a section of $\tilde{E} \to V$. By
construction $\iota^*(\widetilde{\nabla} \tilde{s}) =  \nabla s$, hence by
uniqueness $\widetilde{\nabla} \tilde{s} = \widetilde{\nabla s}$, and in
particular $\widetilde{\nabla}\tilde{s}$ does not depend on the chart.

It remains to prove the relationship between the curvatures of
$\widetilde{\nabla}$ and $\nabla$. This can be seen either from the
local connection forms, using the fact that
\[
  \widetilde{{\rm curv}(\nabla)}=\widetilde{\dd A_i+A_i\wedge A_i}=\partial
  \widetilde{A_i}+\widetilde{A_i}\wedge \widetilde{A_i}={\rm curv}(\widetilde{\nabla})
\]
or from the relationship above between the connections and holomorphic
extensions: given holomorphic vector fields $\widetilde{X}$ and
$\widetilde{Y}$, whose restriction to $M$ are denoted respectively $X$
and $Y$, one has that
\[
  {\rm
    curv}(\widetilde{\nabla})(\widetilde{X},\widetilde{Y}):=\widetilde{\nabla}_{\widetilde{X}}\widetilde{\nabla}_{\widetilde{Y}}-\widetilde{\nabla}_{\widetilde{Y}}\widetilde{\nabla}_{\widetilde{X}}-\widetilde{\nabla}_{[\widetilde{X},\widetilde{Y}]}\]
is the holomorphic extension of
\[
  \nabla_X\nabla_Y-\nabla_Y\nabla_X-\nabla_{[X,Y]}
  ={\rm curv}(\nabla)(X,Y).
  \]
\end{proof}

A crucial application of the previous general principles concerns the
holomorphic extension of a prequantum line bundle over a real-analytic
Kähler manifold.

\begin{corr}\label{prop:extension_data_Kähler}
  Let $(M,\omega,J)$ be a real-analytic, compact, quantizable Kähler manifold. Let $(L,\nabla)\to M$ be a prequantum line
  bundle. The inclusion $\iota:x\mapsto (x,x)$ from $M$ to $M\times M$
  forms a maximally totally real submanifold of $(M\times \overline{M},I):=(M\times M,(J,-J))$.

  There exist a neighbourhood $\widetilde{M}$ of the diagonal in
  $M\times \overline{M}$ and a holomorphic complex line bundle
  $(\widetilde{L},\widetilde{\nabla})\to \widetilde{M}$ such that
  \begin{itemize}
  \item $i\,{\rm curv}(\widetilde{\nabla})$ is the holomorphic extension
    of $\iota_*\omega$, in the sense of Lemma \ref{prop:extension_forms};
  \item the restriction of $\widetilde{L}$ to the diagonal of $M\times
    \overline{M}$ is the image of $L$ by $\iota$.
  \end{itemize}
\end{corr}
In practice, from a chart in the Kähler manifold $(M,\omega,J)$, one
can recover the data of Corollary \ref{prop:extension_data_Kähler}
as follows. In a small holomorphic chart on $M$, the Kähler data is
given by a Kähler potential $\psi$ (a plurisubharmonic function) as follows:
\[
  \omega=i\sum \frac{\partial^2\psi}{\partial z_j \partial \overline{z_k}}\dd z_j\wedge \overline{\dd z_k}.
\]
Writing $G_{j,k}=\frac{\partial^2\psi}{\partial z_j \partial
  \overline{z_k}}$, the $(G_{j,k})_{j,k}$ are real-analytic functions
on the chart.

On the manifold $M\times \overline{M}$, we introduce corresponding
coordinates $(z_j,\overline{\omega_j})$. The
real-analytic functions $G_{j,k}(z,\overline{z})$ on $M$ give rise to holomorphic
functions $\widetilde{G_{j,k}}(z,\overline{w})$, well-defined in a
neighbourhood of the diagonal $\{\overline{w}=\overline{z}\}$. Thus, the following holomorphic $(2,0)$-form on a neighbourhood of the
diagonal extends $\omega$ in the sense of Lemma
\ref{prop:extension_forms}:
\[
  \Omega=i\sum \widetilde{G_{jk}}\dd z_j\wedge \overline{\dd w_k}.
\]
By Lemma \ref{prop:extension_nabla}, $-i\Omega$ is the curvature of
$\widetilde{\nabla}$.

The fact that the original connection $\nabla$ is unitary is reflected
in a similar identity for $\widetilde{\nabla}$, which involves the
``holomorphic extension'' of the Hermitian metric on $L$. This metric
is extended as a sesquilinear form for the compatibility condition to
stay true.
\begin{prop}\label{prop:extend_h}
  Let $(M,\omega,J)$ be a real-analytic, compact, quantizable Kähler manifold. Let $(L,h,\nabla)\to M$ be a prequantum line
  bundle (in
  particular, $h$ is a sesquilinear form on $L$, i.e. a
  linear form on $L\otimes \overline{L}$, 
  and $\nabla$ is unitary for $h$).

  There exists a unique section $\tilde{h}$ of $\widetilde{L}\otimes
  \widetilde{\overline{L}}$ which holomorphically extends $h$. This
  section does not vanish on a neighbourhood of the diagonal, and is
  compatible with $\widetilde{\nabla}$, in the sense that for every
  $I$-holomorphic sections $s,t$ of $\widetilde{L}$ and
  $\widetilde{\overline{L}}$, one has
  \[
    \partial
    \widetilde{h}(s\otimes t)=\widetilde{h}(\widetilde{\nabla}s\otimes
    t)+\widetilde{h}(s\otimes \widetilde{\nabla}t).
  \]
\end{prop}
\begin{proof}
  The existence and uniqueness of $\widetilde{h}$ comes from the usual
  properties of holomorphic extensions of forms; here $h$ is
  real-analytic and non-vanishing.

  To prove compatibility, note that the identity above holds on the
  real locus $\iota(M)$, by Corollary \ref{prop:extension_d}, Proposition
  \ref{prop:extension_nabla}, and the fact that $\nabla$ is unitary
  for $h$. Since all objects are holomorphic, it
  holds on the whole of $\widetilde{M}$.
\end{proof}
Proposition \ref{prop:extend_h} allows us to identify elements of
$\widetilde{L}\otimes \widetilde{\overline{L}}$ with complex numbers,
by silent application of $\widetilde{h}$. To avoid cumbersome
notation, given $v\in L$ and $\overline{w}\in \overline{L}$ over the
same base point, we will denote by $v\cdot \overline{w}$ the
associated complex number. Beware that $\widetilde{h}$ is not a
Hermitian form and therefore $v\mapsto v\cdot \overline{v}$ is not necessarily a
real positive number. 

We will use another complex structure on $M\times \overline{M}$, which
``extends'' the structure $J$ on $M$: it is the structure
$\widetilde{J}=(J,J)$. Both $I$ and $\widetilde{J}$ will come to play in Section
\ref{sec:lagrangian-states-1}. Let us already prove that the notion of
holomorphic extension behaves naturally with respect to these
structures.
\begin{prop}\label{prop:ext_totally_real_mflds}
  Let $(M,J,\omega)$ be a real-analytic Kähler manifold and let $N$ be
  a real-analytic submanifold of $M$. Suppose that $N$ is
  \emph{totally real}:
  \[
    TN\cap JTN=\{0\}.
  \]
  Consider the $I$-holomorphic extension $\widetilde{N}$ of $N$: this is
  the $I$-holomorphic submanifold of $\widetilde{M}$ which is locally given by the
  zero set of $\widetilde{f}$ where $f$ is a (real-analytic) defining
  function for $N$.

  Then $\widetilde{N}$ is $\widetilde{J}$-totally real in a neighbourhood of
  the diagonal in $\widetilde{M}$.
\end{prop}
\begin{proof}
  Notice first that the condition $TN\cap JTN=\{0\}$ forces the
  dimension of $N$ to be at most half of the dimension of $M$. It is
  then an open condition: if a linear space $F$ satisfies $F\cap
  JF=\{0\}$ then for every $F'$ close to $F$ one also has $F'\cap
  JF'=\{0\}$.

  Let $x\in N$. Writing $T_xN=\ker(\dd_xf)$ and using Proposition
  \ref{prop:extension_d}, we find that
  \[
    T_{(x,x)}\widetilde{N}=\{(v+Jw,v-Jw);v,w\in T_xN\}.
  \]
  From this description, if $N$ is $J$-totally real, then
  $T_{(x,x)}\widetilde{N}$ is $(J,J)$-totally real. Now, since being
  totally real is an open condition, it follows that for $(x,y)$ close
  to the diagonal, $T_{(x,y)}\widetilde{N}$ is still $\widetilde{J}$-totally
  real. This concludes the proof.
\end{proof}

\begin{rem}As a holomorphic $(2,0)$-form, $\Omega$ is closed and
  satisfies a non-degeneracy condition: for every nonvanishing holomorphic vector
  field $X$, the one-form $\iota_X\Omega$ does not vanish. Such a form
  is called a \emph{holomorphic symplectic form}; in particular both
  the real part and the imaginary part of $\Omega$ are symplectic
  forms (in the usual sense of the term) on $\widetilde{M}$.

  Holomorphic symplectic forms are a natural object of Hyperkähler
  geometry. More precisely, a Hyperkähler manifold is a Riemannian
  manifold $(N,G)$ endowed with three complex structures $(I,J,K)$
  such that $IJ=K$ and such that $(N,G,I)$, $(N,G,J)$ and $(N,G,K)$
  are Kähler manifolds. Given such a manifold, the complex-valued
  $2$-form $\omega_J+i\omega_K$ happens to be, relatively to the
  structure $I$, a holomorphic symplectic form. Reciprocally, on a
  \emph{compact} (boundaryless) complex manifold $(M,I)$ endowed with
  a holomorphic symplectic form $\Omega$, there exist compatible
  hyperkähler structures, and given a cohomology class in $H^2(M)$ there
  exists a unique hyperkähler structure such that $\omega_I$ belongs
  to this class \cite{beauville_varietes_1983,calabi_metriques_1979}.

  In our situation, it is known that there exists, in a neighbourhood
  of $M$ in $\widetilde{M}$, a Hyperkähler structure
  $(\widetilde{M},I,J',K',g')$ compatible with the data on $M$: $I$ is
  the natural complex structure on $\widetilde{M}$, $M$ is
  $J'$-totally real, and $(J',g')$ coincides with $(J,g)$ on $M$ \cite{feix_hyperkahler_2001,kaledin_canonical_2001,abasheva_feix-kaledin_2022}\footnote{The authors thank Hans-Joachim Heim for provinding them with these references.}. This
  mimics the fact that real-analytic compact Riemannian manifolds
  admit, on their holomorphic extension, a natural Kähler structure \cite{guillemin_grauert_1991,guillemin_grauert_1992}. It
  is important to note, however, that $J'\neq \widetilde{J}$; to the
  contrary, it is a general feature of Hyperkähler geometry that even
  locally there are no non-constant functions that are $I$-holomorphic
  and $J'$-holomorphic at the same time. Since we wish to consider
  $I$-holomorphic extensions of $J$-holomorphic objects, it is unclear
  to us whether the Hyperkähler structure above is useful.

\end{rem}

\subsection{Lagrangian states}
\label{sec:lagrangian-states-1}

WKB-type elements of $H^0(M,L^{\otimes k})$ are very useful in all
aspects of semiclassical analysis, and even more so in quantum integrable
systems, since they approximate joint eigenvectors in the semiclassical limit.

Following \cite{charles_quasimodes_2003,deleporte_real-analytic_2022}
we define and study \emph{Lagrangian states} on Kähler manifolds as
WKB-type states with analytic phases and symbols. Such states naturally correspond, in a precise semiclassical sense, to Lagrangian submanifolds; here these submanifolds will be complex. When the Kähler manifold is 
of the form $\mathcal{M}  = M \times \overline{N}$, these Lagrangian states will be
kernels of Fourier Integral operators.

We begin with the sections
associated with ``reference'' Lagrangians, which are real-analytic and
real. We first recall the associated geometric requirement on the
Lagrangians.

\begin{defn}\label{def:Bohr-Sommerfeld}
  Let $\Lambda\subset M$ be a real-analytic Lagrangian. In particular
  $(L,\nabla)\to \Lambda$ is flat. The
  \emph{Bohr-Sommerfeld class} of $\Lambda$ is the holonomy of $(L,\nabla)\to
  \Lambda$, that is, the group morphism $\pi_1(\Lambda)\to \C^*$
  obtained by parallel transport on $L$ along loops in $\Lambda$ with
  respect to $\nabla$.

  More generally, let $\Lambda\subset \widetilde{M}$ be a
  holomorphic Lagrangian. In particular $(\widetilde{L},\widetilde{\nabla})\to
  \Lambda$ is flat. The Bohr-Sommerfeld class of $\Lambda$ is the
  holonomy of $(\widetilde{L},\widetilde{\nabla})\to \Lambda$.
\end{defn}

\begin{prop}\label{prop:real_Lagr_state}Let $M$ be a real-analytic, quantizable
  Kähler manifold with a prequantum line bundle $(L,h)$.
  Let $\Lambda \subset M$ be a real-analytic
  open Lagrangian, with real-analytic boundary (possibly empty) and trivial Bohr-Sommerfeld class. Over a small neighbourhood $U$ of $\Lambda$,
  there exists a holomorphic section $\Phi_{\Lambda}$ of $L$ such that
  \[
    1-|\Phi_{\Lambda}|_{h}=
    \dist(\cdot,\Lambda)^2+O(\dist(\cdot,\Lambda)^3).
  \]
\end{prop}
\begin{proof}
  Fix arbitrarily the value of $\Phi_{\Lambda}$ at a point $x_0$ of
  $\Lambda$ such that
  its norm is $1$. Then, for $x\in \Lambda$ define
  $\Phi_{\Lambda}(x)$ as the parallel transport of
  $\Phi_{\Lambda}(x_0)$ along a path in $\Lambda$ joining $x_0$ and $x$. The value of
  $\Phi_{\Lambda}(x)$ does not
  depend on the path chosen since $(L|_{\Lambda},h)$ is
  flat with vanishing holonomy. Moreover, since parallel transport
  preserves the Hermitian metric, one has $|\Phi_{\Lambda}|_{h}=1$
  on $\Lambda$.

  $\Lambda$ is a totally real submanifold of
  $M$. Therefore, arbitrary real-analytic sections of $L$ over this
  set admit a unique holomorphic extension on a small neighbourhood. This defines $\Phi_{\Lambda}$ everywhere. Now 
  $(L,h)$ is positively curved with curvature equal to the Kähler
  form, so that $\log |\Phi_{\Lambda}|_h$
  is plurisubharmonic and we can compute its Hessian at every point of
  $\Lambda$, see also \cite{charles_symbolic_2006}, Lemma 4.3. 
  This concludes the proof.
\end{proof}

\begin{defn}\label{def:Lagr-state}
  Let $(M,J,\omega)$ be a quantizable, real-analytic Kähler
  manifold. Let $V$ be an open set in $M$. A \emph{(complex) Lagrangian
    state} on $V$ is a sequence of elements of $H^0(M,L^{\otimes k})$
  of the form
  \[   I^{\Phi}_{k}(a)=\Pi_k(\1_W\Phi^{\otimes k}a)\in
    H^0(M,L^{\otimes k}) \]
  where
  \begin{itemize}
  \item $V\Subset W$ (meaning that $V$ is relatively compact in $W$);
  \item $a$ is an analytic symbol on $W$,
    which is holomorphic;
  \item $\Phi$ is a holomorphic section of $L$ over $W$, which belongs
    to a small neighbourhood (in the topology of holomorphic sections) of the set of sections of the form
    $\Phi_{\Lambda}$ as in Proposition \ref{prop:real_Lagr_state}, where $\Lambda$ is a real-analytic Lagrangian of
    $U\Supset W$.
  \end{itemize}
  If $\Phi$ is the form $\Phi_{\Lambda}$ as in Proposition \ref{prop:real_Lagr_state}, then
  $I_{k}^{\Phi}(a)$ is called a \emph{real Lagrangian state}. If $M$ is of the form $N_1\times \overline{N_2}$, then
  $I_k^{\Phi}(a)$ is called an \emph{analytic Fourier Integral Operator}.
\end{defn}
\begin{rem}~\begin{enumerate}
    \item
  The
  order of the symbol $a$ is not necessarily $\dim_{\C}(M)/4$; in
  particular, real Lagrangian states are not necessarily $L^2$-normalised. Indeed we
  will use these states in a variety of situations, including
  eigenvectors of Berezin--Toeplitz operators but also integral
  kernels of natural operators, such as the Bergman kernel or more
  general Fourier Integral Operators.
\item By Proposition \ref{prop:real_Lagr_state}, for every Lagrangian
  $\Lambda$ and every $r>0,m$ there exist $c_0>0$ and $C_0$ such
  that, for every holomorphic section $\Phi$ of $L$ over $V$, one has
  \begin{equation}\label{eq:bound_norm_Phi}
      |\Phi|_h<\exp(-c_0\dist(\cdot,\Lambda)^2+C_0\||\Phi|_h-1\|_{H^r_m(\Lambda\cap
        V)}).
    \end{equation}
    In particular, the notion of ``closeness to a section of the form
    $\Phi_{\Lambda}$'' used in Definition \ref{def:Lagr-state} then in
    the rest of this article, means in practice that $|\Phi|_h$ is
    close to $1$, in some real-analytic topology, on some real-analytic
    manifold $\Lambda$.
  \item What we call ``real Lagrangian states'' coincide with the
    usual notion of Lagrangian states as used in the literature, starting with \cite{charles_quasimodes_2003}. Our complex Lagrangian
    states will be associated with Lagrangians in $\widetilde{M}$
    (that is, \emph{complex} Lagrangians), see Proposition
    \ref{prop:Lagrangians}; this justifies the choice of
    terminology.
  \end{enumerate}
\end{rem}

The notation for a Lagrangian state does not make the neighbourhood
$W$ of $V$ apparent. The reason for this is the next proposition,
according to which Definition \ref{def:Lagr-state} does not
depend too much on the choice of $W$.
\begin{prop}\label{prop:good_def_Lagr_state}
Near $V$, Definition \ref{def:Lagr-state} does not depend on $W$ modulo exponentially small
errors. Indeed, if $W'\subset W$ is a smaller open neighbourhood of
$ V$, and if $\Lambda$ is a real Lagrangian, if $\Phi$ is close enough
to $\Phi_{\Lambda}$, then
\[
  \|\1_V\Pi_k(\1_{W\setminus W'}\Phi^{\otimes
    k}a)\|_{L^2}=O(e^{-c'k}).
\]
In fact, in the vicinity of $V$, one has \[I^{\Phi}_{V,k}(a)=\Phi^{\otimes k}a+O(e^{-c'k}).\]
\end{prop}
\begin{proof}
  By \eqref{eq:bound_norm_Phi}, $W\setminus W'$ is the union of a region at positive distance from
$\Lambda$ and a region at positive distance from $V$. Moreover,
$\Pi_k$ is exponentially small away from the diagonal. Thus $\1_V\Pi_k(\1_{W\setminus W'}\Phi^{\otimes
    k}a)$ is the sum of two exponentially small contributions.
\end{proof}
The only place where the Lagrangian states above are
ill-defined is a neighbourhood of the ``reference'' real Lagrangian $\Lambda$ from which we remove
a neighbourhood of $V$. Everywhere else, Lagrangian states are either
of WKB form or are exponentially small.

The manipulation of Lagrangian states involves holomorphic Lagrangians near
$\Lambda\cap V$ in $\widetilde{M}$, defined as follows: to
$I_{W,k}^{\Phi}(a)$ we
associate \begin{equation}\label{eq:def_LPhi}\mathcal{L}_{\Phi}:=\{\widetilde{\nabla}\widetilde{\Phi}=0\}
\end{equation}
where $\widetilde{\Phi}$ is the $I$-holomorphic extension of $\Phi$: it is
a section of $\widetilde{L}$ over
a neighbourhood of $\Lambda\cap V$; moreover $\widetilde{\nabla}$ is the connection on
$\widetilde{L}$ defined through Proposition
\ref{prop:extension_nabla}. In the ``real'' case, one can
alternatively define $\mathcal{L}_{\Phi}$ as the set on which the
state $I_{W,k}^{\Phi}(a)$ concentrates; this fails here, but Lagrangian states are exponentially small on (real)
points that lie sufficiently far away from their Lagrangians.

\begin{prop}\label{prop:Lagrangians}Let $\Lambda$ be a real-analytic
  Lagrangian and suppose that $\Phi$ is close (in real-analytic
  topology) to $\Phi_{\Lambda}$. Then
  $\mathcal{L}_{\Phi}$ is a Lagrangian submanifold with trivial
  Bohr-Sommerfeld class; it is close, in real-analytic topology, to the
  holomorphic extension $\widetilde{\Lambda}$ of
  $\Lambda$.

  Conversely, to any Lagrangian $\mathcal{L}$ of $\widetilde{M}$ with trivial
  Bohr-Sommerfeld class and close in real-analytic topology to $\widetilde{\Lambda}$, is associated a section $\Phi$ over
  a neighbourhood of $\Lambda\cap V$ such that
  $\mathcal{L}=\mathcal{L}_{\Phi}$; $\Phi$ is
  unique up to a multiplicative factor and close to $\Phi_{\Lambda}$.

  Writing $\widetilde{M}=M\times \overline{M}$, the Lagrangians above
  are transverse to the fibres of the projection over the first factor.
\end{prop}

\begin{proof}
If $\Phi=\Phi_{\Lambda}$, as defined in Proposition
  \ref{prop:real_Lagr_state}, then
  $\mathcal{L}_{\Phi}=\widetilde{\Lambda}$; in fact since the
  curvature of $\widetilde{\nabla}$ is $-i\Omega$,
  $\widetilde{\nabla}\widetilde{\Phi_{\Lambda}}$ is a defining
  function for $\widetilde{\Lambda}$.

  Suppose that $\Phi$ is close (in a real-analytic topology) to
  $\Phi_{\Lambda}$. In particular, over a neighbourhood of
  $\Lambda\cap V$, $\widetilde{\Phi}$ is close to
  $\widetilde{\Phi_{\Lambda}}$ in the $C^2$ topology, so that
  $\mathcal{L}_{\Phi}$ is a half-dimension,
  $I$-holomorphic submanifold.

  Using again the curvature identities for $\widetilde{\nabla}$, we
  find that $\mathcal{L}_{\Phi}$ is isotropic for the holomorphic
  symplectic form $\Omega$; therefore it is a Lagrangian.

  Reciprocally, the construction of $\Phi$ from $\mathcal{L}$ mimics
  the proof 
  of Proposition \ref{prop:real_Lagr_state}: fixing the value of
  $\widetilde{\Phi}$ at any point on $\mathcal{L}$, one can define
  $\widetilde{\Phi}$ on $\mathcal{L}$ by parallel transport. Now,
  by Proposition \ref{prop:ext_totally_real_mflds}, $\widetilde{\Lambda}$ is
  $\widetilde{J}$-totally real, and therefore $\mathcal{L}$, which lies close
  to it, is also
  $\widetilde{J}$-totally real\footnote{One should be aware of the fact that
    $I$-holomorphic Lagrangians are not necessarily $\widetilde{J}$-totally
    real: on $\widetilde{\C}=\C\times\C$, the manifold $\{w=0\}$ is
    $(J,-J)$-holomorphic, Lagrangian for the holomorphic symplectic
    form $\dd z \wedge \dd \overline{w}$, but also $(J,J)$-holomorphic.}; therefore there exists a unique
  $\widetilde{J}$-holomorphic section $\widetilde{\Phi}$ on a neighbourhood of
  $\mathcal{L}$ which coincides with our construction of
  $\widetilde{\Phi}$ on $\mathcal{L}$. Since $\widetilde{J}=(J,J)$ commutes with $I=(J,-J)$ and
  $\widetilde{\Phi}|_{\mathcal{L}}$ is $I$-holomorphic, then
  $\widetilde{\Phi}$ is $I$-holomorphic.

  To prove the last claim, we consider suitable charts near a point of
  $\Lambda$: a chart for $M$, a suitable Kähler potential $\psi$, and an associated Hermitian chart for $L$. In these charts, the section $\Phi$
  reads
  \[
    y\mapsto \exp\left(-\tfrac{\psi(y)}{2}+\phi(y)\right)
  \]
  where $\phi$ is a holomorphic function. On $\widetilde{M}$, the
  equation $\widetilde{\nabla}\widetilde{\Phi}=0$ boils down to
  the Hamilton-Jacobi equation
  \begin{equation}\label{eq:wbarstar}
    \partial_y\widetilde{\psi}(y',y'')=\partial\phi(y').
  \end{equation}
  Now, in the chart,
  $\partial_{y'}\partial_{y''}\widetilde{\psi}$ is positive
  non-degenerate, hence the equation above has at most one solution of
  the form $\overline{y''}=(y'')^*(y')$, and if it has one, it
  depends continuously on $y'$. This proves that
  $\mathcal{L}=\{\widetilde{\nabla}\widetilde{\Phi}=0\}$ is
  transverse to the first factor of $\widetilde{M}=M\times \overline{M}$.
\end{proof}

 An example of analytic Fourier Integral Operator is the Bergman kernel, which is
associated with the diagonal in $N\times \overline{N}$. In
\cite{deleporte_real-analytic_2022} was performed a study of analytic Fourier Integral Operators when the
reference real Lagrangian $\Lambda$ is the diagonal of $N\times
\overline{N}$. Such operators are close to identity, in the sense that
they do not move microsupports too far, see more generally Proposition
\ref{prop:propag_support}.

\subsection{Calculus of Fourier Integral Operators and Lagrangian
  sections}
\label{sec:calc-four-integr}

In general, the action of an analytic Fourier Integral Operator on a
Lagrangian state is another Lagrangian state, if the domains behave well. This allows us to compose analytic Fourier
Integral operators, and to invert them under a natural condition on
the Lagrangian and the principal symbol.

We first prove a general composition formula (under a transverse
intersection hypothesis), which will be applicable to a variety of
situations: composing Fourier Integral Operators, applying them on
Lagrangian sections, and computing the scalar product between two
Lagrangian sections. Before doing so, we have to clarify the geometric
condition under which one will be able to perform these compositions.

\begin{defn}\label{def:transverse_compo}
  Let $M_0,\cdots,M_\ell$ be smooth Kähler manifolds. For $1\leq j\leq \ell$ let $V_j\Subset U_j$ be
  open subsets of $M_{j-1}\times \overline{M_j}$ and let $\mathcal{L}_j$ be a
  Lagrangian of $U_j$. We say that
  $\mathcal{L}_1,\cdots,\mathcal{L}_\ell$ are \emph{transversally
    composable} near $V_1,\cdots,V_{\ell}$ under the two following conditions:
  \begin{enumerate}
  \item the product $\mathcal{L}_1\times \mathcal{L}_2\times \ldots
    \times \mathcal{L}_\ell$ is transverse to the product of diagonals
    $M_0\times {\rm diag}(\overline{M_1}\times M_1)\times \ldots\times
    {\rm diag}(\overline{M_{\ell-1}}\times M_{\ell-1})\times
    \overline{M_{\ell}}$ on a neighbourhood of the closure of $V_1\times \ldots\times V_{\ell}$
    (recall that this means that the sum of the tangent spaces of
    these manifolds is the total tangent space);
  \item the intersection of these two manifolds is a graph
    over its projection on $M_0\times \overline{M_{\ell}}$.
  \end{enumerate}
  Under these hypotheses, for some $W_j\Supset V_j$, the projection
  \[
    M_0\times \overline{M_{\ell}}\supset \mathcal{L}=\{(x_0,x_\ell)\in
    M_0\times \overline{M_{\ell}},\exists (x_1,\cdots,x_{\ell-1})\in
    M_1\times \ldots\times M_{\ell-1},\forall 1\leq j\leq \ell,
    (x_{j-1},x_j)\in \mathcal{L}_j\cap W_j\}
  \]
  is a Lagrangian. We call this Lagrangian the \emph{composition} and
  denote it
  $\mathcal{L}_1\circ \mathcal{L}_2\circ \ldots \circ \mathcal{L}_{\ell}$.
\end{defn}
In practice, our Lagrangians will live on complexified Kähler
manifolds, therefore we will apply Definition
\ref{def:transverse_compo} to holomorphic Lagrangians; notice that the
composition of holomorphic Lagrangians is again holomorphic.

The composed Lagrangian corresponds to the composition of Fourier
Integral operators. Let us describe this in terms of critical points
of phase functions. To this end we use the identification between
$\widetilde{L}\otimes \widetilde{\overline{L}}$ and $\C$ given by
Proposition \ref{prop:extend_h}.

\begin{prop}\label{prop:composed_phase}
  Let $M_0,\cdots,M_{\ell}$ be compact, real-analytic Kähler
  manifolds. For $1\leq j\leq \ell$ let $V_j\Subset U_j$ be open
  subsets of $\widetilde{M}_{j-1}\times \overline{\widetilde{M}}_j$
  and let $\mathcal{L}_j$ be a $I$-holomorphic Lagrangian of
  $U_j$. Suppose that $\mathcal{L}_1,\cdots,\mathcal{L}_\ell$ are transversally
  composable, and let $\mathcal{L}$ be their composition. Let $\Phi_1,\cdots,\Phi_{\ell}$ be associated
  phase functions. Then for every $(x_0,x_{\ell})$ near $\mathcal{L}$,
  there exists an open subset of $\widetilde{M}_1\times
  \widetilde{M}_2\times \ldots\times \widetilde{M}_\ell$ on which
  \[
    \Phi(x_1,\cdots,x_{\ell-1})\mapsto \Phi_1(x_0,x_1)\cdot
    \Phi_2(x_1,x_2) \cdot \ldots
    \cdot \Phi_{\ell}(x_{\ell-1},x_{\ell})\in (L_0)_{x_0}\otimes (\overline{L_{\ell}})_{x_\ell}
  \]
  is well-defined and
  has a unique critical point; it is non-degenerate. The value of
  $\Phi$ at the critical point defines a holomorphic section $\Phi_{\mathrm{crit}}$ of
  $L_0\boxtimes \overline{L}_{\ell}$ on a neighbourhood of $\mathcal{L}$.

  If $(x_0,x_{\ell})$ belongs to
  $\mathcal{L}$, at the critical point one has $(x_{j-1},x_{j})\in \mathcal{L}_j$ for
  every $1\leq j\leq \ell$, and
  \[
    \widetilde{\nabla}\Phi_{\mathrm{crit}}(x_0,x_{\ell})=0.
  \]
\end{prop}
\begin{proof}Define
  \[
    \Psi(x_0,x_1,\cdots,x_{\ell-1},x_{\ell})=\Phi_1(x_0,x_1)\cdot
    \Phi_2(x_1,x_2) \cdot \ldots
    \cdot \Phi_{\ell}(x_{\ell-1},x_{\ell})\in (L_0)_{x_0}\otimes
    (\overline{L_{\ell}})_{x_\ell}
  \]
  on the intersection of the natural definition domains.

  Let $(x_0,x_{\ell})\in \mathcal{L}$. Let $(x_1,\cdots,x_{\ell-1})$
  be such that $(x_{j-1},x_j)\in \mathcal{L}_j$ for every $1\leq j
  \leq \ell$. Then, by Proposition \ref{prop:extend_h} and \eqref{eq:def_LPhi}, one has, at
  this point,
  \begin{equation*}
  \begin{cases}
    \partial_{x_j}\Psi=0\qquad \forall 1\leq j\leq \ell-1\\
    \widetilde{\nabla}_{x_0}\Psi=0\\
    \widetilde{\nabla}_{x_{\ell}}\Psi=0.
  \end{cases}
\end{equation*}
We claim that $(x_1,\cdots,x_{\ell-1})$ is a
  non-degenerate critical point for $x_0,x_{\ell}$ fixed. From there,
  the rest of the proof proceeds as follows: since existence and
  uniqueness of a non-degenerate critical point is stable under
  deformation, for $(x_0,x_{\ell})$ in a neighbourhood of
  $\mathcal{L}$ there exists a unique critical point close to a point
  as above, and it is non-degenerate. In particular, it depends
  holomorphically on $(x_0,x_{\ell})\in M_0\times
  \overline{M}_{\ell}$, and therefore, from the computation above we
  obtain that $\widetilde{\nabla}\Phi_{\mathrm{crit}}$ vanishes on $\mathcal{L}$.

  To prove that $(x_1,\cdots,x_{\ell-1})$ is a non-degenerate critical
  point, we use the following two facts:
  \begin{itemize}
  \item $\Psi$ is a $\widetilde{J}$-holomorphic function of $x_0$ and a
    $\widetilde{J}$-anti-holomorphic function of $x_{\ell}$;
  \item holomorphic Lagrangians on $\widetilde{M}$ are transverse with
    respect to the projection on the holomorphic factor (Proposition \ref{prop:Lagrangians}).
  \end{itemize}
  It follows from the first fact that the system
  \[
    \partial_{x_j}\Psi=0\qquad \forall 1\leq j \leq \ell-1
  \]
  is $\tilde{J}$-holomorphic with respect to $x_0$ and $\widetilde{J}$-anti-holomorphic with
  respect to $x_{\ell}$, and it follows from the second fact that
  (decomposing $x_j\in \widetilde{M_j}$ into $(x_j',x_j'')\in
  M_j\times \overline{M_j}$) for some holomorphic $f$,
  \[
    (\widetilde{\nabla}_{x_0}\Psi=0 \text{ and }
    \widetilde{\nabla}_{x_{\ell}}\Psi=0) \Leftrightarrow
    (x_0'',x_\ell')=f(x_0',x_{\ell}'',x_1,\cdots,x_{\ell-1})
  \]
  in a non-degenerate way (the functions defining both sides generate
  the same ideal).
  
  Crucially, being a $\widetilde{J}$-holomorphic and $I$-holomorphic
  function of $x_0$ means exactly being a holomorphic function of
  $x_0'$. Following the hypotheses that
  $\mathcal{L}_1,\cdots,\mathcal{L}_{\ell}$ are composable, one has
  also, in a non-degenerate way, for some holomorphic $\mathcal{F}$
  \[
    (\widetilde{\nabla}_{x_0}\Psi=0 \text{ and }
    \widetilde{\nabla}_{x_{\ell}}\Psi=0 \text{ and }
    \partial_{x_j}\Psi=0\,\forall 1\leq j \leq \ell-1)\Leftrightarrow
    (x_0'',x_{\ell}',x_1,\cdots,x_{\ell-1})=\mathcal{F}(x_0',x_{\ell}'').
  \]
  Since $(\partial_{x_j}\Psi)_{1\leq j \leq \ell-1}$ does not depend on
  $x_0''$ and $x_{\ell}'$, we obtain that this system is
  non-degenerate and solved exactly when
  \[
    (x_1,\cdots,x_{\ell-1})=F(x_0',x_{\ell}'')
  \]
  where $F$ contains the last $\ell-1$ components of $\mathcal{F}$,
  and then
  \[
    \mathcal{F}(x_0',x_{\ell}'')=(f(x_0',x_{\ell}'',F(x_0',x_{\ell}'')),F(x_0',x_{\ell}'')).
    \]
\end{proof}

\begin{rem}~
  \begin{enumerate}
  \item As explained before, we treat the composition of an arbitrary
    (finite) number of Lagrangians, in order to perform all stationary
    phases in one go in Proposition \ref{prop:Compo_multi_FIO}. This allows us to
    separate the proofs of the rest of this article into two steps: first proving identities on the (analytic formal) symbolic calculus,
    and then showing that these formal arguments
    can be realised into licit manipulations of objects, modulo
    exponentially small remainders. If we were to prove composition
    ``two by two'', one would then have to check that the
    exponentially small remainders at each step stay exponentially
    small after the next step, even though each Fourier Integral
    Operator can enlarge norms by exponentially large factors.
  \item The first condition of Definition \ref{def:transverse_compo} is traditional in texts concerned with the general
theory of Fourier Integral Operators, see e.g. \cite{hormander_fourier_1971}. The second condition is automatic (provided the first one
holds) when every
$\mathcal{L}_j$ is locally the graph of a symplectomorphism between $M_{j-1}$
and $M_j$.  One can also check that if
$\mathcal{L}_j$ is a local symplectomorphism for
every $j\leq \ell-1$ and if $M_{\ell}=\{0\}$ (corresponding to the
action of several Fourier Integral Operators on a Lagrangian state) then the second condition
is always satisfied. For a more thorough discussion of this second condition see \cite{guillemin_moment_2005}.
\item This second condition can be slightly weakened into the fact
  that the intersection is \emph{locally} a graph over the base. When
  performing stationary phase, this will mean that instead of having
  one critical point, we will have a finite sum of contributions from
  different critical points. This situation will not appear in the
  rest of this article, and it would make the notation in the next
  Proposition substantially more cumbersome; in this case the output
  of stationary phase is a locally finite sum of Lagrangian states,
  and the proof of this more general fact is essentially the same one.
\item 
The definition of the composition is associative,
and moreover if individually, for every $1\leq j \leq \ell-1$,
$\mathcal{L}_{j-1}$ and $\mathcal{L}_j$ are transversally composable,
then $\mathcal{L}_1,\cdots,\mathcal{L}_\ell$ are altogether
transversally composable, but the reciprocal is not true.
  \end{enumerate}
\end{rem}
\begin{prop}\label{prop:Compo_multi_FIO}
  Let $M_0,M_1,\cdots,M_\ell$ be compact, real-analytic, quantizable
  Kähler manifolds. For $1\leq j\leq \ell$ let $V_j\Subset U_j$ be
  open subsets of $M_{j-1}\times \overline{M_j}$ and let
  $\mathcal{L}_j^0\subset U_j$ be transversally composable Lagrangians
  near $V_1,\cdots,V_{\ell}$.

  Let $Z\Subset \C^K$ be an open set
  containing $0$. For $1\leq j\leq \ell$ let $\Phi_j:Z\to
  H^0(U_j,L_{j-1}\boxtimes
  \overline{L_j})$ be holomorphic; suppose that for $z=0$
  one has $|\Phi_j^0|=1$ on $\mathcal{L}_j^0$. Let also $r,R,m>0$.

  Then there exist $C,c,r',R',m'>0$, a neighbourhood $Z'$ of $0$ in
  $\C^K$, small neighbourhoods $W'\Subset W$ of
  $\mathcal{L}_1^0\circ \mathcal{L}_2^0\circ \ldots\circ
  \mathcal{L}_\ell^0$, and
  real-analytic maps
  \[
    \Phi:Z'\to H^0(W,L_0\boxtimes L_{\ell})
  \]
  \[A:Z'\times S^{r,R}_m(U_1)
  \times \ldots \times S^{r,R}_m(U_\ell)\to S^{r',R'}_{m'}(W)\] such
  that, uniformly for $z\in Z'$, $(x_0,x_{\ell})\in W'$ and
  $a_1,\cdots,a_{\ell}\in  S^{r,R}_m(U_1) \times \ldots \times S^{r,R}_m(U_\ell)$,
  \begin{multline}\label{eq:Compo_FIO}
    \left|I^{\Phi(z)}_{W',k}(A)(x_0,x_{\ell})-\int I^{\Phi_1(z)}_{V_1,k}(a_1)(x_0,x_1)\cdot
      I^{\Phi_2(z)}_{V_2,k}(a_2)(x_1,x_2)\cdot \ldots \cdot
      I^{\Phi_\ell(z)}_{V_\ell,k}(a_{\ell})(x_{\ell-1},x_{\ell})\dd
      x_1\ldots \dd x_{\ell-1}\right|\\ \leq
    Ce^{-ck}\|a_1\|_{S^{r,R}_m(U_1)}\ldots
    \|a_{\ell}\|_{S^{r,R}_m(U_{\ell})}.
  \end{multline}
  Moreover,
  \[
    \mathcal{L}_{\Phi(z)}=\mathcal{L}_{\Phi_1(z)}\circ \ldots \circ
    \mathcal{L}_{\Phi_\ell(z)}
  \]
  and the principal symbol of $A$ is of the form
  \[
    s(z,x_0,x_{\ell})a_{1;0}(x_0,(x_1')^*(z,x_0,x_\ell))a_{2;0}((x_1'')^*(z,x_0,x_\ell),(x_2')^*(z,x_0,x_{\ell}))\ldots
    a_{\ell;0}((x_\ell'')^*(z,x_0,x_\ell),x_\ell)
  \]
  where $s$ is holomorphic and non-vanishing, and where, in item 2 of Proposition \ref{prop:Lagrangians}, the
  intersection is of the form
  \[
    \{(x_0,(x_1')^*(z,x_0,x_{\ell}),(x_1'')^*(z,x_0,x_{\ell}),\cdots,(x_{\ell-1}')^*(z,x_0,x_{\ell}),(x_{\ell-1}'')^*(z,x_0,x_\ell),x_{\ell});\text{
      where }(x_0,x_{\ell})\in
    W\}.
  \]
  The order of $A$ is the sum of the orders of the $a_j$'s, minus $\sum_{j=1}^{\ell-1}\dim_{\C}(M_j)$.
\end{prop}
\begin{proof}
  The proof will consist in the application of the stationary phase method
  to the integral featuring in \eqref{eq:Compo_FIO}. We will prove
  that this stationary phase can be performed in a model case where
  the parameter in $Z$ is equal to $0$ and $(x_0,x_{\ell})$ lies on
  the composed Lagrangian, then apply a
  deformation argument.

  Let $\mathcal{L}^0=\mathcal{L}_{\Phi_1}^0\circ \ldots \circ
  \mathcal{L}_{\Phi_{\ell}}^0$. This is a real Lagrangian. Suppose that
  $(x_0,x_{\ell})\in \mathcal{L}^0$ and pick the parameter
  $z\in Z$ to be equal to $0$. The integral in \eqref{eq:Compo_FIO}
  is, by Proposition \ref{prop:good_def_Lagr_state}, of the form
  \[
    \int_{\mathcal{W}} [\Phi_1(0)(x_0,x_1)\cdot \ldots \cdot
    \Phi_{\ell}(0)(x_{\ell-1},x_{\ell})]^{\otimes k}a_1(x_0,x_1)\ldots
    a_\ell(x_{\ell-1},x_{\ell})\dd x_1\ldots \dd x_{\ell-1}+O(e^{-ck})
  \]
  where $\mathcal{W}$ is an open neighbourhood of the intersection between
  $\mathcal{L}_{\Phi_1}^0\times \ldots \times
  \mathcal{L}_{\Phi_{\ell}}^0$ and the interior diagonals as described
  in Definition \ref{def:transverse_compo}.

By Proposition \ref{prop:real_Lagr_state}, the norm of the section
under brackets behaves like $1$ minus the squared distance to this
intersection. Therefore, if $(x_0,x_{\ell})\in \mathcal{L}^0$,
in this oscillatory integral, there is a unique critical point for the
phase, which lies on the real locus by Proposition \ref{prop:composed_phase}. The imaginary part of the phase grows quadratically away
from this critical point. We are in position to apply analytic stationary phase
\cite{sjostrand_singularites_1982}, and the result is of the following form (for $(x_0,x_{\ell})\in
\mathcal{L}^0$ and $z=0$):
\[
  \Phi(x_0,x_{\ell})^{\otimes k}A(x_0,x_{\ell})+O(e^{-ck}).
\]
Here, the value of $\Phi$ is prescribed by the critical points, and in
particular $|\Phi(x_0,x_{\ell})|=1$. The principal symbol of $A$, with respect to
the product of the principal symbols of $a_1,\cdots,a_{\ell}$, picks
up a factor $k^{-\frac{d}{2}}J(x_0,x_{\ell})$, where $J$ does not
vanish and is related to the
Hessian of the phase, and $d$ is the real dimension of the integration set
$\mathcal{W}$, that is, $d=2\sum_{j=1}^{\ell-1}\dim_{\C}(M_j)$.

The hypotheses of stationary phase are stable under small deformation
of the phases involved and the parameters. Therefore, for $z$ close to
$0$ and $(x_0,x_\ell)$ close to $\mathcal{L}^0$, one can
perform a small contour deformation and stationary phase to the
integral above, and we find an expression of the form
\[
  \Phi(z)(x_0,x_{\ell})^{\otimes k}A(z)(x_0,x_{\ell})+O(e^{-ck}),
\]
where the big $O$ depends on the data above as specified in \eqref{eq:Compo_FIO}.

To conclude, by Proposition \ref{prop:composed_phase}, $\Phi(z)$ has
precisely for Lagrangian $\mathcal{L}_{\Phi_1(z)}\circ \ldots\circ \mathcal{L}_{\Phi_{\ell}(z)}$. 
\end{proof}

\begin{rem}
  We will apply Proposition \ref{prop:Compo_multi_FIO} in the context
  of the spectral study of a non-self-adjoint Berezin-Toeplitz
  operator which depends holomorphically on a parameter $z\in \C$. We will always
  proceed by deformation from the real case: we assume that when $z=0$
  the operator is self-adjoint and we can apply the ``usual''
  theory; the typical case is
  \[
    T_k^{\rm cov}(f+izg)
  \]
  where $f,g$ are real-valued.

  The underlying geometric data (notably, normal forms and Lagrangian
  states) will depend holomorphically on $z$, and when $z=0$ we have
  real Lagrangian states. The point of Proposition
  \ref{prop:Compo_multi_FIO} is that the calculus of these Lagrangian
  states is stable under small deformations in $z$.
\end{rem}

\section{Fourier Integral Operators in practice}
\label{sec:prax-four-integr}

Fourier Integral operators correctly propagate the analytic microsupport,
if one is careful to fix the constants in the right order.
\begin{prop}\label{prop:propag_support}
  Let $U\subset M\times \overline{M}$ and let $\mathcal{L}_0$ be a
  real Lagrangian of $U$. Let $V\subset M$ and define
  \[
    \mathcal{L}_0\circ V=\{(x,\overline{y})\in \mathcal{L}_0, y\in
    V\}.
  \]
  Let $z\mapsto \mathcal{L}_z$ be a Lagrangian of $U$ with
  holomorphic dependence on $z$. Let $I_z$ be a corresponding family
  of analytic Fourier Integral Operators as in Definition \ref{def:Lagr-state}.

  For every $W\Subset \mathcal{L}_0\circ V$ and every $c>0$ there
  exists an open neighbourhood $\mathcal{Z}$ of $0$ in $\C$ and there exists $c'>0$ such that, for
  every $u\in H_0(M,L^{\otimes k})$ with $\|u\|_{L^2(M)}=1$ and
  $\|u\|_{L^2(V)}=O(e^{-ck})$, one has $\|I_zu\|_{L^2(W)}=O(e^{-c'k})$.
\end{prop}
\begin{proof}
  It suffices to decompose the integral
  \[
    (I_z u)(x)=\int \Phi_z^{\otimes k}(x,y)a(x,y;k^{-1})u(y)\dd y
  \]
  into two parts.

  If we integrate on $y\in V$, $u$ is exponentially small; moreover
  $|\Phi_z|\leq e^{C|z|}$ for some $C>0$ since $|\Phi_0|\leq 1$. Thus this part of
  the integral is exponentially small.

  We now integrate on $y\notin V$. There, by construction, $u$ is uniformly bounded,
  and $|\Phi_z|\leq e^{-c_1+C|z|}$ since $|\Phi_0|\leq e^{-c_1}$,
  with $c_1>0$. This concludes the proof.
\end{proof}

One can invert Fourier Integral Operators under natural conditions on
their phase and symbol.
\begin{prop}\label{prop:invert_FIO}Let $M_i,M_f$ be real-analytic,
  quantizable Kähler manifolds.
  Let $U\subset M_f\times \overline{M_i}$ and let $\mathcal{L}_0$ be a real
  Lagrangian of $U$ which is the graph of an invertible (symplectic)
  map: for every $x\in M_f$ there exists at most one $y\in M_i$ such that
  $(x,\overline{y})\in \mathcal{L}_0$ and for every $y\in M_i$ there
  exists at most one $x\in M_f$ such that $(x,\overline{y})\in
  \mathcal{L}_0$.

   Let $V\Subset U$ and define the following open sets and Lagrangians:
  \begin{itemize}
  \item $U_{\rm inv}=\{(y,\overline{x})\in M_i\times \overline{M_f},
    (x,\overline{y})\in U\}$;
  \item $V_{\rm inv}=\{(y,\overline{x})\in M_i\times
    \overline{M_f},(x,\overline{y})\in V\}$;
  \item $\mathcal{L}_{0,\rm inv}=\{(y,\overline{x})\in M_i\times
    \overline{M_f},(x,\overline{y})\in \mathcal{L}_0\}$;
  \item $V_i=\{y\in M_i,\exists x\in M_f,(x,\overline{y})\in V\cap \mathcal{L}_0\}$;
  \item $V_f=\{x\in M_f,\exists y\in M_i,(x,\overline{y})\in V\cap \mathcal{L}_0\}$.
  \end{itemize}
  Then for every $\varepsilon>0$, for every
  $\Phi$ close to $\Phi_0$ on $\mathcal{L}_0$ (in a way which depends on $V_i$ and
  $V_f$ and $\varepsilon$), there exists a section $\Phi_{\rm inv}$ close to 1 on 
   $\mathcal{L}_{0,\rm inv}$, and for every
  real-analytic symbol $a$ defined near $V\cap \mathcal{L}_0$ with principal symbol $a_0\neq 0$ there
  exist a real-analytic symbol $a_{\rm inv}$ defined near $V_{\rm
    inv}\cap \mathcal{L}_{0,\rm inv}$ and $c>0$ such that
  \begin{enumerate}
  \item for every $u\in H^0(M_i,L^{\otimes k})$ one has
    \[
      I^{\Phi_{\rm inv}}_{k}(a_{\rm
        inv}) I^{\Phi}_k(a)u=u+O(e^{-ck}\|u\|_{L^2})+O(e^{\varepsilon
        k}\|u\1_{V_i^c}\|_{L^2});
    \]
    \item for every $u\in H^0(M_f,L^{\otimes k})$ one has
    \[
      I^{\Phi}_k(a)I^{\Phi_{\rm inv}}_{k}(a_{\rm
        inv})u=u+O(e^{-ck}\|u\|_{L^2})+O(e^{\varepsilon
        k}\|u\1_{V_f^c}\|_{L^2}).
    \]
  \end{enumerate}
\end{prop}

In this situation, we say that $I^{\Phi}_k(a)$ and $I^{\Phi_{\rm inv}}_{k}(a_{\rm  inv})$ are \emph{microlocal inverses} of each other on the domains $V_i$ and $V_f$.

\begin{proof}
  Let $\mathcal{L}$ be the Lagrangian associated with $\Phi$ (see Proposition
  \ref{prop:Lagrangians}) and define
  \[
    \mathcal{L}_{\rm inv}=\{(y,\overline{x})\in \widetilde{M_i}\times
    \widetilde{\overline{M_f}},(x,\overline{y})\in \mathcal{L}\}.
  \]
  Then $\mathcal{L}_{\rm inv}$ is a holomorphic Lagrangian close to
  $\mathcal{L}_{0,\rm inv}$, with trivial Bohr-Sommerfeld
  class. Therefore there exists $\Phi_{\rm inv}$ over a neighbourhood of
  $V_{\rm inv}\cap \mathcal{L}_{0,\rm inv}$ whose Lagrangian is
  $\mathcal{L}_{\rm inv}$. If $\Phi$ is close to $1$ on
  $\mathcal{L}_0$ in real-analytic topology, then $\Phi_{\rm
    inv}$ is close to a constant on $\mathcal{L}_{0,\rm inv}$ in
  real-analytic topology.

  Let $a_1$ be any real-analytic symbol near $V_{\rm inv}\cap
  \mathcal{L}_{0,\rm inv}$ with nonvanishing principal symbol. By
  Proposition \ref{prop:Compo_multi_FIO}, the section
  \[
    (x,\overline{z})\mapsto \int
    I_k^{\Phi_{{\rm inv}}}(a_1)(x,\overline{y})\cdot
    I_k^{\Phi}(a)(y,\overline{z}) \ \dd y
  \]
  is, near $V_i\times V_i$, a Lagrangian state; its Lagrangian is
  $\mathcal{L}_{\rm inv}\circ \mathcal{L}$, that is, the diagonal in
  $\widetilde{M}_i \times \widetilde{\overline{M}}_i$. By the uniqueness
  part of Proposition \ref{prop:Lagrangians}, the associated phase is
  a multiple of the phase $\Psi$ of the Bergman kernel. Up to
  multiplying $\Phi_{\rm inv}$ by a constant, the phase is then
  precisely the phase of the Bergman kernel, and in particular,
  near $V_i\times V_i$, the integral kernel $I_k^{\Phi_{\rm
      inv}}(a_1)\circ I_k^{\Phi}(a)$ is that of a covariant analytic
  Berezin-Toeplitz operator, with non-vanishing principal symbol. 

  By Proposition \ref{prop:Bergman_asymp}, this operator can be inverted, and therefore
  there exists an analytic symbol $r$ near the diagonal of $V_i\times
  V_i$ such that the integral kernel of $T_k^{\rm cov}(r)\circ
  I_k^{\Phi_{\rm inv}}(a_1)\circ I_k^{\Phi}(a)$ is, near $V_i\times
  V_i$, exponentially close to that of the Bergman kernel on $M_i$.

  Outside of a neighbourhood of $V_i\times V_i$, the integral kernel of $T_k^{\rm cov}(r)\circ
  I_k^{\Phi_{\rm inv}}(a_1)\circ I_k^{\Phi}(a)$ is bounded by
  $N^{K_0}(\sup|\Phi|\sup |\Phi_{\rm inv}|)^k\leq Ce^{\varepsilon k}$,
  for every fixed in advance $\varepsilon>0$ if $\Phi$ was chosen close
  enough to 1 on $\mathcal{L}_0$.

  Applying again Proposition \ref{prop:Compo_multi_FIO} to obtain $T_k^{\rm cov}(r)\circ
  I_k^{\Phi_{\rm inv}}(a_1)=I_k^{\Phi_{\rm inv}}(a_{\rm inv})$, we
  finally have, given two small neighbourhoods $V_i\Subset W_i\Subset U_i$
  of $V_i$, that
  \[
    \|(I_k^{\Phi_{\rm inv}}(a_{\rm
      inv})I_k^{\Phi}(a)-1)u\|_{L^2(W_i)}\leq
    Ce^{-ck}\|u\|_{L^2(U_i)}+Ce^{\varepsilon k}\|u\|_{L^2(U_i^c)}
  \]
  and
  \[
  \|(I_k^{\Phi_{\rm inv}}(a_{\rm
      inv})I_k^{\Phi}(a)-1)u\|_{L^2(W_i^c)}\leq
    Ce^{-ck}\|u\|_{L^2(V_i)}+Ce^{\varepsilon k}\|u\|_{L^2(V_i^c)}
  \]
  This concludes the first part of the claim. It remains to study
  $I_k^{\Phi}(a)\circ I_k^{\Phi_{\rm inv}}(a_{\rm inv})$. Since
  $\mathcal{L}\circ \mathcal{L}_{\rm inv}$ is equal to the diagonal of
  $\widetilde{M_f}\times \widetilde{\overline{M_f}}$ near $V_f\times V_f$,
  the integral kernel of $I_k^{\Phi}(a)\circ I_k^{\Phi_{\rm
      inv}}(a_{\rm inv})$ is, on this set, of the form $e^{\alpha
    k}T_k^{\rm cov}(b)$ where $b$ is a real-analytic symbol (with
  non-vanishing principal symbol) and
  $\alpha\in \C$. Now, let $W_f$ be a small neighbourhood of $V_f$. For all $u$ microlocalised inside $W_f$,
  $I_{k}^{\Phi_{\rm inv}}(a_{\rm inv})u$ is microlocalised on a
  small neighbourhood of $V_i$, and therefore
  \[
    [I_k^{\Phi}(a)\circ I_k^{\Phi_{\rm
        inv}}(a_{\rm inv})]^2u=I_k^{\Phi}(a)\circ I_k^{\Phi_{\rm
        inv}}(a_{\rm inv})u+O(e^{-ck}).
  \]
  In particular, $I_k^{\Phi}(a)\circ I_k^{\Phi_{\rm
        inv}}(a_{\rm inv})$ acts (micro)locally as a projector on
    $W_f$. Thus $\alpha=1$ and $b$ is its own square for the formal product of
    symbols of covariant Toeplitz operators on $W_f$. Thus $b$ is the
    symbol of the Bergman projector (this can be determined, for
    instance, by usual, order-by-order, stationary phase). And finally
    for $u$ microlocalised on $W_f$ one has
    \[
      I_k^{\Phi}(a)\circ I_k^{\Phi_{\rm
          inv}}(a_{\rm inv})u=u+O(e^{-ck}).
    \]
    From this we obtain the desired claim as previously.
  \end{proof}

  Fourier Integral Operators as above conjugate Berezin--Toeplitz
  operators to each other, and
  we can describe their action on principal symbols.
  \begin{prop}\label{prop:conjug_Toep_FIO}
    In the situation of Proposition \ref{prop:invert_FIO}, if $b$ is an analytic
    symbol on a neighbourhood of $V_f$ then there exists an analytic symbol
    $r$ on a neighbourhood of $V_i$ such that, for every $u\in
    H^0(M_i,L_i^{\otimes k})$, 
    \[
      I_k^{\Phi_{\rm inv}}(a_{\rm
        inv})T_k^{\rm cov}(b)I_k^{\Phi}(a)u=T_k^{\rm cov}(b\circ
      \kappa^{-1}+k^{-1}r)u+O(e^{-ck}\|u\|_{L^2})+O(e^{\varepsilon
        k}\|u\1_{V_i^c}\|_{L^2}).
    \]
    Moreover, if $b$ is an analytic symbol on a neighbourhood of $V_i$ then
    there exists an analytic symbol
    $r$ on a neighbourhood of $V_f$ such that, for every $u\in
    H^0(M_f,L_f^{\otimes k})$,
    \[
      I_k^{\Phi}(a)T_k^{\rm cov}(b)I_k^{\Phi_{\rm inv}}(a_{\rm
        inv})=T_k^{\rm cov}(b\circ \kappa+k^{-1}r)+O(e^{-ck}\|u\|_{L^2})+O(e^{\varepsilon
        k}\|u\1_{V_f^c}\|_{L^2}).
    \]
    
  \end{prop}
  \begin{proof}
    Let us prove the first statement. By Proposition \ref{prop:Compo_multi_FIO}, the product of the
    three operators on the right-hand side is a Fourier Integral Operator whose Lagrangian is
    the diagonal of $\widetilde{M}_i \times \widetilde{\overline{M}}_i$, that is, a covariant
    Berezin--Toeplitz operator. Its principal symbol is of the form
    $J b\circ \kappa^{-1}$ for some function $J$ not depending on $b$. However by
    Proposition \ref{prop:invert_FIO} we know that if $b=1$ the
    principal symbol of the output is $1$, therefore $J=1$.

    We now turn to the second statement. Again, the composition yields
    a covariant Berezin-Toeplitz operator whose principal symbol is of
    the form $J'b\circ \kappa$, but this principal symbol is $1$ if
    $b=1$; this concludes the proof.
  \end{proof}

In general, it can be difficult to compute the action of a Fourier
Integral operator on a Lagrangian state at the level of principal
symbols and even more difficult to study the lower-order terms. There
is one notable exception: the action of
Berezin-Toeplitz operators on Lagrangian states where we can, and need to, understand the
subprincipal symbols.

  Recall that $T_k^{{\rm cov}}(f)$ is the operator with kernel
  \[
    (x,y)\mapsto \Pi_k(x,y)\widetilde{f}(x,y)
  \]
  where $\Pi_k$ is the Bergman kernel.

\begin{prop}\label{prop:subp_calc}
  Let $f$ be an analytic symbol on $M$ and let $I^{\Phi}_{k}(a)$ be
  a Lagrangian state on $M$. Recalling by Proposition
  \ref{prop:Lagrangians} that the Lagrangian
  $\mathcal{L}_{\Phi}\subset \widetilde{M}$ is
  transverse to the projection onto the first factor of
  $\widetilde{M}=M\times \overline{M}$, let $\iota:M\to \widetilde{M}$
  be such that $\iota(x)$ is the unique point in
  $\mathcal{L}_{\Phi}$ whose first component is $x$.  

  Then
  \[
    T_k^{\rm cov}(f)I^{\Phi}_{k}(a)=I^{\Phi}_{k}(b)+O(e^{-ck})
  \]
  where $b$ is an analytic symbol whose first two terms are
  \[ b_0 = (\iota^*\widetilde{f_0})a_0, \qquad b_1=(\iota^*\widetilde{f_1})a_0+(\iota^*\widetilde{f_0})a_1-i\widetilde{X_{f_0}}\cdot
         a_0 +B[\widetilde{\overline{\partial}f_0}]a_0 \]
  where $B$ is a linear order $1$ differential operator (see formulas
  \eqref{eq:subprincipal_1} and \eqref{eq:subprincipal_2} in the proof) and $X_{f_0}$ is the Hamiltonian vector field of $f_0$. In the
  specific case where the holomorphic extension $\widetilde{X_{f_0}}$ is tangent to
  $\mathcal{L}_{\Phi}$, one has
  \[
    b_1=(\iota^*\widetilde{f_1})a_0+(\iota^*\widetilde{f_0})a_1-i\widetilde{X_{f_0}}\cdot
    a_0 +\frac{a_0}{2}(-\iota^*\widetilde{\Delta
      f_0}+\iota^*[\overline{\partial}\log(s_0)\cdot\widetilde{\partial
    f_0}] -i{\rm div}_{\mathcal{L}_{\Phi}}(\widetilde{X_{f_0}})).
  \]
  Here, the divergence is considered with respect to the
  non-vanishing (complex-valued) $2d$-form $\iota^*\widetilde{\dd
    vol_M}$. Moreover $s_0$ is the principal symbol of the Bergman
  kernel of Proposition \ref{prop:Bergman_asymp}.
\end{prop}
\begin{proof}
  Away from a small neighbourhood of $\widetilde{\mathcal{L}}_0$ one has
  immediately $T_k^{\rm cov}(f)I^{\Phi}_{k}(a)=O(e^{-ck})$,
  therefore we restrict our attention to a subset $W$ of $V$ where
  $I^{\Phi}_{k}(a)=\Phi^{\otimes k}a+O(e^{-ck})$.
 
  In a Hermitian
  chart near a point of $\widetilde{L}_0$, we have
  \[
    T_k^{\rm cov}(f)I^{\Phi}_{k}(a)(x)=k^d\int_{{\rm Diag}(M\times \overline{M})}
    e^{ik\Psi(x,y',y'')}\tilde{s}(x,y'')\tilde{f}(x,y'')a(y')\dd
    y
  \]
  where
  \[
    \Psi:(x,y',y'')\mapsto
    i\frac{\psi(x)}{2}-i\widetilde{\psi}(x,y'')+i\widetilde{\psi}(y',y'')-i\phi(y'),
  \]
  $\phi$ is a holomorphic function as above, and
  \[
    \widetilde{s_0}(x,y'')=\frac{1}{(2\pi)^d}\det(\widetilde{\partial\overline{\partial}\psi}(x,y'')).
  \]
  The phase $\Psi$ has a unique critical point in the variables
  $(y',y'')$: by the last part of Proposition \ref{prop:Lagrangians},
  it is of the form $(x,(y'')^*(x))$ as given in \eqref{eq:wbarstar}. The
  Hessian of $\Psi$ is of the form
  \[
    H:={\rm Hess}_{y',y''}\Psi=
    \begin{pmatrix}
      M&i(\widetilde{\partial\overline{\partial}\psi})^T\\
      i\widetilde{\partial\overline{\partial}\psi}&0
    \end{pmatrix}
  \]
  where
  \[
    M_{jk}=i\partial_j\partial_k(\widetilde\psi-\phi)
  \]
  and in particular
  \[
    \frac{(2\pi)^d}{\sqrt{\det(H)}}=\frac{1}{\widetilde{s_0}(x,(y'')^*(x))}.
  \]

  Suppose first a model situation where 
  $\mathcal{L}(\Phi)=\mathcal{L}_0$ and $x\in \mathcal{L}_0$. In this case
  one has $(y'')^*(x)=x$, so the critical point is real, and
  $\partial\overline{\partial}\psi$ is positive near the critical
  point. In particular, one can perform analytic stationary phase
  without contour deformation.

  Therefore, for $h(\Phi)|_{\mathcal{L}_0}$ small in real-analytic topology and $x$ close
  to $\mathcal{L}_0$, the conditions of stationary phase are still met
  after a small contour deformation. Therefore one can apply the
  analytic stationary phase theorem and the output has WKB
  form. Let us compute the phase and the first two coefficients.

  First note that 
  $\Psi(x,x,(y'')^*(x))  =-\psi(x)/2+\phi(x)$ so that $e^{k\Psi(x,x,(y'')^*(x))}$ is exactly
  $\Phi^{\otimes k}(x)$.
  To compute the principal symbol, we follow the formula in Theorem
  7.7.5 of \cite{hormander_analysis_2003} (adapted to complex coordinates) and obtain
\[ b_0(x) = \frac{(2\pi)^d}{\sqrt{\det(H)}}\widetilde{s_0}(x,(y'')^*(x))\widetilde{f_0}(x,(y'')^*(x))a_0(x) = \widetilde{f_0}(x,(y'')^*(x))a_0(x). \]
  In fact, the last identity can be thought of as the definition of
  $s_0$: since $T_k^{\rm cov}(1)$ is the identity, one must have
  $b_0(x)=\widetilde{f_0}(x,(y'')^*(x))a_0(x)$. This coincides with
  the claim: by definition of $\iota$ one has
  \[
    (\iota^*\widetilde{f_0})(x)=\widetilde{f_0}(x,(y'')^*(x)).
    \]

  Before computing the subprincipal term, we ease up the notation. We
  consider a holomorphic chart $(z_1,\cdots,z_n)$ on $M$, from which
  we deduce a holomorphic chart on $M\times \overline{M}$ as follows:
  the first $n$ coordinates are $z_j':(y',y'')\mapsto z_j(y')$, for $1\leq
  j\leq n$, and the last $n$ coordinates are $z_j'':(y',y'')\mapsto
  \overline{z_j(y'')}$, for $1\leq j\leq n$. In this chart, given $u$
  analytic on $M$, the holomorphic extension of the holomorphic
  derivative $\partial_ju$ with respect to $z_j$ is the (holomorphic)
  derivative of $\widetilde{u}$ with respect to $z_j'$. Similarly, the
  holomorphic extension of the anti-holomorphic derivative
  $\overline{\partial}_ju$ with respect to $z_j$ is the (holomorphic)
  derivative of $\widetilde{u}$ with respect to $z_j''$. Keeping this
  in mind, in the rest of the proof we remove the $\sim$ signs for
  holomorphic extension of functions and we differentiate functions on
  $M\times \overline{M}$ in charts, denoting $\partial_j$ the
  differentiation with respect to $z_j'$ and $\overline{\partial_j}$ the
  differentiation with respect to $z_j''$. We also adopt the Einstein
  summation convention.

  The subprincipal term reads
  \[
    {f_1}a_0+a_1{f_0}+\frac{{s_1}}{{s_0}}{f_0}a_0+L_1({s_0}{f_0}a_0)\]
  where all terms are evaluated at $(x,\overline{w}^*(x))$ and $L_1$
  is a degree two differential operator which reads as follows:
  \[
    L_1({s_0}{f_0}a_0)=\frac{1}{is_0}\left[-\frac 12\langle
      H^{-1}D,D\rangle(s_0f_0a_0)+\frac
          18\langle
      H^{-1}D,D\rangle^2(Rs_0f_0a_0)-\frac{1}{96}(\langle
      H^{-1}D,D\rangle^3R^2)a_0f_0s_0\right].
  \]
  Note that only anti-holomorphic derivatives hit $f$ or $s_0$, and
  only holomorphic derivatives hit $a_0$.
  
Here
\[
  D=
  \begin{pmatrix}
    \partial\\ \overline{\partial}
  \end{pmatrix}
\]
and $R$ is $\Psi$ minus its order $2$ Taylor term at the critical
point.

If ${f_0}=1$ then $T_k^{\rm cov}(f_0)=\Pi_k$ and therefore
\[\frac{{s_1}}{{s_0}}a_0+L_1({s_0}a_0)=0.
\]
Thus, in general
\begin{equation}\label{eq:magic}
    \frac{{s_1}}{{s_0}}a_0=\frac{i}{s_0}\left[-\frac 12\langle
      H^{-1}D,D\rangle(s_0a_0)+\frac
          18\langle
      H^{-1}D,D\rangle^2(Rs_0a_0)-\frac{1}{96}(\langle
      H^{-1}D,D\rangle^3R)a_0s_0\right]
  \end{equation}
  that is, $s_1$ exactly
compensates for the terms in $L_1$ where no derivative has hit
${f_0}$.

Let us first study the first term. One has first
\[
  H^{-1}=
  \begin{pmatrix}
    0&-i(\partial\overline{\partial}{\psi})^{-1}\\
    -i[(\partial\overline{\partial}\psi)^{-1}]^T&A
  \end{pmatrix}
\]
where
\[
  A_{jk}=-i(\partial\overline{\partial}\psi)^{-1}_{lj}\partial_l\partial_m(\phi-\psi)(\partial
  \overline{\partial}\psi)^{-1}_{mk};\]
it is good to keep in mind that $\partial
\overline{\partial}\psi$ is the metric tensor.

Consequently,
\[
  \langle
  H^{-1}D,D\rangle=A_{jk}\overline{\partial}_j\overline{\partial}_k-2i(\partial
  \overline{\partial}\psi)^{-1}_{jk}\partial_j\overline{\partial}_k
\]
and we can compute the first term in $L_1$:
\begin{align*}
  -\frac{1}{2is_0}\langle
  H^{-1}D,D\rangle(s_0f_0a_0)=
  &\frac{ia_0}{2s_0}A_{jk}\overline{\partial}_j\overline{\partial}_k(s_0f_0)+(\partial\overline{\partial}\psi)_{jk}^{-1}\partial_ja_0\overline{\partial}_kf_0\\
  =&\frac{ia_0}{2}A_{jk}\overline{\partial}_j\overline{\partial}_kf_0
  +\partial a_0\cdot
    \overline{\partial}f_0+ia_0A_{jk}\overline{\partial}_j\log(s_0)\overline{\partial}_kf_0
  +B_1f_0
\end{align*}
where $B_1$ is a multiplication operator acting on $f_0$ whose
contribution is irrelevant by \eqref{eq:magic}.

Let us turn our attention to the second term: one has
\[
  \langle H^{-1}D,D\rangle^2=A_{jk}A_{lm}\overline{\partial}_j\overline{\partial}_k\overline{\partial}_l\overline{\partial}_m-4(\partial\overline{\partial}\psi)^{-1}_{jk}(\partial\overline{\partial}\psi)^{-1}_{lm}\partial_j\overline{\partial}_k\partial_l\overline{\partial}_m-4iA_{jk}(\partial\overline{\partial}\psi)^{-1}_{lm}\overline{\partial}_j\overline{\partial}_k\partial_l\overline{\partial}_m.
\]
Among these four derivatives, at least three must hit $R$ (since $R$
vanishes at order $3$ at the critical point) and at least one must
hit $f_0$ (the rest being compensated by $s_1$). In addition, since $\Psi(x,x,\overline{w})$ does not depend on
$\overline{w}$, one has, at the critical point,
$\overline{\partial}\,\overline{\partial}\,\overline{\partial}R=0$, so the
first term in the expansion of $\langle H^{-1}D,D\rangle^2$
above is completely compensated by $s_1$. Thus
\begin{align*}
  \frac{1}{8is_0}\langle
H^{-1}D,D\rangle^2(Rf_0s_0a_0)=&\frac{a_0}{8i}\langle
                                H^{-1}D,D\rangle^2(Rf_0)+B_2f_0\\
                              =&-\frac{a_0}{2i}(\partial\overline{\partial}\psi)^{-1}_{jk}(\partial\overline{\partial}\psi)^{-1}_{lm}\partial_j\overline{\partial}_k\partial_l\overline{\partial}_m(R{f_0})-\frac{a_0}{2}A_{jk}(\partial\overline{\partial}\psi)^{-1}_{lm}\overline{\partial}_j\overline{\partial}_k\partial_l\overline{\partial}_m(R{f_0})+B_2f_0\\
  =&-a_0(\partial\overline{\partial}\psi)^{-1}_{jk}(\partial\overline{\partial}\psi)^{-1}_{lm}\partial_j\overline{\partial}_k\partial_l{\psi}\overline{\partial}_m{f_0}\\
  &-ia_0A_{jk}(\partial\overline{\partial}\psi)^{-1}_{lm}\overline{\partial}_j\partial_l\overline{\partial}_m{\psi}\overline{\partial}_k{f_0}
  -i\frac{a_0}{2}A_{jk}(\partial\overline{\partial}\psi)^{-1}_{lm}\overline{\partial}_j\overline{\partial}_k\partial_l{\psi}\overline{\partial}_m{f_0}\\
                               &+B_3f_0\\
  =&-a_0\partial \log (s_0)\cdot \overline{\partial}f_0\\
  &-ia_0A_{jk}\overline{\partial}_j\log(s_0)\overline{\partial}_k{f_0}
  -i\frac{a_0}{2}A_{jk}(\partial\overline{\partial}\psi)^{-1}_{lm}\overline{\partial}_j\overline{\partial}_k\partial_l{\psi}\overline{\partial}_m{f_0}\\
                               &+B_3f_0.
\end{align*}
Here $B_2$ and $B_3$ are multiplication operators whose values are
irrelevant.

All in all, one has
\begin{equation}\label{eq:subprincipal_1}
  b_1=\iota^*\left(f_1a_0+a_1f_0+\overline{\partial}f_0\cdot
  \partial
  a_0\right. \left.+a_0\left[-\partial \log(s_0)\cdot \overline{\partial}f_0+\frac{i}{2}A_{jk}\overline{\partial}_j\overline{\partial}_kf_0-\frac{i}{2}A_{jk}(\partial\overline{\partial}\psi)_{lm}^{-1}\overline{\partial}_j\overline{\partial}_k\partial_l\psi\overline{\partial}_mf_0\right]\right).
\end{equation}

It remains to give a suitable geometric interpretation of the term
under brackets, at least in the case where $f_0$ is constant on $\mathcal{L}_{\Phi}$. We
begin by establishing some fundamental identities. Let
us first recall that, in local coordinates on a Kähler manifolds,
the holomorphic Laplacian applied to a function $u$
reads
\begin{equation}
  \label{eq:Laplace_Kahler}
  \Delta u = (\partial \overline{\partial}\psi)^{-1}_{jk}\partial_j\overline{\partial}_ku.
\end{equation}
Now, we go back to
\eqref{eq:wbarstar} and write
\[
  \mathcal{L}_{\Phi}=\{(y',y''_*(y'))\in M\times \overline{M},\text{ where
  }y'\in M\}.
\]
Differentiating \eqref{eq:wbarstar} with respect to $y'$, we now obtain
\[
  \partial_j\partial_k\phi=\partial_j\partial_k\psi+\partial_j\overline{\partial}_l\psi\partial_k(y''_*)_l
\]
and therefore
\[
  \partial_k(y''_*)_l=(\partial\overline\partial\psi)^{-1}_{jl}\partial_j\partial_k(\phi-\psi);
\]
in particular,
\[
  A_{jl}=-i\partial_k(y''_*)_j(\partial\overline{\partial}\psi)^{-1}_{kl}.
\]
Now, let $u:M\to \C$ be real-analytic (read in a chart). Since $\iota^*u$ and $(y'')^*$ are holomorphic, one has, for every $u$
real-analytic
\[
  \partial_j\iota^*u=\iota^*[\partial_ju+\partial_j(y''_*)_l\overline{\partial}_lu].
\]
In particular, replacing $u$ par $\overline{\partial}_{k}u$,
\[
  \partial_j\iota^*(\overline{\partial}_k
  u)=\iota^*[\partial_j\overline{\partial}_ku+\partial_j(y''_*)_{l}\overline{\partial}_k\overline{\partial}_{l}u].
\]
Plugging in the formula for $A_{jl}$ and \eqref{eq:Laplace_Kahler}
we obtain
\begin{align*}
  \iota^*[-i(\partial\overline{\partial}\psi)_{jk}^{-1}]\partial_j\iota^*(\overline{\partial}_ku)&=\iota^*[-i(\partial\overline{\partial}\psi)_{jk}^{-1}\partial_j\overline{\partial}_ku-i(\partial\overline{\partial}\psi)_{jk}^{-1}\partial_j(y''_*)_{l}\overline{\partial}_k\overline{\partial}_{l}u]\\
  &=\iota^*[-i\Delta u +
    A_{kl}\overline{\partial}_k\overline{\partial}_lu]
\end{align*}
where we used the symmetry of $A_{kl}$. 
Replacing $u$ with $f_0$ allows us to rewrite the second term inside
the brackets of
\eqref{eq:subprincipal_1} into
\[
  \frac{i}{2}A_{jk}\overline{\partial}_j\overline{\partial}_kf_0=-\frac
  12 \Delta
  f_0 + \frac
  12 \iota^*[(\partial\overline{\partial}\psi)_{jk}^{-1}]\partial_j\iota^*(\overline{\partial}_kf_0),
\]
and replacing $u$ with $\partial_l\psi$, we obtain
\begin{multline*}
  \iota^*\left[-\frac{i}{2}A_{jk}(\partial\overline{\partial}\psi)^{-1}_{lm}\overline{\partial}_j\overline{\partial}_k\partial_l\psi\overline{\partial}_mf_0\right]\\
  =\frac 12
  \iota^*[(\partial\overline{\partial}\psi)_{jk}^{-1}(\partial\overline{\partial}\psi)^{-1}_{lm}\partial_j\overline{\partial}_k\partial_l\psi\overline{\partial}_mf_0]-\frac
  12
  \iota^*[(\partial\overline{\partial}\psi)_{jk}^{-1}(\partial\overline{\partial}\psi)_{lm}^{-1}\overline{\partial}_mf_0]\partial_j\iota^*(\overline{\partial}_k\partial_l\psi)\\
  =\frac 12 \partial \log s_0\cdot \overline{\partial}f_0-\frac
  12
  \iota^*[(\partial\overline{\partial}\psi)_{jk}^{-1}(\partial\overline{\partial}\psi)_{lm}^{-1}\overline{\partial}_mf_0]\partial_j\iota^*(\overline{\partial}_k\partial_l\psi).
\end{multline*}
At the end of the day, the quantity under brackets in
\eqref{eq:subprincipal_1} is
\begin{equation}\label{eq:subprincipal_2}
  -\frac12 \iota^*[\Delta f_0 +\partial \log s_0\cdot \overline{\partial} f_0] + \frac12
  \partial_j\iota^*[(\partial\overline{\partial}\psi)^{-1}_{jk}\overline{\partial}_kf_0].
\end{equation}
Now the symplectic gradient of $f_0$ is
\[
  X=-i\left[(\partial\overline{\partial}\psi)_{jk}^{-1}\overline{\partial}_kf_0
    \frac{\dd}{\dd z_j}-(\partial
    \overline{\partial}\psi)_{kj}^{-1}\partial_kf_0\frac{\dd}{\dd \overline{z_j}}\right]
\]
and thus
\[
  \widetilde{X}=-i\left[(\partial\overline{\partial}\psi)_{jk}^{-1}\overline{\partial}_kf_0
    \frac{\dd}{\dd z_j}-(\partial
    \overline{\partial}\psi)_{kj}^{-1}\partial_kf_0\frac{\dd}{\dd
      \overline{w_j}}\right]
\]
We assume that $\widetilde{X}$ is tangent to $\mathcal{L}_{\Phi}=\{(x,y''(x))\}$, which means that
\[
  \widetilde{X}\in T_{(x,y''(x))}\mathcal{L}_{\Phi}=\left\{ v_j\frac{\dd}{\dd z_j}+
  v_k\partial_ky''_j\frac{\dd}{\dd \overline{w_j}},(v_1,\cdots,v_n)\in
\C^n\right\}.
\]
In particular, on $\mathcal{L}_{\Phi}$,
\[
  -(\partial\overline{\partial}\psi)^{-1}_{kj}\partial_kf_0=(\partial\overline{\partial}\psi)^{-1}_{kl}\overline{\partial}_lf_0\partial_ky''_j.
\]
Under these hypotheses, let us compute the divergence on $\mathcal{L}_{\Phi}$ of
the vector field $\widetilde{X}$. In the chart on $\mathcal{L}_{\Phi}$ given by
the first coordinate, the coordinates of the vector field $\widetilde{X}$ are
precisely
\[
  i\iota^*((\partial\overline{\partial}\psi)_{jk}^{-1}\overline{\partial}_kf_0)_{1\leq
    j\leq n}.
\]
The $2d$-form with respect to which we consider the divergence is, in
the
chart, \[\det(\partial\overline{\partial}\psi)=(2\pi)^d\iota^*\widetilde{s_0};\]
in particular it is non-vanishing.

Since $\widetilde{X}$ and $\iota^*\widetilde{s_0}$ are holomorphic,
the antiholomorphic divergence vanishes, and it remains precisely
\begin{align*}
  {\rm
    div}_{\mathcal{L}_{\Phi}}(\widetilde{X})&=-i\partial_j\log(\iota^*s_0)(\partial\overline{\partial}\psi)^{-1}_{jk}\overline{\partial}_kf_0
  -i\partial_j\iota^*[(\partial\overline{\partial}\psi)^{-1}_{jk}\overline{\partial}_kf_0]\\
  &=-i\iota^*(\partial \log(s_0)\cdot
    \overline{\partial}f_0)-i\partial_j\iota^*[(\partial\overline{\partial}\psi)^{-1}_{jk}\overline{\partial}_kf_0]-i\iota^*\partial_j(y''_*)_l\overline{\partial}_l\log(s_0)(\partial\overline{\partial}\psi)^{-1}_{jk}\overline{\partial}_kf_0\\
  &=-i\iota^*(\partial \log(s_0)\cdot
    \overline{\partial}f_0)-i\partial_j\iota^*[(\partial\overline{\partial}\psi)^{-1}_{jk}\overline{\partial}_kf_0]+i\iota^*\overline{\partial}\log(s_0)\cdot\partial
    f_0
\end{align*}

At the end of the day, the subprincipal term is
\[
  -\frac 12 \iota^*\Delta f +\frac 12\overline{\partial}\log(s_0)\cdot\partial
    f_0 -
  \frac{i}{2}{\rm div}_{\mathcal{L}_{\Phi}}(\widetilde{X}).
\]
\end{proof}
\begin{rem}\label{rem:subprincipal_Lagrangian}The result of
  Proposition \ref{prop:subp_calc} is a generalisation to the complex
  setting of previously established formulas, for instance Theorem 5.4 in
  \cite{charles_symbolic_2006}. Indeed, in the setting of \cite{charles_symbolic_2006},
  \begin{itemize}
  \item the normalised symbol is obtained from the covariant symbol via
    \[f_0\rightsquigarrow f_0-\frac{\hbar^{-1}}{2}\Delta f_0;\]
  \item the auxiliary bundle $L_1$ is $\delta^{-1}$; in particular,
    the covariant derivative of the trivialising section $t$ of $\delta^{-1}$ with respect to $X_{f_0}$ reads
    \[
    \nabla^{\iota^*(\delta^{-1})}_{X_{f_0}}t=\frac{i}{2}\iota^*(\overline{\partial} \log(s_0)\cdot \partial f_0 - \partial
\log(s_0)\cdot \overline{\partial}f_0)t;
\]
\item in general, if $Y$ is a vector field and $g$ is a Riemannian
  metric, then \[\mathcal{L}_Y({\rm dVol}(g))={\rm div}_g(Y){\rm
      dVol}(g),\]and commutation with $\iota^*$ brings out a
  supplementary factor $\frac{i}{2}\partial\log(s_0)\cdot
  \overline{\partial}f_0$ as in the end of the proof.

  \end{itemize}
  
\end{rem}

\begin{prop}\label{prop:propag_BT}Let $p:\C\times M\to \C$ be
  real-analytic with holomorphic dependence on the first
  factor. Denote $p_z = p(z,\cdot)$ and suppose that $p(0,\cdot)$ is real-valued.
  
  For every $T>0$ there exists $\varepsilon$ such that the time
  propagation $\exp(-itkT_k^{\rm cov}(p_z))$ of the analytic
  (non-self-adjoint) Berezin--Toeplitz operator $T_k^{\rm cov}(p_z)$
  is, for times $t \in (-T,T)$ and $|z|<\varepsilon$, a Fourier
  Integral operator, with Lagrangian close to the real
  Lagrangian \[\{(\varphi_{p_0}^t(x),x),x\in M\}\]where
  $\varphi_{p_z}^t:\widetilde{M}\to \widetilde{M}$ is the Hamiltonian
  flow of $\tilde{p}_z$.

  In particular,
  \[
    e^{itkT^{\rm cov}_k(p_z)}T^{\rm cov}_k(a)e^{-itkT^{\rm
        cov}_k(p_z)}=T^{\rm cov}_k(a(t))+O(e^{-ck})
  \]
  where the principal symbol evolves as
  \[
    a_0(t)=a_0\circ \varphi_{p_z}^t.
  \]
\end{prop}
\begin{proof}
The second part of the claim is a direct consequence of the first part
and of the principal symbol calculus of Proposition
\ref{prop:Compo_multi_FIO} -- here $s=1$, because if
$a=1$ one has to find the identity on the right-hand side.

Note first that for every $T > 0$ there exists $\varepsilon > 0$ such that for every $|z| < \varepsilon$ and for every $x$ in some neighborhood of $M$ in $\widetilde{M}$, $\varphi^t_{p_z}(x)$ makes sense for all $t \in (-T,T)$. Accordingly, we define the Lagrangian $\mathcal{L}(t) \subset \widetilde{M} \times \widetilde{\overline{M}}$ as the graph of $\varphi^t_{p_z}$.  

Let $\Phi_0(t)$ be a phase with Lagrangian $\mathcal{L}(t)$ and let $b(t)$ be any symbol -- real-analytic with respect to $t$ -- defined near the graph of $\varphi_{p_0}^t$. Define \[
  U_0(t)=I_k^{\Phi_0(t)}(b(t)).
\]
Domains here are irrelevant: $U_0$ is a global Fourier Integral
operator.

According to Proposition \ref{prop:subp_calc}, the principal symbol
of $T_k(p_z)U_0(t)$ is $p_zb(t)$. On the
other hand, the principal symbol of $ik^{-1}\partial_tU_0(t)$ is
$i\partial_t\Phi_0 b(t)$, but because
$\widetilde{\nabla}\widetilde{\Phi_0}=0$ precisely on $\mathcal{L}(t)$,
there holds
\[
  i\partial_t\Phi_0|_{\mathcal{L}}=p_z +C(t)
\]
where $C(t)$ is a constant.
Replacing now $\Phi_0$ with
\[
  \Phi_1(t)=\Phi_0(t)-\int_0^tC(s)\dd s,
\]
and letting $U_1(t)=I_k^{\Phi_1(t)}(b(t))$, one has now that the principal symbols of the Fourier Integral
Operators $T_k(p_z)U_1(t)$ and $ik^{-1}\partial_tU_1(t)$
coincide. Thus
\[
  U_1(t)^{-1}(ikT_k(p_z)-\tfrac{\partial}{\partial
    t})U_1(t)=T_k^{\rm cov}(r(t))+O(e^{-ck})
\]
where $r(t)$ is a classical analytic symbol. Letting now $a(t)$ be a
classical analytic symbol solving \[\partial_ta(t)=a(t)\star_{\rm cov} r(t)\]
with $a(0)=1$ 
(this equation satisfies the hypotheses of the Picard-Lindelöf theorem
in some analytic symbol class), one finds that
\[
  U_1(t)=e^{iktT_k(p_z)}T_k^{\rm cov}(a(t))+O(e^{-ck}).
\]
This implies that $U(t) = U_1(t) T_k^{\rm cov}(a(t))^{-1}$ is a Fourier Integral Operator and we conclude by applying Proposition
\ref{prop:Compo_multi_FIO}.
\end{proof}

\begin{rem}\label{rem:subprincipal_propag}
  Using the subprincipal symbol calculus of Proposition
  \ref{prop:subp_calc}, in principle it should be possible to compute the principal symbol of the
  propagator, as in \cite{borthwick_semiclassical_1998,zelditch_pointwise_2018,ioos_geometric_2021,charles_quantum_2020}. Presumably, one would obtain a meaningful generalisation to
  $\widetilde{M}$ of the geometric constructions in the aforementioned works.
\end{rem}

\section{Local model}
\label{sec:local}

In this section we study the quasimodes of $T_k^{\rm cov}(p)$ under an
hypothesis of small perturbation of a real symbol, near a regular
piece of trajectory. More precisely, we will work under the following
hypothesis.

\begin{hyp}\label{hyp:local}~
  \begin{enumerate}
  \item $(M,J,\omega)$ is a real-analytic, compact, quantizable Kähler manifold of complex dimension 1.
  \item $p:\C\times M\to \C$ is a real-analytic, complex-valued
    Hamiltonian with holomorphic dependence on the first
    coordinate. We write
    \[
      p_z=p(z,\cdot).\]
  \item $p_0$ is real-valued.
  \item $\mathcal{C}\subset M$ is a regular, contractible piece of
    level set of $p_0$.
  \end{enumerate}
\end{hyp}

We first give a normal form for $p_z$ near $\mathcal{C}$, conjugating
it to $T_k^{\rm cov}(\xi)$ acting on the Bargmann space $\mathcal{B}_k$. In
the real-valued case, this ``quantum flowbox'' theorem is well-known and already
mentioned, in the pseudodifferential case, in
\cite{sjostrand_singularites_1982}; in the $C^{\infty}$ category for
Berezin--Toeplitz quantization, see \cite{charles_quasimodes_2003}. Then, we
use this normal form to study the quasimodes; in particular, we prove
that exponentially accurate quasimodes always exist and are always
close to Lagrangian states.

\subsection{Normal forms}
\label{sec:normal-forms-1}

\begin{prop}\label{prop:flowbox_classical_real}Assume Hypothesis
  \ref{hyp:local} holds. Let $x_0\in \mathcal{C}$. There exist a neighbourhood $\mathcal{Z}$ of $0$ and a holomorphic symplectic change of variables $\kappa_z$ from a
  neighbourhood of $\mathcal{C}$ in $\widetilde{M}$ to a neighbourhood of
  $[0,T]_x\times \{0\}_\xi$ in $(\C^2,\dd \xi \wedge \dd x)$, with
  holomorphic dependence on $z\in \mathcal{Z}$, such
  that \[\widetilde{p}_z-\widetilde{p}_z(x_0)=\xi\circ \kappa.\]
\end{prop}
\begin{proof}
  A neighbourhood of $\gamma$ in $\widetilde{M}$ is foliated by the
  level sets of $\widetilde{p}$, which are regular holomorphic
  curves. Let $\Lambda_0$ be an open piece of holomorphic Lagrangian
  transverse to $X_{\widetilde{p}}$ and containing $\gamma(0)$. A
  smaller neighbourhood $V$ of $\gamma$ consists of the disjoint union of
  the images of elements of $\Lambda_0$ by the flow of
  $X_{\widetilde{p}}$ for times in a complex neighbourhood $U_x$ of
  $[0,T]$.

  Let $\zeta:\Lambda_0\to \C$ be an arbitrary (holomorphic) parametrisation of
  $\Lambda_0$; extend this function to $V$ by transporting it by the
  flow of
  $X_{\widetilde{p}}$. Let also $x:V\to U_x$ denote the
  (complex-valued) time needed to connect $x$ to a point of
  $\Lambda_0$. Then $(x,\zeta)$ form holomorphic coordinates on $V$;
  since the flow of $X_{\widetilde{p}}$ preserves the original
  holomorphic symplectic form, the pulled-back symplectic form is
  invariant under $x$-translations, and is therefore of the form
  $f(\zeta)\dd \zeta\wedge \dd x$ where $f$ is holomorphic and
  non-vanishing.

  Letting now $\xi=F(\zeta)$ where $F$ is an anti-derivative of $f$,
  in the variables $(x,\xi)$, the symplectic form reads $\dd \xi\wedge
  \dd x$, and in these coordinates,
  $X_{\widetilde{p}}=\frac{\partial}{\partial x}$. Therefore, in these
  coordinates $\widetilde{p}=\xi+C$ for some $C\in \mathbb{C}$. This
  concludes the proof.

  For the parameter-dependent case, it suffices to remark that, once $\Lambda_0$ and $\zeta|_{\Lambda_0}$ are fixed, in the rest of the proof,
  all constructions depend holomorphically on $p$.
\end{proof}

Applying Proposition \ref{prop:invert_FIO}, the conjugation of
$T_k^{\rm cov}(p)$ with a Fourier Integral operator whose Lagrangian
is the graph of $\kappa$ and with arbitrary elliptic principal
symbol is of the form $T_k^{\rm cov}(\xi+k^{-1}q)$, microlocally near
$0$, for some analytic symbol $q$. We now get rid of this subprincipal symbol.

\begin{prop}\label{prop:flowbox_correction}
  Let $x_-<x_+,\xi_-<\xi_+$ be real numbers. Let $q$ be a real-analytic classical
  symbol in a neighbourhood of $[x_-,x_+]\times [\xi_-,\xi_+]$. Then there exists
  a real-analytic classical symbol $a$, with elliptic principal symbol in a neighbourhood of
  $[x_-,x_+]\times [\xi_-,\xi_+]$ such that, microlocally near
  $[x_-,x_+]\times [\xi_-,\xi_+]$, one has
  \[
    T_k^{\rm cov}(\xi+k^{-1} q)T_k^{\rm cov}(a)=T^{\rm cov}_k(a)T^{\rm
    cov}_k(\xi)+O(e^{-ck}).
  \]
\end{prop}
\begin{proof}
  We proceed by deformation. We let $\star_{{\rm cov}}$ denote the formal symbol
  product for covariant Berezin--Toeplitz quantization on $\C$. We want to find $a(t)$, with
  $a(0)=1$, such that
  \[
    (\xi+t k^{-1} q)\star_{{\rm cov}} a = a \star_{{\rm cov}} \xi.
  \]
  With $b=a^{-1}\star_{{\rm cov}}\frac{\partial a}{\partial t} $, we obtain
  \[
    [\xi,b]+k^{-1} a^{-1}\star_{{\rm cov}} q \star_{{\rm cov}} a = 0
  \]
  and again, denoting $p=a^{-1}\star_{{\rm cov}} q\star_{{\rm cov}} a$,
  \begin{equation}\label{eq:evol_subp_flowbox}
    \frac{\dd p}{\dd t}=[b,p].
  \end{equation}
  The solution of the cohomological equation takes the
  following form in terms of Taylor coefficients at $0$: denoting
  \[
    p=\sum_{\ell=0}^{\varepsilon
      k}p_{\ell,i,j}\frac{x^i\xi^jk^{-\ell}}{i!j!}+O(e^{-ck}),\qquad
    \qquad b=\sum_{\ell=0}^{\varepsilon
      k}b_{\ell,i,j}\frac{x^i\xi^jk^{-\ell}}{i!j!}+O(e^{-ck})
  \]
  one must have
  \[
    b_{\ell,i,j}=\begin{cases}\frac{p_{\ell,i,(j-1)}}{j}&\text{ if }j\neq 0\\
      0&\text{ otherwise.}
    \end{cases}
  \]
  In particular, for every $T>0$, following Definition
  \ref{def:Boutet-Kree}, one has $\|b\|_{BK(T)}\leq \|p\|_{BK(T)}$. In
  particular, by Proposition \ref{prop:BK}, one can apply the
  Picard-Lindelöf theorem to the differential equation
  \eqref{eq:evol_subp_flowbox} and obtain that, for all times, $p$ and
  $b$ are well-defined analytic symbols.

  We then recover $a$ by applying the Picard-Lindelöf theorem
  to
  \[
    \frac{\partial a}{\partial t}=b\star_{{\rm cov}} a.
  \]
\end{proof}

By putting together Propositions \ref{prop:flowbox_classical_real},
\ref{prop:conjug_Toep_FIO}, and \ref{prop:flowbox_correction} while
keeping track of the parameter dependence, we
arrive at the following conclusion.

\begin{prop}\label{prop:flowbox_assemblage}Assume Hypothesis
  \ref{hyp:local} holds. There
  exist a small neighbourhood $U$ of $\mathcal{C}$, a set $V$ of the form $(x_-,x_+)\times
  (\xi_-,\xi_+)$ for some $x_-<x_+,\xi_-<\xi_+\in \R$, a neighbourhood $Z$ of $0$
  in $\C$, and for all $z\in Z$, Fourier integral operators
  \[ \mathfrak{U}_z:H^0(M,L^{\otimes k})\to \mathcal{B}_k, \qquad  \mathfrak{V}_z:\mathcal{B}_k\to H^0(M,L^{\otimes k})  \]
   with holomorphic dependence on $z$, which are microlocal inverses of
   each other on the domains $U$ and $V$, and such that, uniformly for $z\in
  \mathcal{Z}$, for every
  $u\in H^0(M,L^{\otimes k})$, 
  \[
    \mathfrak{U}_z T_k^{\rm cov}(p_z)u=T_k^{\rm
      cov}(\xi)\mathfrak{U}_z u+O(e^{-ck}\|u\|_{L^2}+e^{c^{-1}|z|k}\|u\|_{L^2(M\setminus
      U)})
  \]
  and for every $v\in \mathcal{B}_k$,
  \[
    \mathfrak{V}_z T_k^{\rm cov}(\xi)v=T_k^{\rm cov}(p_z)\mathfrak{V}_z v+O(e^{-ck}\|v\|_{L^2}+e^{c^{-1}|z|k}\|v\|_{L^2(\C\setminus
      V)}).
  \]
\end{prop}

\subsection{Microlocal solutions}
\label{sec:micr-sheaf-solut-1}

Inspired by Proposition \ref{prop:flowbox_assemblage} we begin with a
description of the microlocal quasimodes for the model operator.

\begin{prop}\label{prop:microlocal_sheaf_model}
  Let $x_-<x_+\in \R$ and $\xi_-<0<\xi_+\in \R$; let $U=(x_-,x_+)\times
  (\xi_-,\xi_+)$. Let $V\Subset U$. For every $c>0$, the solutions of
  \begin{equation}\label{eq:quasimode_normal}
    u\in \mathcal{B}_k,\;\|T_k^{\rm cov}(\xi)u\|_{L^2(U)}=O(e^{-ck}\|u\|_{L^2})
  \end{equation}
  are, uniformly on $V$, of the form
  \[
    (x,\xi)\mapsto u(\tfrac{x_-+x_+}{2},0)\exp(-k\tfrac{\xi^2}{2})+O(e^{-c'k}\|u\|_{L^2})
  \]
  for every $c'<c$.
\end{prop}
\begin{proof}  
  Without loss of generality, $\|u\|_{L^2(\C)}=1$. Since $u\in
  \mathcal{B}_k$, it satisfies
  \begin{equation}\label{eq:Bargmann_space}
    \frac{\partial}{\partial x}u+i\frac{\partial}{\partial
      \xi}u=-ik\xi u;
  \end{equation}
  and from the hypothesis, \[k^{-1}\frac{\partial}{\partial
    x}u=-iT_k^{\rm cov}(\xi)u\]is exponentially small (in $L^2$ norm) on
$U$. By holomorphy, we obtain directly that
\begin{equation}\label{eq:dxu}
  \|k^{-1}\partial_xu\|_{L^{\infty}(V)}<Ck^de^{-ck}<C_1e^{-c'k}.
\end{equation}
for every $c'<c$.

  Applying the Duhamel formula on \eqref{eq:dxu} we obtain,
  uniformly for $x\in
  (x_-,x_+)$,
  \[
    u(x,0)=u(\tfrac{x_-+x_+}{2},0)+O(e^{-c'k})
  \]
  and then, applying the Duhamel formula a second time in the variable
  $\xi$, we obtain the desired claim.
\end{proof}
Putting together Propositions \ref{prop:microlocal_sheaf_model},
\ref{prop:subp_calc}, and \ref{prop:flowbox_assemblage}, we obtain the following
two-way description for microlocal solutions of the eigenvalue equation near
regular pieces of trajectories: there always exist quasimodes in the
WKB form, and quasimodes are necessarily of this form. Moreover we
have some geometric information on the Lagrangian and principal symbol.

\begin{prop}\label{prop:exists_microlocal_solution_local}
  Assume Hypothesis \ref{hyp:local} holds. There
  exist a small neighbourhood $U$ of $\mathcal{C}$ in $M$, a small
  neighbourhood $\mathcal{Z}$ of $0$ in $\C$, a small neighbourhood
  $\mathcal{E}$ of $p_0(\mathcal{C})$ in $\C$, a constant $c_0>0$,
  such that for every $(z,\lambda)\in \mathcal{Z}\times \mathcal{E}$,
  there exists a Lagrangian state $u_k$, with Lagrangian
  $\{\widetilde{p_z}=\lambda\}$, with holomorphic dependence on $z$ and $\lambda$,
   such that, when $(z,\lambda)=(0,p_0(\mathcal{C}))$, one has $\|u_k\|_{L^2}=1$, and satisfying
  \[
    \|T_k^{\rm cov}(p_z-\lambda)u_k\|_{L^2(U)}=O(e^{-c_0k}\|u_k\|_{L^2(U)}).
  \]
\end{prop}
\begin{proof}
   Let $\mathfrak{U}_z$ and $\mathfrak{V}_z$ be Fourier Integral Operators satisfying the
  conclusion of Proposition
  \ref{prop:flowbox_assemblage}. Recall that
  $\mathfrak{U}_z$ and $\mathfrak{V}_z$ depend holomorphically on $z$. By
  Proposition \ref{prop:flowbox_assemblage}, one has, for every $v\in
  \mathcal{B}_k$,
  \[
    T_k^{\rm cov}(p_z-\lambda)\mathfrak{V}_z v = \mathfrak{V}_z T_k^{\rm
      cov}(\xi-\lambda+p_0(\mathcal{C}))v+O(e^{-ck}\|v\|_{L^2}+e^{c^{-1}|z|k}\|v\|_{L^2(\C\setminus
      V_z)}).
  \]
    Letting $W_z\Subset V_z$, the proof
  consists in applying $\mathfrak{V}_z$ to the sequence
  \[
    v_k=\Pi_k\left(\1_{W_z}\exp(-k\tfrac{\xi^2}{2})\exp(ik(p_0(\mathcal{C})-\lambda)(x+i\xi))\right).
  \]
  $v_k$ is a quasimode for $T_k^{\rm
    cov}(\xi-\lambda+p_0(\mathcal{C}))$ on a neighbourhood of
  $\kappa_0(\mathcal{C})$ where $\kappa_0$ is the Hamiltonian
  diffeomorphism associated with $\mathfrak{U}_z$. Moreover, $v_k$ is a Lagrangian state with Lagrangian
  $\{\widetilde{\xi}=\lambda-p_0(\mathcal{C})\}$; it
  depends holomorphically on $\lambda$. By Proposition
  \ref{prop:Compo_multi_FIO}, $u_k:=\mathfrak{V}_zv_k$ is a Lagrangian
  state with Lagrangian $\{\widetilde{p_z}=\lambda\}$ and it satisfies
  all desired requirements.
\end{proof}

\begin{prop}\label{prop:local_WKB}
  Assume Hypothesis \ref{hyp:local} holds. Given
  $c>0$ and $U_1\Subset U$, there exist a small neighbourhood $\mathcal{E}$ of
  $p_0(\mathcal{C})$ in $\C$, a small neighbourhood $\mathcal{Z}$ of
  $0$ in $\C$, and
  $c'>0$, such that, uniformly for $(\lambda,z)\in \mathcal{E}\times \mathcal{Z}$, the solutions of
  \begin{equation}\label{eq:quasimode_local_general}
    \|T_k^{\rm
      cov}(p(z,\cdot)-\lambda)u\|_{L^2(U)}=O(e^{-ck}\|u\|_{L^2})
  \end{equation}
  are $O(e^{-c'k})$-close, on $U_1$, to Lagrangian states with
  Lagrangian $\Lambda=\{\widetilde{p(z,\cdot)}=\lambda\}$ (which is close to the
  real Lagrangian $\mathcal{C}$). Once $z$ and $\lambda$ are fixed, the
  total symbol of such a
  Lagrangian state is unique up to a multiplicative constant (possibly
  depending on $k$). In
  particular, the principal symbol of the Lagrangian
  states satisfies the following transport equation on
  $\Lambda$:
  \begin{equation}\label{eq:princ_symb_quasimode}
    X_{\widetilde{p}(z,\cdot)}a_0=-i\frac{a_0}{2}\left(
      \iota^*\Delta\widetilde{p}(z,\cdot)-\iota^*\overline{\partial}\log(s_0)\cdot\partial
    f_0 -i{\rm
        div}_{\Lambda}(X_{\widetilde{p}(z,\cdot)})\right).
  \end{equation}
These Lagrangian states depend holomorphically on $z$ and $\lambda$.
\end{prop}
\begin{proof}
  Let $\mathfrak{U}_z$ and $\mathfrak{V}_z$ be Fourier Integral Operators satisfying the
  conclusion of Proposition
  \ref{prop:flowbox_assemblage}. By Proposition
  \ref{prop:propag_support}, if $|z|$ and $|\lambda-p_0(\mathcal{C})|$
  are small enough (depending on
  $c$), if $u$ satisfies
  \eqref{eq:quasimode_local_general}, then $\mathfrak{U}_z u$ satisfies
  \eqref{eq:quasimode_normal} (with a smaller constant
  $c>0$). In particular, by Proposition
  \ref{prop:microlocal_sheaf_model}, $\mathfrak{U}_z u$ is exponentially close to a
  Lagrangian state which is prescribed up to a multiplicative factor.

  Applying $\mathfrak{V}_z$ as in Proposition
  \ref{prop:invert_FIO}, we find, again up to restricting the domain
  of $|z|$, that $u$ is exponentially close
  to a Lagrangian state which is prescribed up to a
  multiplicative factor.

  From there, it only remains to apply Proposition \ref{prop:subp_calc} to obtain the
  equation on the principal symbol.
\end{proof}

\section{Semiglobal model}
\label{sec:semiglobal}

In this section, we improve the results above into a description near
regular energy curves. More precisely, we work under the following
list of hypotheses.

\begin{hyp}\label{hyp:semiglobal}~
  \begin{enumerate}
  \item $(M,J,\omega)$ is a real-analytic, compact quantizable Kähler manifold of complex dimension 1.
  \item $p:\C\times M\to \C$ is a real-analytic, complex-valued
    Hamiltonian with holomorphic dependence on the first
    coordinate. We write
    \[
      p_z=p(z,\cdot).\]
  \item $p_0$ is real-valued.
  \item $\mathcal{C}\subset M$ is a regular, complete, connected piece
    of energy level of $p_0$.
  \end{enumerate}
\end{hyp}

In spirit, this description involves gluing together the quasimodes of Section
\ref{hyp:local}; gluing conditions for quasimodes will yield
conditions on the eigenvalues. We prefer developing a more global
approach and we conjugate the problem to a (non-trivial) spectral
function of $T_k^{\rm cov}(\xi)$ acting on the relevant quantum space $\mathcal{B}_k^{S^1}$. This
construction is more geometric and makes apparent the role played by
the Bohr-Sommerfeld action.

Recall from Proposition \ref{prop:Lagrangians} that Lagrangians can be associated with Fourier Integral
Operators only if they have trivial Bohr-Sommerfeld class. In our
situation, the Bohr-Sommerfeld class is a single number, corresponding
to the integral of a well-chosen antiderivative of $\Omega$ (the
connection form for $\widetilde{\nabla}$) along a curve.

\begin{defn}\label{def:action}
  Let $M$ be a quantizable Kähler manifold and let $L\to M$ be a prequantum
  line bundle. Let $\Lambda$ be a holomorphic Lagrangian in $\widetilde{M}$ and
  suppose that $\Lambda$ is a neighbourhood of a real, oriented closed curve, so
  that $\pi_1(\Lambda)=\Z$. We denote by $I(\Lambda)\in \R/\Z$ the generator of
  the Bohr-Sommerfeld class of $\Lambda$ associated with the curve;
  that is, $\exp(2i\pi I(\lambda))$ is the ratio of values between
  points in $L$ above the same point of the curve, before and after
  parallel transport along the curve.
\end{defn}

As before, we first produce a normal form for $T_k^{\rm cov}(p_z)$
near $\mathcal{C}$, then use it to describe the quasimodes. An
additional feature of dealing with a semiglobal normal form is that we
obtain that there is no Jordan block phenomenon: approximate solutions
for $(T_k^{\rm cov}(p-\lambda))^2$ are also approximate solutions for
$T_k^{\rm cov}(p-\lambda)$.

The main additional difficulty is to produce normal forms. At the
level of principal symbols, we develop action-angle coordinates, the
theory of which does not seem well-developed in the complex
holomorphic case. The simplification of subprincipal terms, again,
calls for a careful proof in the context of analytic symbols.

\subsection{Normal forms}
\label{sec:regul-bohr-somm}

We first transform Proposition \ref{prop:flowbox_assemblage} into a
normal form on $T^*S^1$ -- with Berezin--Toeplitz
quantization. As before, we begin with the ``classical'' problem. Note that in the case of $M = T^*S^1$ such a classical normal form has been derived in \cite[Section 2.2.1]{rouby_bohrsommerfeld_2017}; we will follow a different method for the proof.

We work under the Hypotheses \ref{hyp:semiglobal}. Let $\widetilde{U}$ be a
neighbourhood of $\mathcal{C}$ in $\widetilde{M}$. For $z$ close to
$0$ and $\lambda$ close to $p_0(\mathcal{C})$, the map
\[
  (z,\lambda)\mapsto I(\{\widetilde{p}_z=\lambda\}\cap \widetilde{U})
\]
is well-defined (at $z=0$, the real level set inherits an orientation
from $X_{p_0}$) and holomorphic. We claim that
\[
  \frac{\partial I}{\partial \lambda}\neq 0.
\]
Indeed, by Stokes' theorem, at $z=0$ and for real variations of $\lambda$, $\frac{\partial
  I}{\partial \lambda}$ is the inverse period of the flow of $p_0$. In
particular, for every fixed $z$ close to $0$, the map 
\begin{equation} I_z: \lambda \mapsto I(\{\widetilde{p}_z=\lambda\}\cap \widetilde{U}) \label{eq:Iz_def} \end{equation}
admits a reciprocal denoted by $I_z^{-1}$.

\begin{prop}\label{prop:angle-action}
 Suppose Hypothesis \ref{hyp:semiglobal} holds.
  There exist an open neighbourhood $\mathcal{Z}$ of $0$ in $\C$, an open
  neighbourhood $\widetilde{U}$ of $\mathcal{C}$ in $\widetilde{M}$ and a real-analytic map
  $\kappa:\mathcal{Z}\times \widetilde{U}\to \widetilde{T^*S^1}$, with holomorphic
  dependence in the first variable, such that:
  \begin{itemize}
  \item for every $z\in \mathcal{Z}$, the map $\kappa_z:\widetilde{U}\to \widetilde{T^*S^1}$ is a symplectomorphism;
  \item the graph of $\kappa_z$ has trivial Bohr-Sommerfeld class;
  \item on $\widetilde{U}$ there holds
 \[ I_z\circ \widetilde{p}_z=\xi\circ \kappa_z  \]
 where $I_z$ is as in Equation \eqref{eq:Iz_def}.
 \end{itemize}
  Moreover $\kappa_0$ is the holomorphic extension of a real symplectic map which maps
  $\mathcal{C}$ to $\{\xi=I_0(p_0(\mathcal{C}))\}$.
\end{prop}
\begin{proof}
  Since $\partial_\lambda I_z\neq 0$, the Hamiltonian $q_z=\frac{1}{2\pi}I_z\circ
  p_z$ satisfies the conditions in Hypothesis \ref{hyp:semiglobal}. Moreover,
  $\widetilde{q}_z$ satisfies a crucial supplementary assumption: its Hamiltonian flow is
  $2\pi$-periodic. This occupies the first part of the proof. Let first $\lambda\in \R$ be close to
  $q_0(\mathcal{C})$. There exists a periodic trajectory for $q_0$ at
  energy $\lambda$, which goes along the circle
  $q_0^{-1}(\lambda)$. Let $T(0,\lambda)$ denote its period. Now let
  $x_{z,\lambda}\in \{\widetilde{q_z}^{-1}(\lambda)\}$ with holomorphic dependence on $z$
  and $x_{0,p_0(\mathcal{C})}\in M$. The map $t\mapsto \phi_{\widetilde{q_z}}^t(x_{z,\lambda})$ is
  a local biholomorphism from $\C$ to $\widetilde{M}$; therefore, by
  the inverse function theorem, for $z$ close to $0$ in $\C$ there exists a unique $T(z,\lambda)$
  close to $T(0,\lambda)$ such
  that \[\phi_{\widetilde{q_z}}^{T(z,\lambda)}(x_{z,\lambda})=x_{z,\lambda}.\]
  Moreover, $T(z,\lambda)$ has real-analytic dependence on $\lambda\in
  \R$,
  and therefore the property above holds for $(z,\lambda)$ close to
  $(0,q_0(\mathcal{C}))$ in $\C\times \C$.

  Define now, for fixed $z$ near $0$,
  \[
    s:(\lambda,\theta)\mapsto
    \phi^{\frac{\theta}{2\pi} T(z,\lambda)}_{\widetilde{q_z}}(x_{z,\lambda}).
  \]
  The trajectory $\mathcal{C}_{z,\lambda}=\{s(\lambda,\theta),\theta\in
  S^1\}$ forms a loop inside $\{\widetilde{q_z}=\lambda\}$, which is
  close to $\mathcal{C}$, and along which one can compute the action
  of $\{\widetilde{q_z}=\lambda\}$ as
  \begin{align*}
    I(\lambda)&=\oint_{\mathcal{C}_{z,\lambda}}\alpha(s(\lambda,\theta))\\
    &=\int_0^{2\pi}\alpha(\tfrac{\partial s}{\partial \theta})\dd
      \theta\\
    &=\frac{T(z,\lambda)}{2\pi}\int_0^{2\pi}\alpha(X_{\widetilde{q_z}}(s))\dd \theta
  \end{align*}
  where $\partial \alpha=\Omega$. Now, by assumption $I(\lambda)=2\pi\lambda$,
  and moreover, by Stokes' formula,
  \begin{align*}
    \frac{\partial}{\partial
    \lambda}I(\lambda)&=\int_0^{2\pi}\frac{\partial}{\partial
                        \lambda}\alpha(\tfrac{\partial s}{\partial
                        \theta})\dd \theta\\
    &=\int_0^{2\pi}\Omega(\tfrac{\partial s}{\partial
      \theta},\tfrac{\partial s}{\partial \lambda})\dd \theta\\
    &=\frac{T(z,\lambda)}{2\pi}\int_0^{2\pi}\Omega(X_{\widetilde{p_z}},\tfrac{\partial
      s}{\partial \lambda})\dd \theta\\
    &=\frac{T(z,\lambda)}{2\pi}\int_0^{2\pi}\dd \widetilde{p_z}(\tfrac{\partial
      s}{\partial \lambda})\dd \theta\\
    &=T(z,\lambda)
  \end{align*}
  since by definition
  $\widetilde{q_z}(s(\lambda,\theta))=\lambda$. Thus $T(z,\lambda)=2\pi$,
  and we find that $X_{\widetilde{q_z}}$ is $2\pi$-periodic on
  $\mathcal{C}_{z,\lambda}$.

  To conclude this part of the proof, since $\mathcal{C}_{z,\lambda}$
  is a maximally totally real submanifold of
  $\{\widetilde{q_z}=\lambda\}$, the holomorphic equation
  \[
    \phi_{\widetilde{q_z}}^{2\pi}(x)=x,
  \]
  valid on $\mathcal{C}_{z,\lambda}$, is therefore true on the whole
  of $\{\widetilde{q_z}=\lambda\}$.
  
  Since the flow of $\widetilde{q}_z$ is $2\pi$-periodic, the dynamical
  construction in the proof of Proposition
  \ref{prop:flowbox_classical_real} can be closed into a symplectic change of
  variables $\kappa_z$ from $\widetilde{U}$ to $\widetilde{T^*S^1}$, which maps
  $\widetilde{q}_z$ to $\xi$.

  It remains to show that the graph of $\kappa_z$ has trivial
  Bohr-Sommerfeld class. This graph contracts onto a curve whose
  projection on the first variable is a closed trajectory $\Lambda_0$ of $\widetilde{q}_z$ and
  whose projection on the second variable is $\{\theta\in \R,\xi=\widetilde{q}_z(\Lambda_0)\}$. In these
  circumstances, the Bohr-Sommerfeld class of the graph is generated
  by \[\frac{\exp(2i\pi I(\lambda))}{\exp(2i\pi I(\{\theta\in \R,\xi=\widetilde{q}_z(\Lambda_0)\}))}\] and
  by construction the two actions coincide. 
\end{proof}
Following Proposition \ref{prop:conjug_Toep_FIO}, the map $\kappa_z$ is
quantized by a Fourier Integral operator which conjugates $T_k^{\rm
  cov}( p_z)$ to $T_k^{\rm cov}(I_z^{-1}\circ \xi+k^{-1}r)$ for some
analytic symbol $r$. As before, it remains to correct this
subprincipal error.

\begin{prop}\label{prop:correction_sousprincipal_semiglobal}
  Let $f_0:\C\to \C$ be real-analytic with $f_0'\neq 0$. Let $r$ be a real-analytic symbol on a neighbourhood $U$ of a horizontal
  curve in $T^*S^1$. There exist a real-analytic amplitude $g:\C\to
  \C$ and a real-analytic symbol $a$ on $U$ such that
  \[
    (f_0\circ \xi+k^{-1}r)=a^{-1}\star_{\rm cov}(f_0+k^{-1}g)\circ \xi\star_{\rm cov}a
  \]
\end{prop}
\begin{proof}
  The proof mostly follows the same lines as that of Proposition
  \ref{prop:flowbox_correction} but we have to take into account the
  non-trivial topology of the problem. As before, we proceed by
  deformation and try to solve
  \[
    (f_0\circ\xi+tk^{-1}r)=a_t^{-1}\star_{\rm cov}(f_0+k^{-1}g)\circ
    \xi\star_{\rm cov}a_t;
  \]
  at $t=0$ we set $a_t=1$ and $g_t=0$. Again we let
  $b=\partial_ta_t\star_{\rm cov}a_t^{-1}$, and differentiate with
  respect to $t$ to find
  \[
    k^{-1}a_t\star_{\rm cov}r\star_{\rm cov}a_t^{-1}=[(f_0+k^{-1}g_t)(\xi),b_t]+k^{-1}\partial_tg_t.
  \]
  Letting $p=a_t\star_{\rm cov}r\star_{\rm cov}a_t^{-1}$, the
  commutator $[(f_0+k^{-1}g_t)(\xi),b_t]$ has zero average over $\theta$ so it
  remains to solve the following system of ODEs in an appropriate
  analytic symbol space:
    \[ \begin{cases}
    \partial_tg_t(\xi)=-\langle
        p\rangle_{\theta}(\xi),\\
        \partial_t p=[b,p],\\
        k[f_0(\xi),b]+[g(\xi),b]=p-\langle p\rangle_{\theta}.
    \end{cases} \]
  It remains to show that one can apply the Picard-Lindelöf theorem to
  this system. The point is that there exists a unique $b$ with zero
  average over $\theta$ such that $k[(f_0+k^{-1} g)(\xi,b)]=p-\langle
  p\rangle_{\theta}$, and $(p,g)\mapsto b$ is Lipschitz on good analytic
  symbol spaces.

  Indeed, one has first that $(b,g)\mapsto [g,b]$ is Lipschitz-continuous on $BK(T)$,
  with a Lipschitz constant proportional to $T$ (since the first order
  vanishes).

  Moreover, let $A$ be the linear operator of antiderivation on the
  space of analytic symbols with vanishing $\theta$ average. Then $A$
  is automatically continuous on the spaces $BK(T)$.

  Next,
  \[
    -ik[f_0(\xi),b]=f'(\xi)\partial_{\theta}b+\underbrace{\sum_{k=1}^{+\infty}\frac{k^{-2j}}{(2j+1)!}f_0^{(2j+1)}(\xi)\partial_{\theta}^{2j+1}b}_{R(b)}
  \]
  where $b\mapsto AR(b)$ is Lipschitz-continuous on $BK(T)$ for $T$ small
  enough, with Lipschitz constant proportional to $T$. Indeed $f_0$ is
  a fixed real-analytic function, applying $A$ lowers by $1$ the
  order of differentiation in $\theta$, and the principal symbol
  vanishes. We obtain
  \[
    AR(b)=\sum_{k=1}^{+\infty}\frac{f_0^{(2j+1)}(\xi)}{(2j+1)!}{\rm
      ad}_{\xi}^{2j}b
  \]
  where $\|{\rm ad}_{\xi}\|_{BK(T)\to BK(T)}=O(T)$ and
  $\|f_0^{(2j+1)}\|_{BK(T)}\leq C_0^j$; for $T$ small enough the sum
  converges in $BK(T)$ and the result is $O(T^2)$.

  All in all, the last line of the system above reads
  \[
    b=\frac{1}{f'(\xi)}A\left[p-\langle
      p\rangle_0-[g(\xi),b]+iR(b)\right]\]
  and by the Banach fixed point theorem, for $T$ small enough, there
  exists a unique solution $b\in BK(T)$ to this problem and
  $(p,g)\mapsto b$ is Lipschitz-continuous in this topology.
\end{proof}

Putting together Propositions \ref{prop:angle-action} and
\ref{prop:correction_sousprincipal_semiglobal} we obtain the following
semiglobal normal form.

\begin{prop}\label{prop:forme_normale_semiglobale}Assume Hypothesis
  \ref{hyp:semiglobal} holds.
  There exist $c>0$, a neighbourhood $\mathcal{Z}$ of $0$ in $\C$, an open
  neighbourhood $U$ of $\mathcal{C}$ in $\widetilde{M}$,
  an open neighbourhood $V$ of $\{\xi=I(p_0(\mathcal{C}))\}$ in
  $T^*S^1$, an analytic amplitude $f:\mathcal{Z}_z\times \R_{\xi}\to
  \C$ with $\partial_{\xi}f\neq 0$, with holomorphic dependence on $z$
  and Fourier integral operators
  \[ \mathfrak{U}_z:H^0(M,L^{\otimes k})\to \mathcal{B}_k^{S^1}, \qquad \mathfrak{V}_z:\mathcal{B}_k^{S^1}\to H^0(M,L^{\otimes k}) \]
  with holomorphic dependence on $z\in \mathcal{Z}$, which are
  microlocal inverses of each other on the domains $U$ and $V$, and such that, uniformly for $z\in
  \mathcal{Z}$, for every
  $u\in H^0(M,L^{\otimes k})$, 
  \[
    \mathfrak{U}_z T_k^{\rm cov}(p_z)u=T_k^{\rm
      cov}(f_z(\xi;k^{-1}))\mathfrak{U}_z u+O(e^{-ck}\|u\|_{L^2}+e^{c^{-1}|z|k}\|u\|_{L^2(M\setminus
      U)})
  \]
  and for every $v\in \mathcal{B}_k^{S^1}$,
  \[
    \mathfrak{V}_z T_k^{\rm cov}(f_z(\xi;k^{-1}))v=T_k^{\rm cov}(p_z)\mathfrak{V}_z v+O(e^{-ck}\|v\|_{L^2}+e^{c^{-1}|z|k}\|v\|_{L^2(\C\setminus
      V)}).
  \]
  The principal symbol of $f$ is the reciprocal of the action map
  $\lambda\mapsto I(\{\widetilde{p}_z=\lambda\}\cap \widetilde{U})$ for some
  neighbourhood $\widetilde{U}$ of $\mathcal{C}$ in $\widetilde{M}$.
\end{prop}

\subsection{Quasimodes}
\label{sec:quasimodes}

Thanks to Proposition \ref{prop:forme_normale_semiglobale} we can study the quasimodes for
$T_k^{\rm cov}(p_z)$ near $U$. As before, they are necessary
Lagrangian states, but contrary to Proposition
\ref{prop:exists_microlocal_solution_local}, the topology forces a
Bohr-Sommerfeld rule on the energies.

\begin{prop}\label{prop:study_quasimodes_semiglobal}Suppose Hypothesis
  \ref{hyp:semiglobal} holds.
  Given $c>0$ and $U_1\Subset U$ there exists $c'>0$ such that, if
  there exists $u\in H^0(M,L^{\otimes k})$ with $\|u\|_{L^2}=1$ and
  $\|u\|_{L^2(U)}>\frac 12$ and
  \[
    \|T_k^{\rm cov}(p_z-\lambda)u\|_{L^2(U_1)}=O(e^{-ck})
  \]
  then
  \[
    f_z^{-1}(\lambda;k^{-1})=\frac{j}{k}+O(e^{-c'k}),\qquad j\in \Z
  \]
  where $f_z^{-1}$ denotes the reciprocal map of $\xi\mapsto
  f_z(\xi)$.

  Moreover, if $u\in H^0(M,L^{\otimes k})$ is normalised with $\|u\|_{L^2(U)}>1$ and satisfies
  \[
    \|(T_k^{\rm cov}(p_z-\lambda))^2u\|_{L^2(U_1)}=O(e^{-ck})
  \]
  then one also has
  \[
    \|(T_k^{\rm cov}(p_z-\lambda))^2u\|_{L^2(U_1)}=O(e^{-c'k})
  \]

  Reciprocally, there exist $c_0>0$, $c_1>0$, and neighbourhoods
  $\mathcal{Z}$ of $0$ in $\C$ and $\mathcal{E}$ of
  $p_0(\mathcal{C})$ in $\C$, such that if $z\in \mathcal{Z}$ and
  $\lambda\in \mathcal{E}$ solve
  \[
    \exists j\in \Z\qquad f_z(\lambda;k^{-1})=\frac{j}{k}+O(e^{-c_0k}),
  \]
  then there exists $u\in H^0(M,L^{\otimes k})$ normalised with
  $\|u\|_{L^2(U)}>\frac 12$ and
  \[
    \|T_k^{\rm cov}(p_z-\lambda)u_k\|_{L^2(U_1)}=O(e^{-c_1k}).
  \]
  
\end{prop}
We recall from Proposition \ref{prop:local_WKB} that the quasimodes above
are necessary of WKB form.
\begin{proof}
  First, obtain $\mathfrak{U}_z$, $\mathfrak{V}_z$ and $f_z$ as in Proposition \ref{prop:forme_normale_semiglobale}. Note that
  \[
    T_k^{\rm cov}(f_z(\xi;k^{-1}))=f_z(T_k^{\rm cov}(\xi);k^{-1}).
  \]
  The operator $T_k^{\rm cov}(\xi)=ik^{-1}\frac{\partial}{\partial \theta}$ is self-adjoint on $T^*S^1$ and
  its eigenvalues are of the form $jk^{-1}$ for $j\in \Z$. Thus, the
  operator $f_z(T_k^{\rm cov}(\xi);k^{-1})$ is normal, and its
  eigenvalues are of the form $f_z(jk^{-1};k^{-1})$ for $j\in
  \Z$. Normality means that the resolvent is bounded, from above and
  below, by the inverse distance to the spectrum.

  Applying $\mathfrak{V}_z$ to an eigenfunction of $T_k^{\rm cov}(\xi)$
  will yield a quasimode of $T_k^{\rm cov}(p_z)$ at eigenvalue
  $f_z(jk^{-1};k^{-1})$; reciprocally, by normality of $f_z(T_k^{\rm
    cov}(\xi))$, applying $\mathfrak{U}_z$ to a quasimode of $T_k^{\rm
    cov}(p_z)$ yields a condition on the eigenvalue.

  It remains to prove that quasimodes of $(T_k^{\rm
    cov}(p_z-\lambda))^2$ are also quasimodes of $T_k^{\rm
    cov}(p_z-\lambda)$. To this end we use again the normality of the
  model operator: applying $\mathfrak{U}_z$ to a quasimode of $(T_k^{\rm
    cov}(p_z-\lambda))^2$ yields a zero quasimode for
  $(f_z(T_k^{\rm cov}(\xi);k^{-1})-\lambda)^2$; thus it is also a
  zero quasimode for $f_z(T_k^{\rm cov}(\xi);k^{-1})-\lambda$, and by
  application of $\mathfrak{V}_z$ we recover a zero quasimode for $T_k^{\rm cov}(p_z)-\lambda$.
\end{proof}

In practice, one can use Proposition
\ref{prop:study_quasimodes_semiglobal} to give explicit necessary
conditions on quasimodes; for instance, one can give Bohr-Sommerfeld
conditions up to $O(k^{-2})$ which involve the action and a geometric
subprincipal contribution.

\begin{prop}\label{prop:BS_mod_k-2}
  In the context of Proposition
  \ref{prop:study_quasimodes_semiglobal}, decompose
  \[
    p_z=p_{z;0}+k^{-1}p_{z;1}+O(k^{-2}).
  \]
  Let $\delta$ be a topologically trivial half-form bundle over $U$. As in Definition
  \ref{def:action}, denote by $I_{\rm sub}(\Lambda)$ the generator of the
  Bohr-Sommerfeld class of a Lagrangian $\Lambda$ close to
  $\widetilde{C}$, relative to the bundle $\delta$.

  Then one has
  \begin{equation}
    f_z^{-1}(\lambda;k^{-1})=I(\{\widetilde{p_{z;0}}=\lambda\})+k^{-1}\left[\oint_{\{\widetilde{p_z}=\lambda\}}(\widetilde{p_{z;1}}-\tfrac
      12 \widetilde{\Delta p_{z;0}})\kappa-I_{\rm sub}(\{\widetilde{p_z}=\lambda\})\right]+O(k^{-2})\label{eq:BS_mod_k-2}
  \end{equation}
  where $\kappa$ is the unique one-form on
  $\{\widetilde{p_z}=\lambda\}$ such that
  $\kappa(X_{\widetilde{p_{z;0}}})=1$.
\end{prop}
\begin{proof}We already know from the construction
  that \[f_z^{-1}(\lambda;k^{-1})=I(\{\widetilde{p_{z;0}}=\lambda\})+O(k^{-1})
  \]
  and it remains to compute the subprincipal term. To do so, we will
  use the subprincipal calculus of Proposition \ref{prop:subp_calc}.
  
  Let us lift Proposition \ref{prop:forme_normale_semiglobale} (and thus obtain Fourier Integral Operators $\mathfrak{U}_z$ and $\mathfrak{V}_z$) to the
  universal cover of $U$, on one side, and the universal cover of
  $T^*S^1$, on the other side. Now, by Proposition
  \ref{prop:exists_microlocal_solution_local}, there exists a local
  quasimode (which goes around $U$ at least once) for every $\lambda$
  near $p_0(\mathcal{C})$.

  We now claim that, in this picture,
  $\exp(2i\pi kf_z^{-1}(\lambda;k^{-1}))$ is the phase shift between two
  different points projecting down to the same point of $U$ and
  separated by one period. Indeed, after conjugation by $\mathfrak{U}_z$
  and $\mathfrak{V}_z$, $f_z^{-1}(\lambda;k^{-1})=\mu$ is the eigenvalue
  of $T_k^{\rm cov}(\xi)$ (acting on $\mathcal{B}_k$) for
  which we are considering a quasimode. We already know that this
  quasimode is of the form
  \[
    v:(x,\xi)\mapsto \exp(-k\tfrac{\xi^2}{2})\exp(ik\mu(x+i\xi))
  \]
  and therefore
  \[
    v(x+2\pi,\xi)=\exp(2i\pi k \mu)v(x,\xi).
  \]
  The phase shift is preserved by $\mathfrak{V}_z$, since this operator
  commutes with translation along one entire period.

  On the other hand, we can compute the subprincipal term in this
  phase shift by using Proposition \ref{prop:subp_calc}. Letting
  $u=\mathfrak{V}_z v$, then $u$ is of the form $I_k^{\Phi}(a)$ for some
  phase $\Phi$ and symbol $a$, and
  solves
  \[
    T_k^{\rm cov}(p_z-\lambda)I_k^{\Phi}(a)=O(e^{-ck})
  \]
  for some $c>0$. In the setting of Proposition \ref{prop:subp_calc},
  the fact that $b_0$ and $b_1$ vanish yields conditions on
  $\Phi$ and $a_0$. More precisely
  \[
    \text{if }\iota^*\widetilde{p_{z;0}}=0\text{ then }\{\widetilde{\nabla}
    \Phi=0\}=\{\widetilde{p_{z;0}}=\lambda\}
  \]
  so we know that $\Phi$ is a phase associated with the Lagrangian
  $\{\widetilde{p_{z;0}}=\lambda\}$. As such, the phase shift in
  $\Phi^{\otimes k}$ after one period is equal to the parallel transport
  along one period in $L$, which is equal to
  $\exp(2i\pi k I(\{\widetilde{p_{z;0}}=\lambda\}))$ by definition.

  Now, the vanishing of the subprincipal term yields a transport
  equation on $a_0$; we solve it using the decomposition of
  Proposition \ref{prop:subp_calc}. The first two terms in the
  expansion
  \[
    -\frac 12 \Delta p_{z;0}+\frac 12 \nabla^{\iota^*(\delta^{-1})}_{X_{\widetilde{p_{z;0}}}}
  \]
  respectively yield the integral of $\Delta f_0$ along the flow and
  the subprincipal action along $\delta^{-1}$, by definition. Now, since
  $\delta$ is topologically trivial, the integral along one period of
  $\mathcal{L}_{X_{\widetilde{p_{z;0}}}}$ applied to the trivialising
  section must vanish (since it is a closed form).

  We conclude the proof with the remark that the choice of a
  topologically non-trivial square-root $\delta$ yields an addition of
  $\pi$ in the subprincipal action, which is exactly compensated by
  the contribution of $\mathcal{L}_{X_{\widetilde{p_{z;0}}}}$.
\end{proof}
\section{Global results}
\label{sec:global-results}
Now we are ready to glue together the results of our analysis near
each connected component into a study of the actual spectrum, provided
the level set is regular.

\begin{hyp}\label{hyp:global}~
  \begin{enumerate}
  \item $(M,J,\omega)$ is a real-analytic, compact, quantizable Kähler manifold of complex dimension 1.
  \item $p:\C\times M\to \C$ is a real-analytic, complex-valued
    Hamiltonian with holomorphic dependence on the first
    coordinate. We write
    \[
      p_z=p(z,\cdot).\]
  \item $p_0$ is real-valued.
  \item $\lambda_0$ is a regular energy level for $p_0$.
  \item $\mathcal{C}_1,\cdots,\mathcal{C}_N$ are the connected
    components of $\{p_0=\lambda_0\}$.
  \end{enumerate}
\end{hyp}
We begin with a resolvent bound away from the Bohr-Sommerfeld solutions.
\begin{prop}\label{prop:resolvent_bound_away}
  Suppose Hypothesis \ref{hyp:global} holds. For every $c>0$, there
  exist $c'>0$, a neighbourhood $\mathcal{Z}$ of $0$ in $\C$ and a
  neighbourhood $\mathcal{E}$ of $\lambda_0$ in $\mathcal{C}$ such
  that for every $(z,\lambda)\in \mathcal{Z}\times \mathcal{E}$, the
  existence of $u\in H^0(M,L^{\otimes k})$ normalised such that
  \[
    \|(T_k^{\rm cov}(p_z)-\lambda)u\|_{L^2(M)}=O(e^{-ck})
  \]
  implies that there exist $1\leq n\leq N$ and $j\in \N$ such that
  \[
    f_n(z;k^{-1})=\frac{j}{k} + O(e^{-c'k}).
  \]
\end{prop}
\begin{proof}
  Let $u$ be normalised and satisfying
  \[
    \|(T_k^{\rm cov}(p_z)-\lambda)u\|_{L^2(M)}=O(e^{-ck}).
  \]
  Away from a neighbourhood of $\{p_0=\lambda_0\}$, one can microlocally
  invert $T_k^{\rm cov}(p_z-\lambda)$; therefore $u$ is exponentially
  small away from $\{p_0=\lambda_0\}$. In particular, letting $U_n$ be
  a neighbourhood of $\mathcal{C}_n$, then
  \[
    u_n=\Pi_k(\1_{U_n}u)
  \]
  satisfies, for some $c_1>0$,
  \[
    u=u_1+\ldots+u_N+O(e^{-c_1k})
  \]
  as well as
  \[
    \|(T_k^{\rm cov}(p_z)-\lambda)u_n\|_{L^2(M)}=O(e^{-c_1k}).
  \]
  There exists $1\leq n\leq N$ such that $\|u_n\|_{L^2}\geq
  \frac{1}{2N}$, and therefore we can apply the first part of Proposition
  \ref{prop:study_quasimodes_semiglobal} to
  $\frac{u_n}{\|u_n\|}$. This concludes the proof.
\end{proof}
On these points where the resolvent norm is not too large, the
resolvent can be obtained by the local models. In particular, the
resolvent is ``local'' in the sense that, on those points, one does
not see any interaction between the different components $\mathcal{C}_1,\cdots,\mathcal{C}_N$.
\begin{prop}\label{prop:local_resolvent}
  Suppose Hypothesis \ref{hyp:global} holds. For every $n \in \{1, \cdots, N\}$, let $\mathfrak{U}_n$ and $\mathfrak{V}_n$ be Fourier Integral Operators as in Proposition \ref{prop:forme_normale_semiglobale} applied to $\mathcal{C}_n$ (here $z$ is silent to simplify notation). For every $c>0$, there
  exist $c'>0$, a neighbourhood $\mathcal{Z}$ of $0$ in $\C$, a
  neighbourhood $\mathcal{E}$ of $\lambda_0$ in $\mathcal{C}$, and
  neighbourhoods $ U_1,\cdots, U_N$ of
  $\mathcal{C}_1,\cdots,\mathcal{C}_N$, such that the following is
  true. Suppose
  that $(z,\lambda)\in \mathcal{Z}\times \mathcal{E}$
  does not satisfy
  \[
    f_n(z;k^{-1})=\frac{j}{k} +O(e^{-c'k}) 
  \]
  for any $j\in \N$ and $1\leq n\leq N$.
  Then for every $1\leq n\leq N$, \[(T_k^{\rm
    cov}(p_z)-\lambda)^{-1}\1_{U_n}=\mathfrak{V}_n(T_k^{\rm
    cov}(f_n(\xi;k^{-1}))-\lambda)^{-1}\mathfrak{U}_n\1_{U_n}+O(e^{-c'k}).\]
Moreover, letting $r_z(\lambda)$ be the inverse to $p_z-\lambda$ for the
product of covariant analytic symbols on $M\setminus (U_1\cup\cdots
\cup U_n)$, one has
\[
  (T_k^{\rm cov}(p_z)-\lambda)^{-1}\1_{M\setminus (U_1\cup\cdots
    \cup U_n)}=T_k^{\rm cov}(r_z(\lambda))+O(e^{-c'k}).
\]
\end{prop}
\begin{proof}
  To simplify notation, let
  \[ \begin{cases} A=T_k^{\rm
    cov}(p_z)-\lambda \\ R=(T_k^{\rm
    cov}(p_z)-\lambda)^{-1} \end {cases} \qquad \text{and} \qquad \begin{cases} B=\mathfrak{V}_n(T_k^{\rm
    cov}(f_n(\xi;k^{-1}))-\lambda)\mathfrak{U}_n\1_{U_n} \\ S=\mathfrak{V}_n(T_k^{\rm
    cov}(f_n(\xi;k^{-1}))-\lambda)^{-1}\mathfrak{U}_n\1_{U_n} \end{cases}  
   \]
Moreover, let $\chi=\Pi_k\1_{U_n}\Pi_k$ and $\chi_1=\Pi_k\1_{W_n}\Pi_k$ where
  $U_n\Subset W_n$. One has of course $RA=AR=\Pi_k$ and
\[
  SB\chi=\chi+O(e^{-c_1k}), \qquad BS\chi=\chi+O(e^{-c_1k}), \qquad
  (A-B)\chi_1=O(e^{-c_1k}),\qquad (1-\chi_1)S\chi=O(e^{-c_1k});
\]
moreover $R,S,A,B$ are all bounded in operator norm by $O(e^{\varepsilon k})$ for some
$\varepsilon>0$ much smaller than $c_1$ (up to restricting $\mathcal{Z}$
and $\mathcal{E}$). Now
\begin{align*}
  (R-S)\chi&=R(AR-AS)\chi\\
           &=R(\Pi_k-AS)\chi\\
           &=R(\Pi_k-A\chi_1S)\chi+O(e^{-(c_1-\varepsilon)k})\\
           &=R(\Pi_k-B\chi_1S)\chi+O(e^{-(c_1-2\varepsilon)k})\\
           &=R(\Pi_k-BS)\chi+O(e^{-(c_1-2\varepsilon)k})\\
           &=O(e^{-(c_1-2\varepsilon)k}).
\end{align*}
This concludes the proof.

Similarly, away from $U_1\cup\cdots\cup U_n$, $(T_k^{\rm
  cov}(p_z)-\lambda)T_k^{\rm cov}(r_z(\lambda))$ is close to $1$, and
we can multiply by $R$ on the right to obtain the desired result.
\end{proof}
The structure of the resolvent allows us to study the spectral problem
inside of the regions where
$f_n(z;k^{-1})+O(e^{-c'k})=\frac{j}{k}$. For fixed $c'>0$, for every
$n$, said region is a union of open neighbourhoods of size
$O(e^{-c'k})$ of points separated by at least $ck^{-1}$ for $c>0$. The
union over $1\leq n\leq N$ forms a discrete family of open sets of
size $O(e^{-c'k})$, but now any of those open sets may overlap at most
$N-1$ others (each corresponding to one curve $\mathcal{C}_j$). Thus,
\[
  \Omega_{c'}=\bigcup_{n=1}^N\{f_n(z;k^{-1})=\tfrac{j}{k}+O(e^{-c'k})\}
\]
is a discrete union of connected sets of diameter $O(e^{-c'k})$. Each
of the connected components of $\Omega_{c'}$ has a
\emph{Bohr-Sommerfeld multiplicity}
which corresponds to the amount of $n$ in $[1,N]$ such that there
exists a solution for the corresponding $n$.

\begin{prop}\label{prop:multiplicity}
  Suppose Hypothesis \ref{hyp:global} holds. For every $c'>0$, there exist a neighbourhood $\mathcal{Z}$ of $0$
  in $\C$ and a neighbourhood $\mathcal{E}$ of $\lambda_0$ in $\C$
  such that, for every $z\in \mathcal{Z}$, the number of eigenvalues
  (counted with geometric multiplicity) of $T_k^{\rm cov}(p_z)$ within any connected
  component of $\Omega_{c'}$ contained in $\mathcal{E}$ is equal to
  the Bohr-Sommerfeld multiplicity of this connected
  component.
\end{prop}
\begin{proof}
  Consider a loop $\gamma\subset (\C\setminus \Omega_{c'})$ around a
  connected component $W$ of $\Omega_{c'}$. The space spanned by the generalised
  eigenvectors (in the Jordan sense) in the
  spectral decomposition of $T_k^{\rm cov}(p_z)$ with eigenvalues in
  $W$ has for spectral projector
  $\Pi_W=\frac{1}{2i\pi}\oint_{\gamma}(T_k^{\rm cov}(p_z)-\lambda)^{-1}\dd
  \lambda$.

  Now we use Proposition \ref{prop:local_resolvent} and its notation to study this
  integral. First, away from the curves $\mathcal{C}_1,\cdots,\mathcal{C}_N$, we can replace $(T_k^{\rm
    cov}(p_z)-\lambda)^{-1}$ with $T_k^{\rm cov}(r_z(\lambda))$ up to
  an exponentially small error. $\lambda\mapsto r_z(\lambda)$ is
  holomorphic on $\mathcal{E}$ and therefore
  \[
    \Pi_W\1_{M\setminus U_1\cup \cdots \cup U_N}=O(e^{-c''k})
  \]
  for some $c''>0$, up to reducing $\mathcal{Z}$.

  Let now $1\leq n\leq N$, and suppose that
  $f_n(z;k^{-1})=\frac{j}{k}+O(e^{-c'k})$ has no solution in $W$. Then
  \[
    \mathfrak{V}_n(T_k^{\rm cov}(f_n(\xi;k^{-1}))-\lambda)^{-1}\mathfrak{U}_n
  \]
  is holomorphic on $W$; therefore again
  \[
    \Pi_W\1_{U_n}=O(e^{-c''k}).
  \]
  Now, if $1\leq n\leq N$ is such that
  $f_n(z;k^{-1})=\frac{j}{k}+O(e^{-c'k})$ admits a solution in $W$,
  then $j$ is fixed, and
  \[
    \mathfrak{V}_n(T_k^{\rm cov}(f_n(\xi;k^{-1}))-\lambda)^{-1}\mathfrak{U}_n
  \]
  has exactly one pole in $W$; denoting
  \[
    v_n:(\theta,\xi)\mapsto
    =\exp(-k\tfrac{\xi^2}{2})\exp(ij(\theta+i\xi))\qquad u_n=\mathfrak{V}_nv_n,
  \]
  we obtain, by Proposition \ref{prop:local_resolvent}, that
  \[
    \Pi_W\1_{U_n}=\Pi_{\C u_n}+O(e^{-c''k}).
  \]
  All in all, denoting by $\mathcal{N}(W)$ the set of $1\leq n\leq N$
  such that $f_n(z;k^{-1})=\frac{j}{k}+O(e^{-c'k})$ admits a solution
  in $W$, we find
  \[
    \Pi_W=\sum_{n\in \mathcal{N}(W)}\Pi_{\C u_n}+O(e^{-c''k}).
  \]
  In particular,\[{\rm Rank}(\Pi_W)=|\mathcal{N}(W)|+O(e^{-c'k})
  \]
  and the rank must be an integer.
\end{proof}

\begin{rem}\label{rem:Jordan}
  In Proposition \ref{prop:multiplicity}, there may exist Jordan blocks in $W$, but
  their effect is small. Indeed, importing the notation from the end
  of the proof, in an orthogonal basis constructed from $(\Pi_Wu_n)_{n\in
    \mathcal{N}(W)}$ by the Gram-Schmidt process, the matrix of $T_k^{\rm
    cov}(p_z)$ will be exponentially close to a diagonal matrix
  $\lambda I_n$ where $\lambda$ is any element of $W$. Thus, in
  \emph{any} orthogonal basis for ${\rm Ran}(\Pi_W)$, the matrix of
  $T_k^{\rm cov}(p_z)$ will be exponentially close to $\lambda
  I_n$. In particular, this is true of orthogonal Jordan bases, where
  one picks an orthogonal basis for the eigenspaces, then completes it
  into an orthogonal basis for the second level generalised
  eigenvectors, and so on. The obtained matrix has eigenvalues
  exponentially close to $\lambda$ on the diagonal, upper-diagonal
  coefficients which are all exponentially small, and all other
  coefficients are zero.
\end{rem}

\section{An example on the sphere}

To conclude, we illustrate our results by investigating an example on $S^2$; more precisely, we consider the operator 
\begin{equation} T_k(\varepsilon) = T_k^{\text{cov}}(x_3) + i \varepsilon T_k^{\text{cov}}(x_1^2) \label{eq:Tk_S2} \end{equation}
with $(x_1, x_2, x_3)$ the usual Cartesian coordinates of the
embedding $S^2\to \R^3$. Here $\varepsilon$ is a parameter which will be chosen small enough (but independent of $k$).

Let us explain how to obtain $T_k(\varepsilon)$. In fact we start from
$(M,\omega) = (\C\P^1,\omega_{\text{FS}})$, the complex projective
line endowed with the Fubini-Study form (normalised to give a volume
of $2\pi$), which we identify with $S^2$ by means of the stereographic
projection $\pi_N$ from the north pole to the equatorial plane. It is
standard that the hyperplane bundle $L = \mathcal{O}(1)$ is a
prequantum line bundle, and that the quantum space $H^0(M,L^{\otimes
  k})$ identifies with the space of homogeneous polynomials of degree
$k$ in two complex variables. In fact it is more convenient to work in
the chart $U_{\infty} = \{ [z_1:z_2] \in \C\P^1; \ z_2 \neq 0 \}$ with
holomorphic coordinate $z = \frac{z_1}{z_2}$, so that
$H^0(M,L^{\otimes k})$ identifies with the space of polynomials of
degree at most $k$ in one complex variable. In this identification, in
a Hermitian chart for $L$, the Hermitian product reads
\[ \langle P, Q \rangle_k = \int_{\C} \frac{P(z) \overline{Q(z)}}{(1 + |z|^2)^{k+2}} | \dd z \wedge \dd \bar{z} |.  \]
One readily checks that an orthonormal basis is given by
\[ e_{\ell} = \sqrt{\frac{(k+1) \binom{k}{\ell}}{2\pi}} \ z^l, \qquad 0 \leq \ell \leq k,  \]
so that the Bergman kernel reads
\[ \Pi_k(z,w) = \frac{k+1}{2\pi} \frac{(1 + z \bar{w})^k}{(1 + |w|^2)^{k+2}}.  \]
Before computing $T_k(\varepsilon)$, we will need a slightly technical lemma.

\begin{lem}
\label{lem:int_S2}
Let $\alpha, \beta, \gamma, \delta \in \N$ be such that $\alpha + \beta + \gamma < 2(\delta-1)$ and let $z \in \C$. If $\beta \leq \alpha \leq \beta + \gamma$, then
\[ I(\alpha,\beta,\gamma,\delta; z) := \int_{\C} \frac{w^{\alpha} \bar{w}^{\beta} (1+z\bar{w})^{\gamma}}{(1 + |w|^2)^{\delta}} | \dd w \wedge \dd \bar{w} | = 2 \pi \binom{\gamma}{\alpha - \beta}  \frac{\alpha! (\delta - \alpha - 2)!}{(\delta-1)!} z^{\alpha - \beta}. \]
Otherwise $I(\alpha,\beta,\gamma,\delta; z) = 0$.
\end{lem}

\begin{proof}
By expanding the numerator and passing to polar coordinates, we compute 
\[ I(\alpha,\beta,\gamma,\delta; z) = 2 \sum_{p=0}^{\gamma} \binom{\gamma}{p} z^p \left( \int_0^{+\infty} \frac{\rho^{\alpha+\beta+p+1}}{(1 + \rho^2)^{\delta}} \dd \rho \right) \underbrace{\left( \int_0^{2\pi} e^{i(\alpha-\beta-p)} \dd \theta \right)}_{=2\pi \delta_{p,\alpha-\beta}}. \]
So if $\alpha - \beta \notin \{0, \cdots, \gamma\}$, then $I(\alpha,\beta,\gamma,\delta; z) = 0$. Otherwise 
\[ \begin{split} I(\alpha,\beta,\gamma,\delta; z) & = 4 \pi \binom{\gamma}{\alpha - \beta} z^{\alpha - \beta}  \int_0^{+\infty} \frac{\rho^{2\alpha+1}}{(1 + \rho^2)^{\delta}} \dd \rho \\
& = 2 \pi \binom{\gamma}{\alpha - \beta} z^{\alpha - \beta} \int_0^{+\infty} \frac{t^{\alpha}}{(1 + t)^{\delta}} \dd t \\
& = 2 \pi \binom{\gamma}{\alpha - \beta} z^{\alpha - \beta} B(\alpha + 1, \delta - \alpha - 1) \\
& = 2 \pi \binom{\gamma}{\alpha - \beta} z^{\alpha - \beta} \frac{\alpha! (\delta - \alpha - 2)!}{(\delta-1)!}.  \end{split}  \]
Here $B$ is the beta function.
\end{proof}

This allows us to quickly compute $T_k(\varepsilon)$.

\begin{prop}
For every $\ell \in \{0, \cdots, k\}$,
\[ T_k^{\mathrm{cov}}(x_3) e_{\ell} = \frac{2 \ell - k}{k} e_{\ell}. \]
Moreover, for every $\ell \in \{0, \cdots, k\}$ (with a slight abuse of notation for the extreme cases) 
\[ \begin{split} T_k^{\mathrm{cov}}(x_1^2) e_{\ell} & = \frac{1}{k(k-1)} \left( \sqrt{\ell (\ell-1) (k-\ell+2)(k-\ell+1)} e_{\ell-2} + 2 \ell (k-\ell) e_{\ell} \right. \\
& \left. + \sqrt{(\ell+1) (\ell+2) (k-\ell)(k-\ell-1)} e_{\ell+2} \right). \end{split} \]
\end{prop}

\begin{proof}
Let us prove the claim for $T_k^{\text{cov}}(x_1^2)$, since the case of $T_k^{\text{cov}}(x_3)$ follows from a similar (but easier) computation. We compute (for $\ell \in \{2, \cdots, k-2\}$, but the extreme cases are similar), using Lemma \ref{lem:int_S2} and its notation,
\[ \begin{split} T_k^{\text{cov}}(x_1^2) z^{\ell} & = \frac{k+1}{2\pi}  \int_{\C} \frac{(1 + z \bar{w})^k}{(1 + |w|^2)^{k+2}} \frac{(z + \bar{w})^2}{(1 + z \bar{w})^2} w^{\ell} |\dd w \wedge \dd\bar{w}| \\
& = \frac{k+1}{2\pi}  \int_{\C} \frac{(1 + z \bar{w})^{k-2}}{(1 + |w|^2)^{k+2}} (z^2 + 2 z \bar{w} + \bar{w}^2) w^{\ell} |\dd w \wedge \dd\bar{w}| \\
& = \frac{k+1}{2\pi} \left( z^2 I(\ell,0,k-2,k+2; z) + 2 z I(\ell,1,k-2,k+2; z) + I(\ell,2,k-2,k+2; z) \right) \\
& = (k+1) \left(  \binom{k-2}{\ell}  \frac{\ell! (k-\ell)!}{(k+1)!} z^{\ell +2} + \binom{k-2}{\ell -1}  \frac{\ell! (k-\ell)!}{(k+1)!} z^{\ell} + \binom{k-2}{\ell-2} \frac{\ell! (k-\ell)!}{(k+1)!} z^{\ell - 2} \right) \\
& = \frac{1}{\binom{k}{\ell}}  \left( \binom{k-2}{\ell} z^{\ell +2} + \binom{k-2}{\ell -1}  \frac{\ell! (k-\ell)!}{(k+1)!} z^{\ell} + \binom{k-2}{\ell-2} \frac{\ell! (k-\ell)!}{(k+1)!} z^{\ell - 2} \right). \end{split} \]
To conclude, it only remains to carefully keep track of the normalisation constants when passing from $z^{\ell}$, $z^{\ell-2}$, $z^{\ell+2}$ to $e_{\ell}$, $e_{\ell-2}$, $e_{\ell+2}$.
\end{proof}

Using these formulas, we can compute the spectrum of
$T_k(\varepsilon)$ numerically. To compare it with the approximate
eigenvalues given by the Bohr-Sommerfeld conditions in Theorem
\ref{thr:1} and Proposition \ref{prop:BS_mod_k-2}, we also need to compute numerically the complex action. In order to do so, we first come up with a parametrisation for a good cycle $\mathcal{C}_{z,\varepsilon}$ inside $\widetilde{p_{\varepsilon}}^{-1}(\lambda)$ with $p_{\varepsilon} = x_3 + i \varepsilon x_1^2$. Recall that in our complex coordinate $z$ on $U_{\infty}$,
\[  x_1 = \frac{\Re(z)}{1 + |z|^2}, \quad x_3 = \frac{|z|^2 - 1}{1 + |z|^2}  \]
hence 
\[ \widetilde{p_{\varepsilon}}(z,w) = \frac{z \bar{w} - 1}{1 + z \bar{w}} + i \varepsilon \frac{(z + \bar{w})^2}{(1 + z \bar{w})^2} = \frac{z^2 \bar{w}^2 - 1 + i \varepsilon (z + \bar{w})^2}{(1 + z \bar{w})^2}.  \]
Therefore, a straightforward computation shows that $(z, \bar{w})$ belongs to $\widetilde{p_{\varepsilon}}^{-1}(\lambda)$ if and only if 
\begin{equation} \left( (1-\lambda) z^2 + i \varepsilon \right) \bar{w}^2 + 2 (i \varepsilon - \lambda) \bar{w} + i \varepsilon z^2 - 1 - \lambda = 0. \label{eq:level_ex} \end{equation}
For fixed $z$, this equation gives two possibilities for $\bar{w}$, and we need the one which coincides with $\bar{z}$ when $\varepsilon = 0$ and $\lambda \in \R$, call it $\bar{w}_+(z)$. We choose the cycle $\mathcal{C}_{z,\varepsilon}$ as the image of 
\[ \gamma_{z,\varepsilon}: \R \to \CP^1 \times \CP^1, \quad \theta \mapsto \left(\rho_0 e^{i\theta}, w_+(\rho_0 e^{i\theta}) \right), \qquad \rho_0 = \sqrt{\frac{1 + \Re(\lambda)}{1 - \Re(\lambda)}} \]
and write the principal action as
\[ I_{z,\varepsilon}(\lambda) = \int_{\mathcal{C}_{z,\varepsilon}} \tilde{\alpha} \]
with 
\[ \alpha = \frac{i(z \dd \bar{z} - \bar{z} \dd z)}{2(1 + |z|^2)} \quad \text{so} \quad \tilde{\alpha} = \frac{i(z \dd \bar{w} - \bar{w} \dd z)}{2(1 + z\bar{w})}.  \]
Our different choices ensure that when $\lambda \in \R$ and $\varepsilon = 0$, this recovers the usual action. By differentiating Equation \eqref{eq:level_ex}, one can compute the restriction of $\tilde{\alpha}$ to $\mathcal{C}_{z,\varepsilon}$, so $I_{z,\varepsilon}$ can easily be computed numerically using some integration routine. Then we solve numerically the implicit equation
\begin{equation} I_{z,\varepsilon}(\lambda) \in 2 \pi k^{-1} \Z \label{eq:BS_princ_S2} \end{equation}
corresponding to Equation \eqref{eq:BS_mod_k-2} where we only keep the principal term, to obtain the approximate spectrum of $T_k^{\text{cov}}(\varepsilon)$. This is illustrated in Figure \ref{fig:spectre_zix2_eps02}.

\begin{figure}
\begin{center}
  \includegraphics[width=0.9\linewidth]{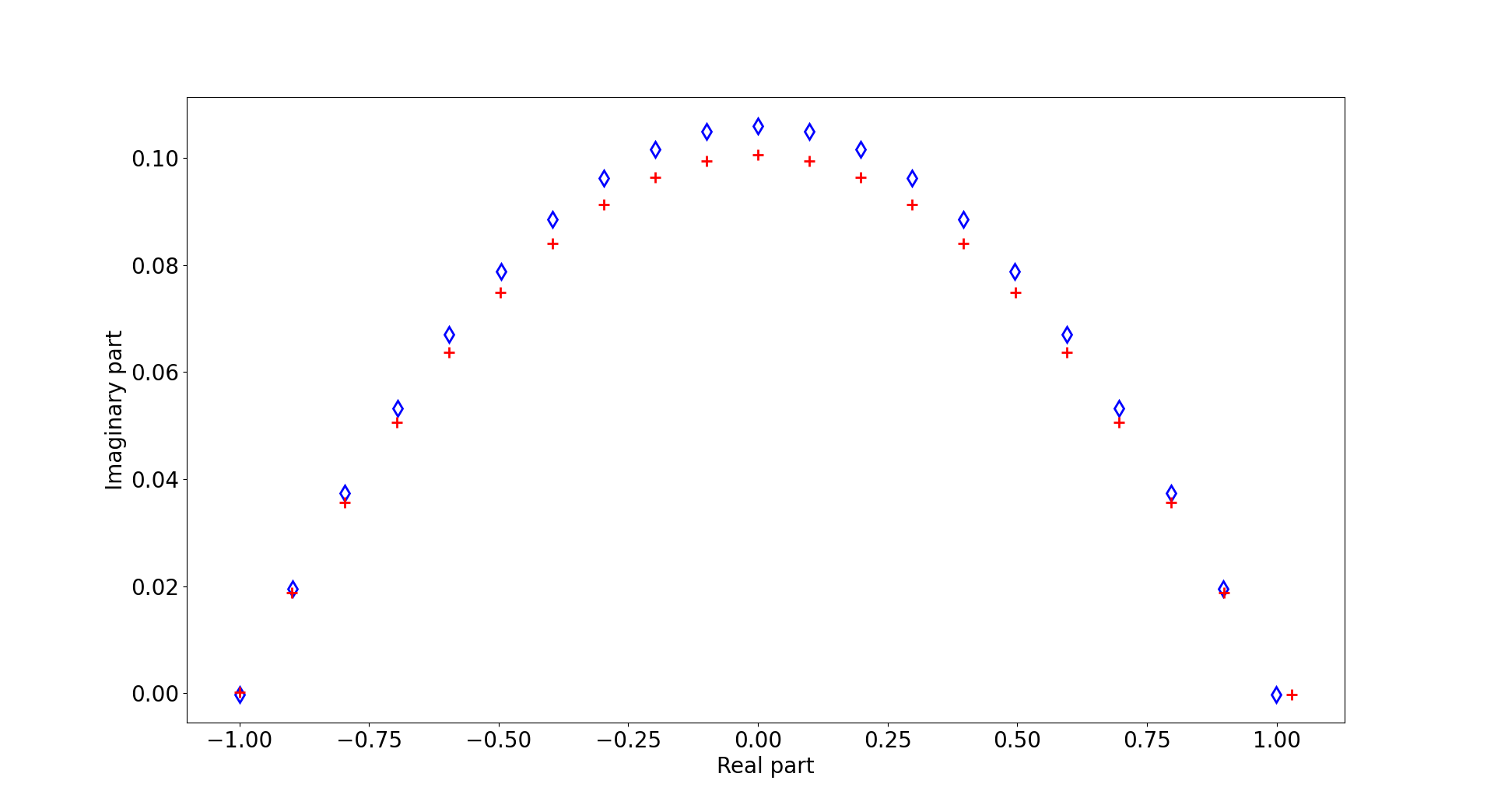}
  \includegraphics[width=0.9\linewidth]{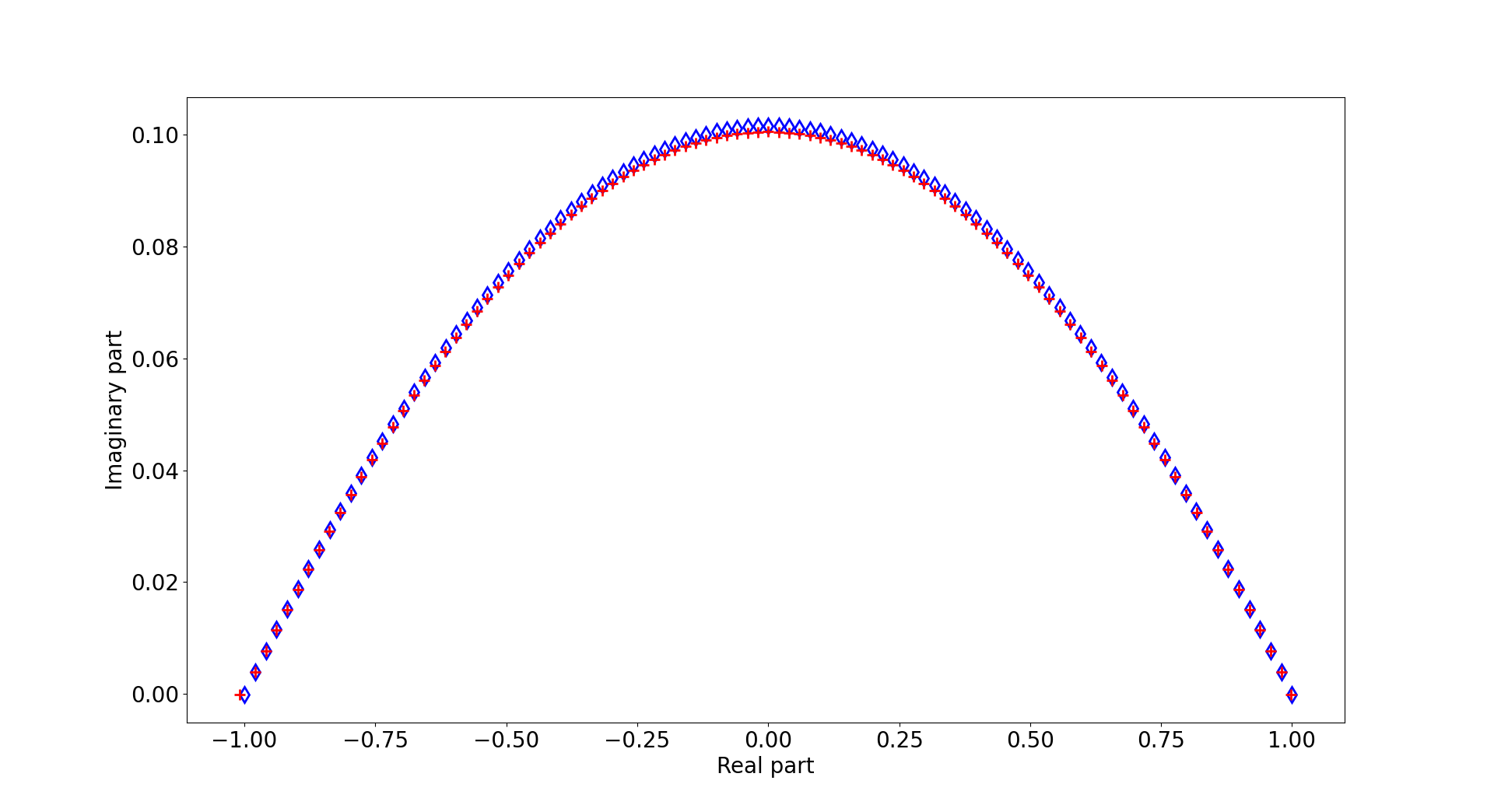}
\end{center}
\caption{Zeroth order approximation: the spectrum of the operator
  $T_k^{\text{cov}}(\varepsilon)$ from Equation \eqref{eq:Tk_S2} (blue
  diamonds) and the approximate eigenvalues given by the solutions of
  Equation \eqref{eq:BS_princ_S2} (red crosses)  for $\varepsilon =
  0.2$ at $k=20$ (above) and $k=100$ (below).}
\label{fig:spectre_zix2_eps02}
\end{figure}

We can also take into account the subprincipal term in Equation \eqref{eq:BS_mod_k-2}. In fact, using the precise subprincipal term given in Proposition \ref{prop:BS_mod_k-2} would be too cumbersome, but thankfully this can be circumvented as follows. As explained in Remark \ref{rem:subprincipal_Lagrangian}, when $\varepsilon = 0$ (or more generally $z=0$ in the above notation) and when we consider a Berezin-Toeplitz operator $T_k$ acting $H^0(M,L^{\otimes k} \otimes \delta)$ with $\delta$ a half-form bundle, we recover the usual Bohr-Sommerfeld conditions stated in \cite{charles_symbolic_2006}. In that setting, when the so-called normalised subprincipal symbol of $T_k$ vanishes, the subprincipal term in the Bohr-Sommerfeld conditions simply equals $\epsilon \pi$ where $\epsilon \in \{0,1\}$ is an index associated with the connected component of $p_0^{-1}(\lambda)$ that we are interested in and coming from $\delta$. This will not change with small changes in the parameter; so in the rest of this section, we replace $T_k(\varepsilon)$ with $S_k(\varepsilon)$ acting on $H^0(M,L^{\otimes k} \otimes \delta)$ with vanishing normalised subprincipal symbol. Coming back to our precise example, the tautological bundle $\delta = \mathcal{O}(-1)$ is a half-form bundle, so acting on $H^0(M,L^{\otimes k} \otimes \delta)$ only consists in shifting $k$ by $1$. Moreover, starting from $T_k(\varepsilon)$, we can obtain such a $S_k(\varepsilon)$ as follows. Recall from \cite{charles_symbolic_2006} that the normalised and covariant subprincipal symbols of a Berezin-Toeplitz operator $T_k$ are related by
\[ \sigma_1^{\text{norm}}(T_k) =  \sigma_1^{\text{cov}}(T_k) - \frac{1}{2} \Delta \sigma_0^{\text{cov}}(T_k). \]
Taking all these remarks into account, the operator
\[ S_k(\varepsilon) = T_{k-1}^{\text{cov}}\left(x_3 + i \varepsilon x_1^2\right) + \frac{1}{2k} T_{k-1}^{\text{cov}}\left(\Delta(x_3 + i \varepsilon x_1^2)\right) \] 
acts on $H^0(M,L^{\otimes k} \otimes \delta)$ with principal symbol $x_3 + i \varepsilon x_1^2$  and vanishing normalised subprincipal symbol. Using that 
\[ \Delta x_3 = -2 x_3, \qquad \Delta x_1^2 = 2 - 6 x_1^2, \]
we finally obtain that 
\begin{equation} S_k(\varepsilon) = \left(1 - \frac{1}{k}\right) T_{k-1}^{\text{cov}}(x_3) + i \varepsilon \left( 1 - \frac{3}{k} \right) T_{k-1}^{\text{cov}}(x_1^2) + \frac{i \varepsilon}{k}. \label{eq:Sk_S2} \end{equation}
In our situation, $\epsilon = 1$ so the approximate eigenvalues of $S_k(\varepsilon)$ are the solutions of the implicit equation
\begin{equation} I_{z,\varepsilon}(\lambda) + k^{-1} \pi  \in 2 \pi k^{-1} \Z. \label{eq:BS_half_S2} \end{equation}
The comparison between these solutions and the actual spectrum of $S_k^{\varepsilon}$ is performed in Figure \ref{fig:spectre_zix2_eps02_half}. In Figure \ref{fig:spectre_zix2_eps02_zoom}, we zoom on a region containing a few eigenvalues to illustrate the difference in the precision of the approximation with or without the subprincipal correction. Note that in this example, the Bohr-Sommerfeld rules accurately describe the whole spectrum; this is natural since the only singularities are encountered at the minimum and maximum of $p_0 = x_3$.

\begin{figure}
\begin{center}
  \includegraphics[width=0.9\linewidth]{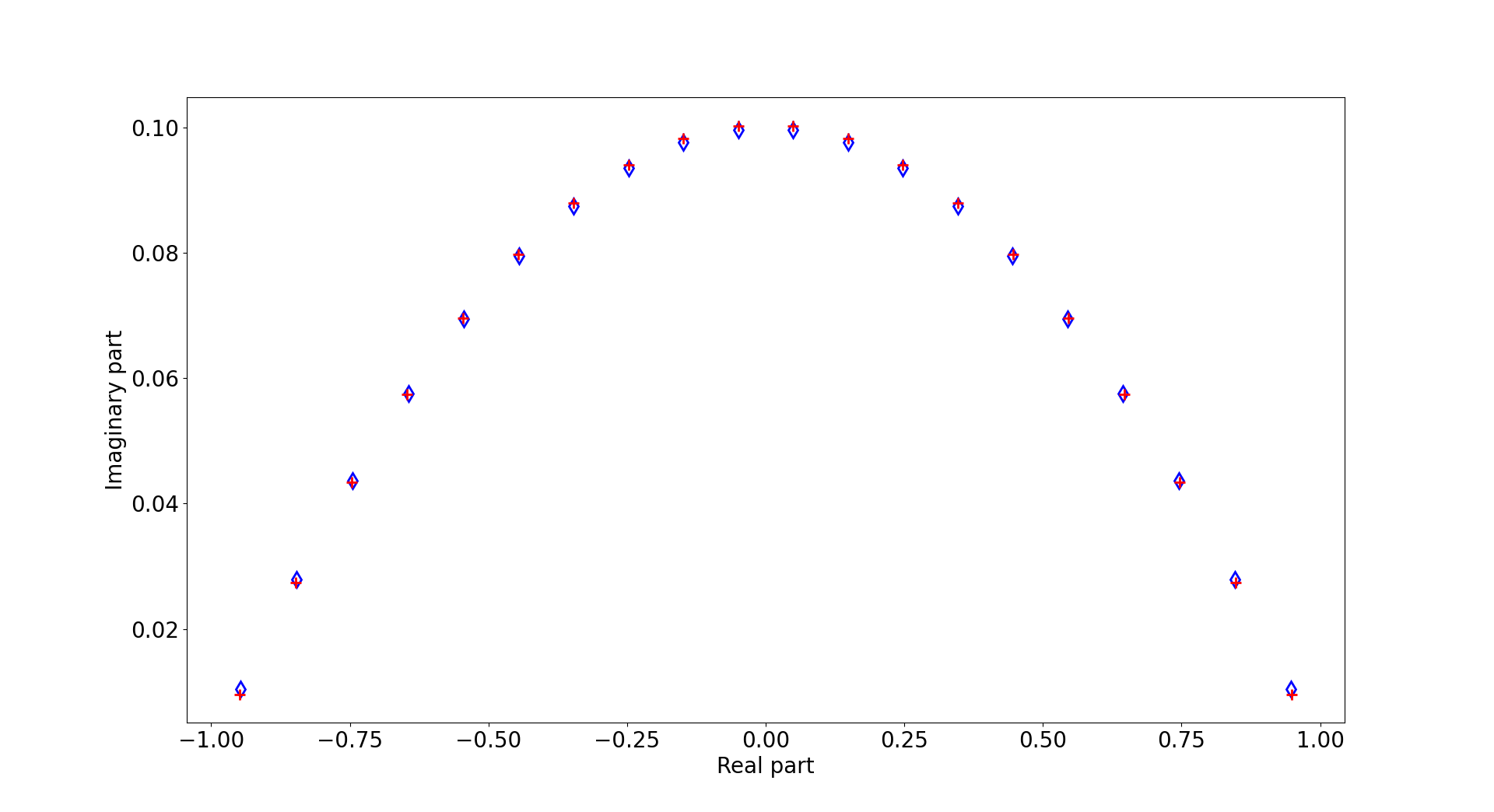}
  
\includegraphics[width=0.9\linewidth]{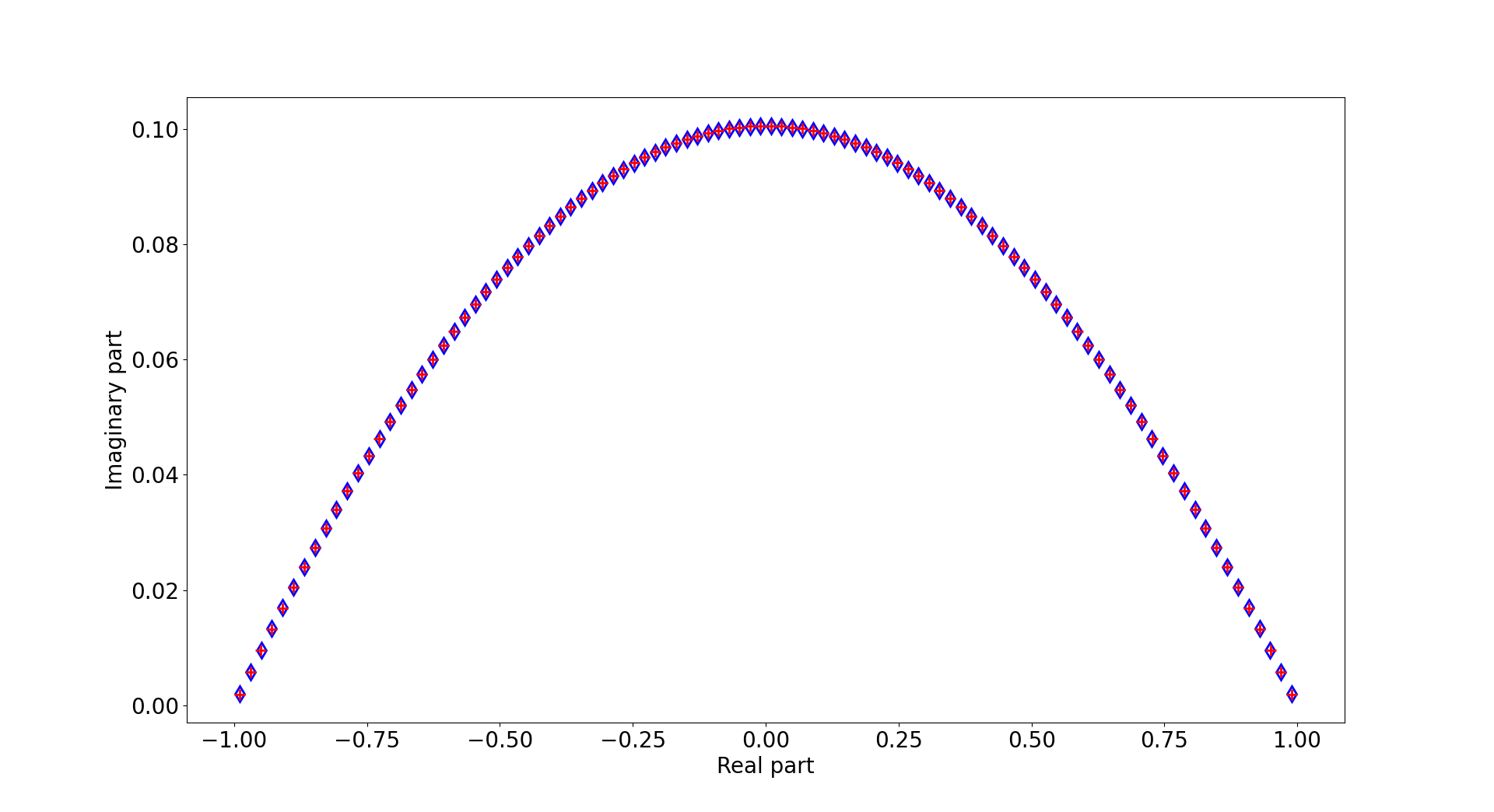}
\end{center}
\caption{Comparison between the spectrum of the operator
  $S_k(\varepsilon)$ from Equation \eqref{eq:Sk_S2} (blue diamonds)
  and the approximate eigenvalues given by the solutions of Equation
  \eqref{eq:BS_half_S2} (red crosses) for $\varepsilon = 0.2$ at $k=20$
(above) and $k=100$ (below).}
\label{fig:spectre_zix2_eps02_half}
\end{figure}

\begin{figure}
\begin{center}
  \includegraphics[width=0.9\linewidth]{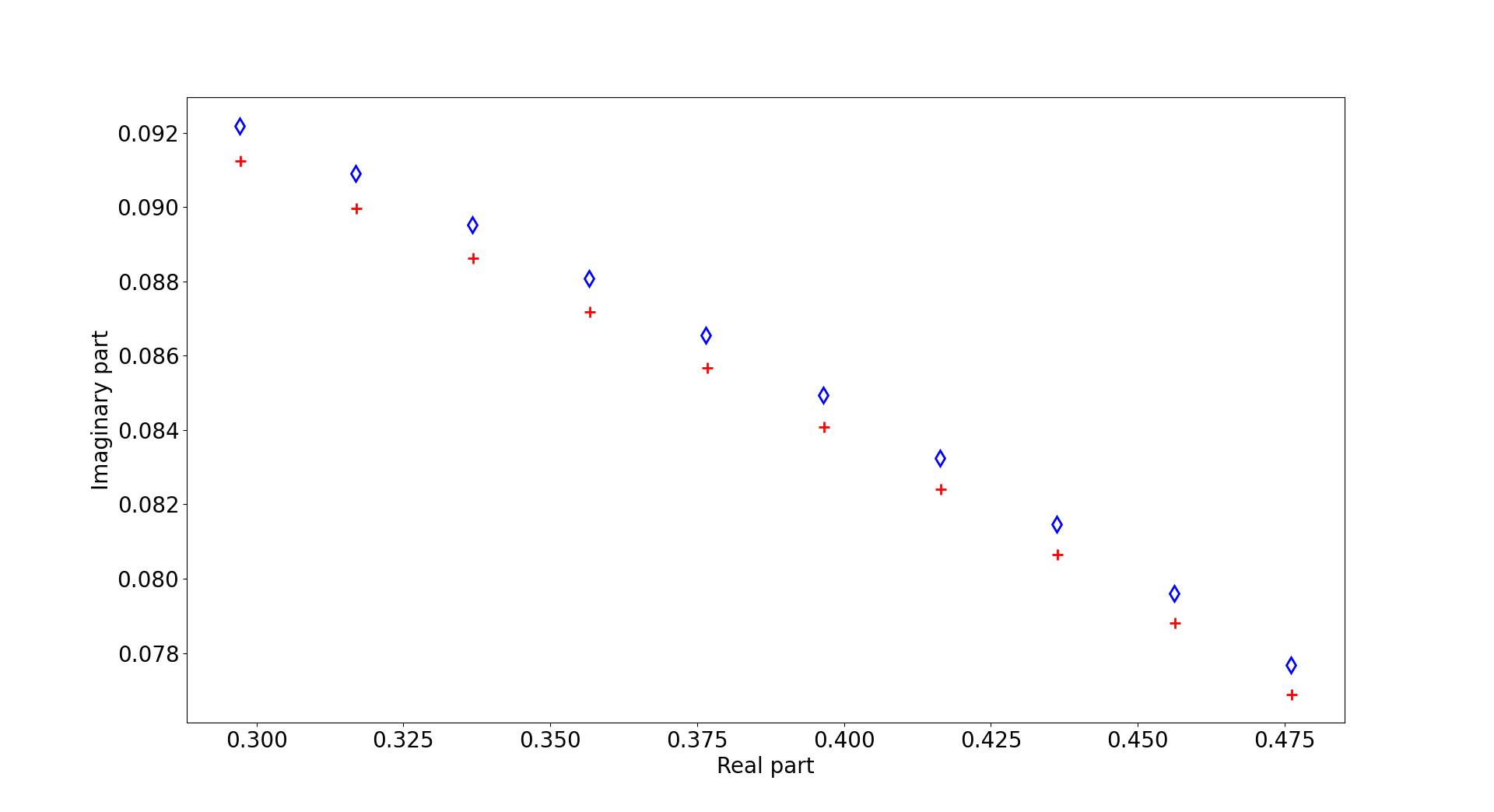}
  
\includegraphics[width=0.9\linewidth]{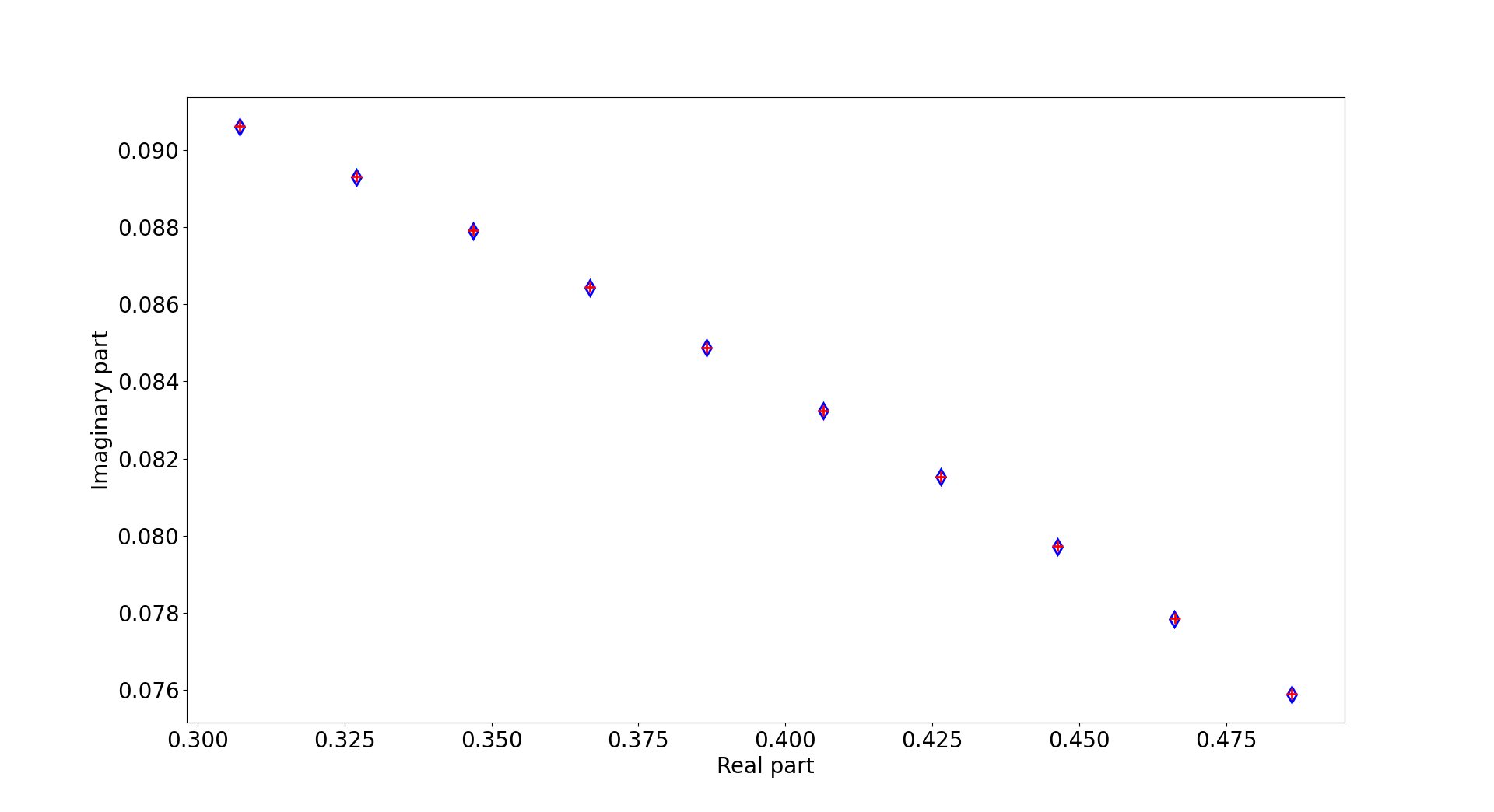}
\end{center}
\caption{Zoom on a few eigenvalues in the $k=100$ plots displayed in Figures \ref{fig:spectre_zix2_eps02} (top) and \ref{fig:spectre_zix2_eps02_half} (bottom). Recall that for exposition reasons, the top figure displays eigenvalues of $T_k(\varepsilon)$ while the bottom one displays eigenvalues of $S_k(\varepsilon)$, but the important information here is the difference in the precision of the approximation thanks to the subprincipal correction.}
\label{fig:spectre_zix2_eps02_zoom}
\end{figure}

\bibliographystyle{abbrv}
\bibliography{main}
\end{document}

%% file: thm-config.tex
\usepackage{amsthm}
\usepackage{amssymb}
\usepackage{xparse}
\usepackage{thmtools}
\usepackage{stackrel}
\usepackage{bbold}
\usepackage{relsize}
\usepackage{tikz}

\usetikzlibrary{hobby}

\newcommand{\R}{\mathbb{R}}
\newcommand{\C}{\mathbb{C}}
\newcommand{\Z}{\mathbb{Z}}
\newcommand{\N}{\mathbb{N}}

\newcommand{\dd}{\mathrm{d}}

\renewcommand{\P}{\mathbb{P}}

\newcommand{\CP}{\mathbb{CP}}

\DeclareMathOperator{\dist}{dist}

\makeatletter
\newtheoremstyle{indented}
{7pt} 
{7pt} 
{} 
{1.5em} 
{\bfseries} 
{.} 
{.5em} 
{} 

\theoremstyle{definition}

\newtheorem{defn}{Definition}[section]

\theoremstyle{plain}
\newtheorem*{theorem*}{Theorem}

\newtheorem{theorem}{Theorem}

\newtheorem{prop}[defn]{Proposition}
\newtheorem{corr}[defn]{Corollary}

\newtheorem{lem}[defn]{Lemma}

\theoremstyle{definition}
\newtheorem{rem}[defn]{Remark} 

\newtheorem{hyp}[defn]{Hypothesis}

\makeatletter
\renewcommand*\env@matrix[1][*\c@MaxMatrixCols c]{%
  \hskip -\arraycolsep
  \let\@ifnextchar\new@ifnextchar
  \array{#1}}
\makeatother


%% file: Article.bbl
\begin{thebibliography}{10}

\bibitem{abasheva_feix-kaledin_2022}
A.~Abasheva.
\newblock Feix-{{Kaledin}} metric on the total spaces of cotangent bundles to
  {{K{\"a}hler}} quotients.
\newblock {\em International Mathematics Research Notices. IMRN},
  (17):13054--13098, 2022.

\bibitem{alphonse_polar_2023}
P.~Alphonse and J.~Bernier.
\newblock Polar decomposition of semigroups generated by non-selfadjoint
  quadratic differential operators and regularizing effects.
\newblock In {\em Annales {{Scientifiques}} de l'{{{\'E}cole Normale
  Sup{\'e}rieure}}}, volume~56, pages 323--382, 2023.

\bibitem{beauville_varietes_1983}
A.~Beauville.
\newblock Vari{\'e}t{\'e}s {{K{\"a}hleriennes}} dont la premi{\`e}re classe de
  {{Chern}} est nulle.
\newblock {\em Journal of Differential Geometry}, 18(4):755--782, Jan. 1983.

\bibitem{berezin_general_1975-1}
F.~A. Berezin.
\newblock General concept of quantization.
\newblock {\em Communications in Mathematical Physics}, 40(2):153--174, June
  1975.

\bibitem{bonthonneau_microlocal_2024}
Y.~G. Bonthonneau.
\newblock Microlocal {{Projectors}}.
\newblock (arXiv:2407.06644), July 2024.

\bibitem{bordemann_toeplitz_1994}
M.~Bordemann, E.~Meinrenken, and M.~Schlichenmaier.
\newblock Toeplitz quantization of {{K{\"a}hler}} manifolds and gl ({{N}}),
  {{N}}{$\rightarrow\infty$} limits.
\newblock {\em Communications in Mathematical Physics}, 165(2):281--296, 1994.

\bibitem{borthwick_semiclassical_1998}
D.~Borthwick, T.~Paul, and A.~Uribe.
\newblock Semiclassical spectral estimates for {{Toeplitz}} operators.
\newblock In {\em Annales de l'institut {{Fourier}}}, volume~48, pages
  1189--1229, 1998.

\bibitem{borthwick_pseudospectra_2003}
D.~Borthwick and A.~Uribe.
\newblock On the {{Pseudospectra}} of {{Berezin-Toeplitz Operators}}.
\newblock {\em Methods and Applications of Analysis}, 10(1):031--066, Mar.
  2003.

\bibitem{boussekkine_ptsymmetry_2016}
N.~Boussekkine and N.~Mecherout.
\newblock {{PT}}-symmetry and potentiel well. {{The}} simple well case.
\newblock {\em Mathematische Nachrichten}, 289(1):13--27, Jan. 2016.

\bibitem{boutet_de_monvel_spectral_1981}
L.~{Boutet de Monvel} and V.~Guillemin.
\newblock {\em The Spectral Theory of {{Toeplitz}} Operators}.
\newblock Number~99 in Annals of {{Mathematics Studies}}. Princeton University
  Press, 1981.

\bibitem{boutet_de_monvel_pseudo-differential_1967}
L.~{Boutet de Monvel} and P.~Kr{\'e}e.
\newblock Pseudo-differential operators and {{Gevrey}} classes.
\newblock {\em Annales de l'Institut Fourier}, 17(1):295--323, 1967.

\bibitem{calabi_metriques_1979}
E.~Calabi.
\newblock M{\'e}triques k{\"a}hl{\'e}riennes et fibr{\'e}s holomorphes.
\newblock {\em Annales scientifiques de l'{\'E}cole Normale Sup{\'e}rieure},
  12(2):269--294, 1979.

\bibitem{charles_berezin-toeplitz_2003}
L.~Charles.
\newblock Berezin-{{Toeplitz Operators}}, a {{Semi-Classical Approach}}.
\newblock {\em Communications in Mathematical Physics}, 239(1-2):1--28, Aug.
  2003.

\bibitem{charles_quasimodes_2003}
L.~Charles.
\newblock Quasimodes and {{Bohr-Sommerfeld Conditions}} for the {{Toeplitz
  Operators}}.
\newblock {\em Communications in Partial Differential Equations},
  28(9-10):1527--1566, Jan. 2003.

\bibitem{charles_symbolic_2006}
L.~Charles.
\newblock Symbolic calculus for {{Toeplitz}} operators with half-form.
\newblock {\em Journal of Symplectic Geometry}, 4(2):171--198, 2006.

\bibitem{charles_analytic_2021}
L.~Charles.
\newblock Analytic {{Berezin}}--{{Toeplitz}} operators.
\newblock {\em Mathematische Zeitschrift}, 299(1):1015--1035, 2021.

\bibitem{charles_quantum_2020}
L.~Charles and Y.~Le~Floch.
\newblock Quantum propagation for {{Berezin-Toeplitz}} operators.
\newblock (2009.05279), 2020.

\bibitem{colin_de_verdiere_spectre_1980}
Y.~{Colin de Verdi{\`e}re}.
\newblock Spectre conjoint d'op{\'e}rateurs pseudo-diff{\'e}rentiels qui
  commutent {{II}}. {{Le}} cas int{\'e}grable.
\newblock {\em Mathematische Zeitschrift}, 171(1):51--73, 1980.

\bibitem{deleporte_toeplitz_2018}
A.~Deleporte.
\newblock Toeplitz operators with analytic symbols.
\newblock {\em Journal of Geometric Analysis}, 31:3915--3967, 2021.

\bibitem{deleporte_direct_2024}
A.~Deleporte, M.~Hitrik, and J.~Sj{\"o}strand.
\newblock A direct approach to the analytic {{Bergman}} projection.
\newblock {\em Annales de la Facult{\'e} des Sciences de Toulouse :
  Math{\'e}matiques}, 33(1):153--176, 2024.

\bibitem{deleporte_real-analytic_2022}
A.~Deleporte and S.~Zelditch.
\newblock Real-analytic geodesics in the {{Mabuchi}} space of {{K{\"a}hler}}
  metrics and quantization.
\newblock (2210.00763), Oct. 2022.

\bibitem{dencker_pseudospectra_2004}
N.~Dencker, J.~Sj{\"o}strand, and M.~Zworski.
\newblock Pseudospectra of semiclassical (pseudo-)differential operators.
\newblock {\em Communications on Pure and Applied Mathematics}, 57(3):384--415,
  2004.

\bibitem{duraffour_these}
A.~Duraffour.
\newblock {\em Effet tunnel quantique microlocal et estimées analytiques en
  une dimension}.
\newblock PhD thesis, 2024.

\bibitem{duraffour_analytic_2025}
A.~Duraffour.
\newblock Analytic microlocal {{Bohr-Sommerfeld}} expansions.
\newblock (arXiv:2501.06046), Jan. 2025.

\bibitem{feix_hyperkahler_2001}
B.~Feix.
\newblock Hyperk{\"a}hler metrics on cotangent bundles.
\newblock {\em Journal f{\"u}r die Reine und Angewandte Mathematik. [Crelle's
  Journal]}, 532:33--46, 2001.

\bibitem{folland_harmonic_1989}
G.~Folland.
\newblock {\em Harmonic {{Analysis}} in Phase Space}.
\newblock Number 122 in Annals of {{Math}}. {{Studies}}. Princeton University
  Press, 1989.

\bibitem{guillemin_star_1995}
V.~Guillemin.
\newblock Star products on compact pre-quantizable symplectic manifolds.
\newblock {\em Letters in Mathematical Physics}, 35(1):85--89, Sept. 1995.

\bibitem{guillemin_grauert_1991}
V.~Guillemin and M.~Stenzel.
\newblock Grauert tubes and the homogeneous {{Monge-Ampere}} equation.
\newblock {\em Journal of Differential Geometry}, 34(2):561--570, 1991.

\bibitem{guillemin_grauert_1992}
V.~Guillemin and M.~Stenzel.
\newblock Grauert tubes and the homogeneous {{Monge-Ampere}} equation. {{II}}.
\newblock {\em Journal of Differential Geometry}, 35(3):627--641, 1992.

\bibitem{guillemin_moment_2005}
V.~Guillemin and S.~Sternberg.
\newblock The moment map revisited.
\newblock {\em Journal of Differential Geometry}, 69(1):137--162, 2005.

\bibitem{hitrik_boundary_2004}
M.~Hitrik.
\newblock Boundary spectral behavior for semiclassical operators in dimension
  one.
\newblock {\em International Mathematics Research Notices}, (64):3417--3438,
  2004.

\bibitem{hitrik_non-selfadjoint_2004}
M.~Hitrik and J.~Sj{\"o}strand.
\newblock Non-selfadjoint perturbations of selfadjoint operators in 2
  dimensions {{I}}.
\newblock {\em Annales Henri Poincar{\'e}}, 5(1):1--73, Feb. 2004.

\bibitem{hitrik_overdamped_2025}
M.~Hitrik and M.~Zworski.
\newblock Overdamped {{QNM}} for {{Schwarzschild}} black holes.
\newblock (arXiv:2406.15924), Jan. 2025.

\bibitem{hormander_fourier_1971}
L.~H{\"o}rmander.
\newblock Fourier integral operators. {{I}}.
\newblock {\em Acta mathematica}, 127(1):79--183, 1971.

\bibitem{hormander_analysis_2003}
L.~H{\"o}rmander.
\newblock {\em The Analysis of Linear Partial Differential Operators. {{I}}.}
\newblock Springer, Berlin, 2003.

\bibitem{ioos_geometric_2021}
L.~Ioos.
\newblock Geometric quantization of symplectic maps and {{Witten}}'s asymptotic
  conjecture.
\newblock {\em Advances in Mathematics}, 387:107840, 2021.

\bibitem{kaledin_canonical_2001}
D.~Kaledin.
\newblock A canonical {{hyperK{\"a}hler}} metric on the total space of a
  cotangent bundle.
\newblock pages 195--230. World Sci. Publ., River Edge, NJ, 2001.

\bibitem{le_floch_singular_2014}
Y.~Le~Floch.
\newblock Singular {{Bohr}}--{{Sommerfeld}} conditions for {{1D Toeplitz}}
  operators: Elliptic case.
\newblock {\em Communications in Partial Differential Equations},
  39(2):213--243, 2014.

\bibitem{le_floch_singular_2014-1}
Y.~Le~Floch.
\newblock Singular {{Bohr}}--{{Sommerfeld}} conditions for {{1D Toeplitz}}
  operators: Hyperbolic case.
\newblock {\em Analysis \& PDE}, 7(7):1595--1637, 2014.

\bibitem{ma_toeplitz_2008-1}
X.~Ma and G.~Marinescu.
\newblock Toeplitz {{Operators}} on {{Symplectic Manifolds}}.
\newblock {\em Journal of Geometric Analysis}, 18(2):565--611, Apr. 2008.

\bibitem{mecherout_pt-symmetry_2016}
N.~Mecherout, N.~Boussekkine, T.~Ramond, and J.~Sj{\"o}strand.
\newblock {{PT-symmetry}} and {{Schr{\"o}dinger}} operators. {{The}} double
  well case.
\newblock {\em Mathematische Nachrichten}, 289(7):854--887, 2016.

\bibitem{melin_determinants_2002}
A.~Melin and J.~Sj{\"o}strand.
\newblock Determinants of pseudodifferential operators and complex deformations
  of phase space.
\newblock {\em Methods and Applications of Analysis}, 9(2):177--237, 2002.

\bibitem{melin_bohr-sommerfeld_2003}
A.~Melin and J.~Sj{\"o}strand.
\newblock Bohr-{{Sommerfeld}} quantization condition for non-selfadjoint
  operators in dimension 2.
\newblock In {\em Ast{\'e}risque}, number 284, pages 181--244. 2003.

\bibitem{reguer_these}
N.~Reguer.
\newblock {\em Valeurs propres d'op{é}rateurs de Toeplitz analytiques
  non-auto-adjoints pr{è}s d'un point critique elliptique}.
\newblock PhD thesis, 2026+.

\bibitem{rouby_bohrsommerfeld_2017}
O.~Rouby.
\newblock Bohr--{{Sommerfeld Quantization Conditions}} for {{Nonselfadjoint
  Perturbations}} of {{Selfadjoint Operators}} in {{Dimension One}}.
\newblock {\em Int. Math. Research Notices}, rnw309, 2017.

\bibitem{rouby_analytic_2018}
O.~Rouby, J.~Sj{\"o}strand, and S.~V{\~u}~Ng{\d o}c.
\newblock Analytic {{Bergman}} operators in the semiclassical limit.
\newblock {\em Duke Mathematical Journal}, 169(16):3033--3097, 2020.

\bibitem{sjostrand_singularites_1982}
J.~Sj{\"o}strand.
\newblock {\em Singularit{\'e}s Analytiques Microlocales}, volume~95 of {\em
  Ast{\'e}risque}.
\newblock Soc. Math. de France, 1982.

\bibitem{vu_ngoc_systemes_2006}
S.~V{\~u}~Ng{\d o}c.
\newblock {\em Syst{\`e}mes Int{\'e}grables Semi-Classiques: {{Du}} Local Au
  Global}, volume~22 of {\em Panoramas et Synth{\`e}ses}.
\newblock Soci{\'e}t{\'e} math{\'e}matique de France, 2006.

\bibitem{zelditch_pointwise_2018}
S.~Zelditch and P.~Zhou.
\newblock Pointwise {{Weyl Law}} for {{Partial Bergman Kernels}}.
\newblock In M.~Hitrik, D.~Tamarkin, B.~Tsygan, and S.~Zelditch, editors, {\em
  Algebraic and {{Analytic Microlocal Analysis}}}, Springer {{Proceedings}} in
  {{Mathematics}} \& {{Statistics}}, pages 589--634, Cham, 2018. Springer
  International Publishing.

\bibitem{zworski_remark_2001}
M.~Zworski.
\newblock A remark on a paper of {{E}}. {{B}}. {{Davies}}.
\newblock {\em Proceedings of the American Mathematical Society},
  129(10):2955--2957, 2001.

\bibitem{zworski_semiclassical_2012}
M.~Zworski.
\newblock {\em Semiclassical Analysis}, volume 138.
\newblock American Mathematical Soc., 2012.

\end{thebibliography}
